\documentclass[a4paper, 11pt, twoside]{article}
\usepackage{hyperref}

\usepackage{amsmath,amssymb}
\usepackage{amsfonts, latexsym}
\usepackage{amsthm}
\usepackage{mathrsfs}
\usepackage[pdftex]{graphicx}

\usepackage[normalem]{ulem}
\usepackage{url}
\usepackage{comment}
\usepackage{bm}
\usepackage{accents}

\usepackage{enumerate}
\usepackage[all]{xy}
\usepackage{mathtools}
\usepackage{graphicx}
\usepackage{overpic}
\usepackage{pdfpages}

\usepackage[margin=1in]{geometry}

\usepackage{fancyhdr}
\newcommand\shorttitle{Legendrian non-isotopic unit conormal bundles in high dimensions}
\newcommand\authors{Yukihiro Okamoto}

\usepackage{tikz}
\usepackage{tikz-cd}
\usepackage{pgfplots}

\usetikzlibrary{arrows}
\usetikzlibrary{patterns}

\theoremstyle{definition}
\newtheorem{defi}{Definition}[section]
\newtheorem{rem}[defi]{Remark}
\newtheorem{ex}[defi]{Example}

\theoremstyle{plain}
\newtheorem{thm}[defi]{Theorem}

\newtheorem{prop}[defi]{Proposition}
\newtheorem{lem}[defi]{Lemma}

\newcommand{\Hom}{\operatorname{Hom}}

\renewcommand{\ker}{\operatorname{Ker}}
\newcommand{\coker}{\operatorname{Coker}}
\renewcommand{\Re}{\operatorname{Re}}
\renewcommand{\Im}{\operatorname{Im}}

\newcommand{\C}{\mathbb{C}}
\newcommand{\R}{\mathbb{R}}

\newcommand{\Z}{\mathbb{Z}}
\newcommand{\N}{\mathbb{N}}
\renewcommand{\H}{\mathbb{H}}

\newcommand{\ind}{\operatorname{ind}}
\newcommand{\pr}{\operatorname{pr}}

\newcommand{\id}{\operatorname{id}}

\newcommand{\Dbar}{\bar{\partial}}

\renewcommand{\tilde}{\widetilde}

\newcommand{\sing}{\mathrm{sing}}
\newcommand{\ev}{\operatorname{ev}}
\newcommand{\len}{\mathrm{Len}}
\renewcommand{\epsilon}{\varepsilon}

\newcommand{\la}{\langle}
\newcommand{\ra}{\rangle}

\newcommand{\str}{\mathrm{string}}

\newcommand{\codim}{\operatorname{codim}}

\newcommand{\tb}{\operatorname{tb}}
\newcommand{\Span}{\operatorname{Span}}
\newcommand{\vdim}{\operatorname{f.dim}}

\newcommand{\rest}[2]{\left.#1\right|_{#2}}
\newcommand{\ftimes}[2]{\ {}_{#1}\!\times_{#2}}

\newcommand{\M}{\mathcal{M}}
\newcommand{\W}{\mathcal{W}}
\renewcommand{\N}{\mathcal{N}}
\newcommand{\T}{\mathcal{T}}
\newcommand{\LCH}{\operatorname{LCH}}
\renewcommand{\mathring}[1]{\accentset{\circ}{#1}}

\renewcommand{\sp}{\operatorname{sp}}

\usepackage{authblk}

\AtEndDocument{%
  \par
  \medskip
  \begin{tabular}{@{}l@{}}%
    Department of Mathematical Sciences,
Tokyo Metropolitan University, \\
Minami-osawa, Hachioji, Tokyo, 192-0397, Japan
 \\
    \textit{E-mail address}: \texttt{yukihiro@tmu.ac.jp}
  \end{tabular}}

\begin{document}

\title{Legendrian non-isotopic unit conormal bundles in high dimensions}
\author{Yukihiro Okamoto }
\date{}
\maketitle

\begin{abstract}
\noindent
For any compact connected submanifold $K$ of $\R^n$, let $\Lambda_K$ denote its unit conormal bundle, which is a Legendrian submanifold of the unit cotangent bundle of $\mathbb{R}^n$.
In this paper, we give examples of pairs $(K_0,K_1)$ of compact connected submanifolds of $\mathbb{R}^n$ such that $\Lambda_{K_0}$ is not Legendrian isotopic to $\Lambda_{K_1}$, although they cannot be distinguished by classical invariants. Here, $K_1$ is the image of an embedding $\iota_f \colon K_0 \to \mathbb{R}^n$ which is regular homotopic to the inclusion map of $K_0$ and the codimension in $\mathbb{R}^n$ is greater than or equal to $4$.
As non-classical invariants, we define the strip Legendrian contact homology and a coproduct on it under certain conditions on Legendrian submanifolds.
Then, we give a purely topological description of these invariants for $\Lambda_K$ when the codimension of $K$ is greater than or equal to $4$.
The main examples $\Lambda_{K_0}$ and $\Lambda_{K_1}$ are distinguished by the coproduct, which is computed by using an idea of string topology.
\end{abstract}

\section{Introduction}

\noindent
\textbf{Backgrounds.}
Let $Q$ be an $n$-dimensional $C^{\infty}$ manifold with a fixed Riemannian metric. Its unit cotangent bundle $U^*Q$ has a canonical contact form  defined as the pullback of the canonical Liouville $1$-form on $T^*Q$.
In this paper, we consider compact Legendrian submanifolds of this contact manifold $U^*Q$.

Two Legendrian submanifolds $\Lambda_0$ and $\Lambda_1$ of $U^*Q$ are called \textit{Legendrian isotopic} if there exists a $C^{\infty}$ isotopy $(\Lambda_t)_{t\in [0,1]}$ of $C^{\infty}$ submanifolds from $\Lambda_0$ to $\Lambda_1$ such that $\Lambda_t$ is a Legendrian submanifold for every $t\in [0,1]$.
There are three invariants under Legendrian isotopies:
(i) the Thurston-Bennequin number,
(ii) the Legendrian regular homotopy class and
(iii) the $C^{\infty}$ isotopy class as an $(n-1)$-dimensional submanifold.
For the definitions of (i) and (ii), see Subsection \ref{subsubsec-Maslov-TB}. 
In this section, let us call them the \textit{classical invariants} of Legendrian submanifolds.

Let $K$ be a $C^{\infty}$ submanifold of $\R^n$. 
We define its \textit{unit conormal bundle} by
\[\Lambda_K\coloneqq \{(q,p)\in U^*Q \mid q\in K,\ p(v)=0\text{ for every }v\in T_qK\}.\]
Then, $\Lambda_K$ is a Legendrian submanifold of $U^*Q$.
Obviously, for any two submanifolds $K_0$ and $K_1$ of $Q$,
\[
  K_0 \text{ is }C^{\infty} \text{ isotopic to } K_1 \text{ in }Q 
 \Rightarrow  \Lambda_{K_0} \text{ is Legendrian isotopic to } \Lambda_{K_1} \text{ in }U^*Q.
\]
Then, it is natural to ask the following (see, for example, \cite[Section 1]{S}): 
\textit{If} $\Lambda_{K_0}$ \textit{is Legendrian isotopic to} $\Lambda_{K_1}$ \textit{in} $U^*Q$, \textit{ is }$K_0$ $C^{\infty}$\textit{ isotopic to }$K_1$ \textit{in} $Q$\textit{?}

There are several preceding researches on this subject.  Shende \cite{S} and Ekholm-Ng-Shende \cite{ENS} observed the case where $Q=\R^3$ and $K$ is a knot or link in $\R^3$.
\cite{S} used the derived category of sheaves with microsupport in  $\{ (q,p) \in T^*\R^3 \mid p= 0 \text{ or } (q, \frac{p}{|p|}) \in \Lambda_K \text{ if }p\neq 0\}$, and \cite{ENS} used the enhanced knot contact homology of $K$, in order to show the following:
For any knots $K_0$ and $K_1$ in $\R^3$, if $\Lambda_{K_0}$ is Legendrian isotopic to $\Lambda_{K_1}$ in $U^*\R^3$, then $K_0$ is isotopic to either $K_1$ or the mirror of $K_1$ \cite[Theorem 2]{S}, \cite[Theorem 1.1]{ENS}.

In higher dimension, such a strong answer has not been discovered.
By \cite[Theorem 7]{S}, for any compact submanifolds $K_0,K_1$ of a non-compact manifold $Q$, if $\Lambda_{K_0}$ is Legendrian isotopic to $\Lambda_{K_1}$, then there exists an isomorphism between the integral group rings $\Z[\pi_1(Q \setminus K_0)] \cong \Z[ \pi_1(Q \setminus K_1)]$.
Asplund \cite{Asp} observed higher dimensional case by using wrapped Floer cohomology stopped by a unit conormal bundle.  For $n=5$ or $n\geq 7, $ let $U$ be the unknot in $S^n$ of codimension $2$. By \cite[Theorem 2]{Asp}, there exists of an embedded sphere $K\subset S^n$ of codimension $2$ such that
$\Lambda_{K\sqcup \{x_0\}}$ is not Legendrian isotopic to $\Lambda_{U\sqcup \{x_0\}}$ and $\pi_1(S^n \setminus K)\cong \pi_1(S^n \setminus U)\cong \Z$. Here, $x_0$ is a point in $S^n\setminus (U\cup K)$. If we regard $U$ and $K$ as submanifolds of $S^{n}\setminus \{x_0\} \cong \R^n$, then $\Lambda_U$ is not Legendrian isotopic to $\Lambda_K$ in $U^*\R^n$.

In order to distinguish unit conormal bundles up to Legendrian isotopy, the classical invariants are not sufficient. 
For instance, for any pair $(K_0, K_1)$  of knots in $\R^3$, $\Lambda_{K_0}$ and $\Lambda_{K_1}$ have the same classical invariants.
An effective invariant would be the \textit{Chekanov-Eliashberg differential graded algebra (DGA)} defined for Legendrian submanifolds $\Lambda$ under certain conditions. Its homology is called the \textit{Legendrian contact homology of }$\Lambda$. Let us denote it by $\LCH_*(\Lambda)$ with coefficients in $\Z/2$. It was introduced by Chekanov \cite{C} for Legendrian knots in $\R^3$ and by Eliashberg \cite{El} including higher dimensional case. A rigorous definition and a proof of invariance was given by \cite{EES-R, EES}.
For example, by computing the DGA, \cite{EES-non} gave an infinite family of pairwise non-isotopic Legendrian submanifolds of $\R^{2n+1}$ for $n\geq 2$ (with a standard contact structure) whose classical invariants are the same.
However, it is difficult in general to compute the Legendrian contact homology, especially 
when the dimension of a contact manifold is greater than $3$.

\

\noindent
\textbf{Main results.}
In this paper, we give new examples of pairs of compact connected submanifolds of $\R^n$ whose unit conormal bundles are not Legendrian isotopic in $U^*\R^n$. The differences from the examples in previous works \cite{S, ENS, Asp} are that the codimension is greater than or equal to $4$ and that they are not spheres but various topological types are allowed.

For $n,d \in \Z$ with $\frac{n}{2}\geq d\geq 2$, let $M$ be a compact connected submanifold in $\R^{n-d}$ of codimension $d-1$.
By taking the direct product with the unit sphere $S^{d-1}\subset \R^{d}$, we obtain a compact submanifold of $\R^{n}$
\[K_0\coloneqq S^{d-1}\times M \subset \R^{d}\times  \R^{n-d}= \R^{n}\]
whose codimension is $d$.
In Section \ref{sec-application}, we explicitly construct a $C^{\infty}$ embedding $\iota_f \colon K_0 \to \R^{n}$ as a connected sum of $K_0$ and a specific embedding $f\colon S^{n-d}\to \R^n$. Then, we define
\[K_1 \coloneqq \iota_f (K_0) \subset \R^{n} . \]
This is inspired by the construction of the \textit{Hudson torus} in \cite[Example 2.10]{Sk}, which corresponds to the case where $n=p+2d-1$ and $M=S^p\subset \R^{p+ d-1}$ for some $p\in \Z_{\geq 1}$.

We consider the unit conormal bundles $\Lambda_{K_0}$ and $\Lambda_{K_1}$.
From the construction of $K_0$ and $K_1$, we can check that they cannot be distinguished up to Legendrian isotopy by the classical invariants. (In order to define Thurston-Bennequin number as in Definition \ref{def-TB}, we need a topological condition on $K_0$, which is satisfied when $d$ is an even number.)
In addition, if $d\geq 3$, then $\pi_1(\R^n\setminus K_i)\cong \{1\}$ for $i=0,1$, hence they cannot be distinguished by using \cite[Theorem 7]{S}.
%
The main result of this paper is the following.
\begin{thm}[Theorem \ref{thm-Leg-non-isotopic}]\label{thm-1}
If $d\geq 4$ and $H_k(M;\Z/2)\neq 0$ for some $k\in \Z$ with $1\leq k\leq n-2d$, then
$\Lambda_{K_0}$ is not Legendrian isotopic to $\Lambda_{K_1}$ in $U^*\R^n$.
\end{thm}
Note that $n-2d = \dim M-1$. The lowest dimensional case is when $n=9$, $d=4$ and $M\ncong S^2$ is a compact surface embedded in $\R^5$. 
For a possible approach using \cite[Theorem 1]{Asp}, see Remark \ref{rem-Asp}.

We also compute the dimension of Legendrian contact homology with coefficients in $\Z/2$ in low degrees.
\begin{thm} [Theorem \ref{thm-LCH-conormal}]\label{thm-2}
If $d\geq 4$ and $H_k(M;\Z/2)\neq 0$ for some $k\in \Z$ with $1\leq k\leq \min \{n-2d, \frac{d-3}{2}\}$, then
\[\dim_{\Z/2} \LCH_{2k +2d-4}(\Lambda_{K_0}) > \dim_{\Z/2} \LCH_{2k+2d-4}(\Lambda_{K_1}) . \]
\end{thm}

Let us explain the strategy of the proofs of Theorem \ref{thm-1}. For any compact submanifold $K$ of $\R^n$ of codimension $d$, let us abbreviate in this section
\[ \H_*(K) \coloneqq H_{*+1-d}(K\times K,\Delta_K ;\Z/2) , \]
where $\Delta_K\coloneqq \{(q,q') \in K\times K \mid q=q'\}$ is the diagonal.
In Subsection \ref{subsubsec-coproduct-singular}, we will define a $\Z/2$-linear map
\[ \delta_K\colon \H_*(K) \to \H_*(K)^{\otimes 2} \]
in a topological way.
We regard it as a coproduct on $\H_*(K)$.
In Subsection \ref{subsec-coproduct-Morse}, we also give its another description using Morse theory on $K\times K$.
As explained in Remark \ref{rem-string}, the construction is inspired by \textit{string topology} on the path space $\{\gamma \colon [0,1]\to \R^n\colon C^{\infty} \mid \gamma(0),\gamma(1)\in K\}$.

As examples, we show in Subsection \ref{subsubsec-Leg-non-isotopic} that $\delta_{K_0} =0 $ for any $M$ and $\delta_{K_1}\neq 0$ if $H_k(M;\Z/2) \neq 0$ for some $k\in \Z$ with $1\leq k\leq n-2d$. To deduce Theorem \ref{thm-1},
the key result in general case is the following:
\begin{thm}[Theorem \ref{thm-Leg-delta-K}]\label{thm-3}
Let $K$ and $K'$ be compact connected submanifolds of $\R^n$ whose codimensions are greater than or equal to $4$. If $\Lambda_{K}$ is Legendrian isotopic to $\Lambda_{K'}$ in $U^*\R^n$, then there exists a $\Z/2$-linear isomorphism $\Theta \colon \H_*(K) \to \H_*(K')$ such that the following diagram commutes:
\[\xymatrix{
\H_*(K) \ar[r]^-{\delta_{K}} \ar[d]_-{\Theta}^-{\cong} & \H_*(K)^{\otimes 2} \ar[d]_-{\Theta\otimes \Theta}^-{\cong} \\
\H_*(K') \ar[r]^-{\delta_{K'} }& \H_*(K')^{\otimes 2} .
}\]
\end{thm}

Let us outline the proof.
This theorem is proved via a chain complex $(CL_*(\Lambda),d_J)$ of $\Z/2$-vector spaces and a chain map
\[ \delta_J \colon CL_{*+1}(\Lambda) \to CL_*(\Lambda)^{\otimes 2}\]
defined for a Legendrian submanifold $\Lambda$ of $U^*\R^n$ and an almost complex structure $J$ on $\R\times U^*\R^n$ satisfying certain conditions (Subsection \ref{subsubsec-degree-bound}).
$d_J$ is defined by counting $J$-holomorphic strips in $\R\times U^*\R^n$ with boundary in $\R\times \Lambda$, so we call the homology of $(CL_*(\Lambda),d_J)$ the \textit{strip Legendrian contact homology} of $\Lambda$. $\delta_J$ is defined by counting $J$-holomorphic curves with a single positive puncture and two negative punctures in $\R\times U^*\R^n$. A condition in ($\bigstar$), introduced in Subsection \ref{subsubsec-degree-bound}, on the Conley-Zehnder index of the Reeb chords of $\Lambda$ is crucial for the construction of $((CL_*(\Lambda),d_J) , \delta_J)$. This is an analogy to the condition on a contact manifold  to be \textit{index-positive} \cite[Definition 3.10]{U}.
Under Legendrian isotopy, these are invariant up to chain homotopy.

For any compact connected submanifold $K$ of $\R^n$, $((CL_*(\Lambda_K),d_J) , \delta_J)$ can be defined if $\codim K\geq 4$ (Subsection \ref{subsec-LCH-conormal}).
In Subsection \ref{subsec-isom-chain-cpx}, we construct an isomorphism from the strip Legendrian contact homology of $\Lambda_K$ to $\H_{*+1}(K)$. Moreover, in Subsection \ref{subsec-compatibility}, we show that this isomorphism intertwines $\delta_K$ and the map on the homology induced by $\delta_J$. (In fact, the isomorphism factors through the Morse homology of $(K\times K,\Delta_K)$. See Theorem \ref{thm-isom-LM} and Theorem \ref{thm-diagram}.)
The idea of the proof is derived from a generalization of some arguments in \cite{CELN}  to higher dimensional cases. To construct the isomorphism, we will use moduli spaces of pseudo-holomorphic curves in $T^*\R^n$ with switching Lagrangian boundary conditions: the conormal bundle of $K$ and the image of the zero section. See Subsection \ref{subsec-switching}. Such moduli spaces and their compactification have been studied in \cite{CEL, CELN}.

\

\noindent
\textbf{Possible extension.}
If $K$ is a compact oriented submanifold of $\R^n$, the definition of $\delta_K$ in Subsection \ref{subsubsec-coproduct-singular} can be lifted to a $\Z$-linear map
\[ \delta^{\Z}_K\colon H_*(K\times K;\Delta_K ;\Z) \to H_{*+1-d}((K\times K;\Delta_K)^{\times 2} ;\Z) \]
by using the Thom class of the oriented normal bundle of $K$.
The author expects that Theorem \ref{thm-3} can be enhanced to a statement about the invariance of $\delta^{\Z}_K$ under Legendrian isotopies of $\Lambda_K$.
We do not pursue it here because it requires more elaborations to enhance the argument outlined above to $\Z$-coefficients, involving discussion on orientations of moduli spaces.
More precisely, we will need to fix an appropriate spin structure on $\Lambda_K$, and then define the strip Legendrian contact homology for $\Lambda_K$ and a coproduct on it over $\Z$ (one may refer to \cite[Section 4.1]{EES-ori} for the definition of the Chekanov-Eliashberg DGA of $\Lambda_K$ over $\Z$ or the group ring $\Z[H_1(\Lambda_K)]$).
Accordingly, an enhancement of Theorem \ref{thm-3} would be about Legendrian isotopies preserving the spin structures.


\

\noindent
\textbf{Relation to \cite{O}.}
In \cite{O}, the author defined a graded $\R$-algebra $H^{\str}_*(Q,K)$ for any oriented manifold $Q$ and its compact oriented submanifold $K$. It is an extension of the $0$-th degree part of the string homology (with the coefficient ring reduced from $\Z[H_1(\Lambda_K)]$ to $\R$) defined by \cite{CELN} when $\dim Q=3$ and $K$ is a knot in $Q$.
The author made a conjecture \cite[Conjecture 1.4]{O} that when $Q=\R^n$, $H^{\str}_*(\R^n,K)$ is isomorphism to the Legendrian contact homology of $\Lambda_K$ with coefficients in $\R$.

$H^{\str}_*(Q,K)$ is constructed from a chain complex whose differential involves a chain-level coproduct operation which comes from the same idea as Remark \ref{rem-string} in string topology.
There is a spectral sequence induced by a filtration on this chain complex such that the differential of the second page is given by the homology-level coproduct \cite[Section 4.1]{O}. From a simple argument on the grading, it is possible to see that when $\codim K  = d \geq 3$, $\dim_{\R} H^{\str}_p(Q,K)$ for $p\leq 3d-7$ can be computed from the second page.
Therefore, if the conjecture in \cite{O} is true, we can claim as follows: When $\codim K\geq 3$, the Legendrian contact homology of $\Lambda_K$ with coefficients in $\R$ in degree less than or equal to $3d-7$ can be computed in terms of the coproduct in homology-level.
In fact, Theorem \ref{thm-2} is a consequence of Theorem \ref{thm-LCH-deltaK} which proves a variant of this claim in $\Z/2$-coefficient when $\codim K\geq 4$.

\

\noindent
\textbf{Organization of paper.}
In Section \ref{sec-Leg-submfd}, we define the Chekanov-Eliashberg differential graded algebra for a Legendrian submanifold $\Lambda$ in a contact manifold satisfying several conditions. Under the condition ($\bigstar$) on $\Lambda$ introduced in Subsection \ref{subsubsec-degree-bound}, we define the strip Legendrian contact homology and a coproduct on it, and show their invariance under Legendrian isotopy.
In Section \ref{sec-Morse}, for any compact submanifold $K$ of $\R^n$, we define coproducts on both the Morse homology and the singular homology of $(K\times K,\Delta_K)$. The coproduct on singular homology is denoted by $\delta_K$. We also check the coincidence of the two definitions.
In Section \ref{sec-Floer-Morse}, we consider the case where $K$ is connected and its codimension is greater than or equal to $4$. We give an isomorphism from the strip Legendrian contact homology of $\Lambda_K$ to the Morse homology of $(K\times K,\Delta_K)$. The important result is that it intertwines the coproducts on the two homology.
In Section \ref{sec-application}, we first prove Theorem \ref{thm-3} in a general case. Next, we explicitly define a pair $(K_0,K_1)$ of compact submanifolds of $\R^n$ and compute $\delta_{K_0}$ and $\delta_{K_1}$ to prove the main results Theorem \ref{thm-1} and Theorem \ref{thm-2}.
Appendix \ref{appendix} gives proofs of several fundamental results.  

\

\noindent
\textbf{Notations and conventions.}
Throughout this paper, any manifold and its submanifold are of class $C^{\infty}$, unless otherwise specified. Let us use the following notations:
\begin{itemize}
\item For any pair of maps $f\colon X\to Z$ and $g\colon Y\to Z$, we denote their fiber product over $Z$ by $X\ftimes{f}{g} Y\coloneqq \{(x,y) \in X\times Y \mid f(x)=g(y)\}$.
\item For any finite set $S$, let $\#_{\Z/2}S\in \Z/2$ denote the cardinality of $S$ modulo $2$.
\end{itemize}

\noindent
\textbf{Acknowledgements.}
The author would like to thank Kei Irie for reading the
draft and making suggestions for improvements. He would also like to thank Kaoru Ono and Kai Cieliebak for  making valuable comments.
This work, especially on Subsection \ref{subsec-coproduct-Morse} describing $\delta_K$ using Morse theory, has been encouraged by Mari\'{a}n Poppr's research, which focuses on the case where $K$ has codimension $1$.
The author would like to thank the referee for many valuable comments and suggestions.
This work was supported by JSPS KAKENHI Grant Number JP23KJ1238.

\section{Legendrian submanifold and isotopy invariants}\label{sec-Leg-submfd}

\subsection{Preliminaries}

\subsubsection{Setup and notations}
For $n\in \Z_{\geq 1}$, let $Y$ be a ($2n-1$)-dimensional manifold. A $1$-form $\alpha\in \Omega^1(Y)$ is called a \textit{contact form} on $Y$ if the $(2n-1)$-form $\alpha\wedge (d\alpha)^{n-1}$ is nowhere vanishing. We call the pair $(Y,\alpha)$ a contact manifold with a contact form $\alpha$.
The subbundle $\xi \coloneqq \ker\alpha$ of $TY$ is a contact distribution on $Y$. The bilinear form $\rest{d \alpha}{\xi}$ gives $\xi$ a structure of a symplectic vector bundle.

The \textit{Reeb vector filed} $R_{\alpha}$ of $(Y,\alpha)$ is a vector field on $Y$ characterized by
\[ \begin{array}{cc}
 (d\alpha)_x(R_{\alpha}, \cdot )=0 , & \alpha_x(R_{\alpha})=1,
\end{array}\]
for every $x\in Y$. We denote its flow by $(\varphi^t_{\alpha})_{t\in \R}$ when $R_{\alpha}$ is a complete vector field.

Next, let $\Lambda$ be an $(n-1)$-dimensional submanifold of $Y$.
$\Lambda$ is called a \textit{Legendrian submanifold} of $(Y,\alpha)$ if $\rest{\alpha}{\Lambda}=0 \in \Omega^1(\Lambda)$. We define the set of \textit{Reeb chords} of $\Lambda$ by
\[\mathcal{R}(\Lambda) \coloneqq \{a \colon [0,T]\to Y \mid T>0,\ \textstyle{\frac{da}{dt}(t)} = (R_{\alpha})_{a(t)} \text{ for every }t\in [0,T], \text{ and }a(0),a(T)\in \Lambda\}.\]
A Reeb chord $(a\colon [0,T]\to Y)\in \mathcal{R}(\Lambda)$ is called \textit{non-degenerate} if 
$ d\varphi^T_{\alpha}(T_{a(0)}\Lambda) $ is transversal to $T_{a(T)}\Lambda$ in $\xi_{a(T)}$.

Let us also fix several notations. 
On the complex vector space $\C^n$ with the coordinate $(z_1,\dots ,z_n)$, we fix the symplectic form $\frac{1}{2}\sum_{i=1}^n dz_i\wedge d\bar{z_i}$.
Let $\mathcal{L}(\C^n)$ denote the space of all Lagrangian subspaces of $\C^n$.
For any continuous path $\Gamma\colon [0,T]\to \mathcal{L}(\C^n)$ with $T>0$ and $V\in \mathcal{L}(\C^n)$, the Maslov index $\mu(\Gamma,V) \in \frac{1}{2}\Z$ is defined by \cite[Lemma 2.1, Lemma 2.2]{RS}.
When $\Gamma(0)=\Gamma(T)$, we get a continuous loop $\Gamma \colon S^1 \cong \R/ T\Z \to \mathcal{L}(\C^n)$, and its Maslov index $\mu(\Gamma) \coloneqq \mu(\Gamma, V) \in \Z$ is independent of $V$ \cite[Remark 3.6]{RS}.

\subsubsection{Maslov class and Thurston Bennequin number}\label{subsubsec-Maslov-TB}
We define two invariants of Legendrian submanifolds. Hereafter, we abbreviate $(Y,\alpha)$ by $Y$.
First, referring to \cite[Section 2.2]{EES}, we define the Maslov class for a Legendrian submanifold.
\begin{defi}\label{def-Maslov-class}
Suppose that $2c_1(\xi)=0 \in H^2(Y;\Z)$ and $H_1(Y;\Z)=0$.
For any Legendrian submanifold $\Lambda$ of $Y$,
its \textit{Maslov class} $\mu_{\Lambda}\in \Hom (H_1(\Lambda;\Z) ;\Z) \cong H^1(\Lambda;\Z)$ is defined as follows:
For any loop $\gamma\colon S^1 \to \Lambda$, since $H_1(Y;\Z)=0$, there exists a compact oriented surface $\Sigma$ and a continuous map $u\colon \Sigma\to Y$ such that $\partial \Sigma \cong S^1$ and $\rest{u}{\partial \Sigma}=\gamma$. We choose a symplectic trivialization $\Phi \colon u^*\xi \to \C^{n-1}$, then we obtain a loop
\[ \Gamma\colon S^1 \to \mathcal{L}(\C^{n-1})\colon t\mapsto (\rest{\Phi}{S^1})_t(T_{\gamma(t)}\Lambda).\]
We define $\mu_{\Lambda}([\gamma]) \coloneqq \mu(\Gamma) \in \Z$ for $[\gamma]\in H_1(\Lambda ;\Z)$. From the condition that $2c_1(\xi)=0$, this definition is independent of $u$ and $\Phi$.
\end{defi}

Next, we consider a natural exact sequence for the pair $(Y,\Lambda)$
\begin{align}\label{ex-seq-Hn}
\xymatrix{
H_{n}(Y;\Z) \ar[r]^-{j} & H_n(Y,\Lambda;\Z) \ar[r]^-{\partial} & H_{n-1}(\Lambda;\Z) \ar[r]^-{i} & H_{n-1}(Y;\Z)
}
\end{align}
and the intersection product
\[H_{n}(Y,\Lambda;\Z)\otimes H_{n-1}(Y\setminus \Lambda ;\Z) \to \Z\colon a\otimes b\mapsto a\bullet b,\]
where $Y$ is oriented so that $\alpha \wedge (d \alpha)^{n-1}$ is a positive volume form. 
Referring to \cite[Section 3.4]{EES-non}, we give the following definition.
\begin{defi}\label{def-TB}
Let $\Lambda$ be a compact Legendrian submanifold of $Y$.
We suppose that  $H_n(Y;\Z)=0$ and that $\Lambda$ is an oriented submanifold with $i ([\Lambda]) =0$ for the fundamental class $[\Lambda]\in H_{n-1}(\Lambda;\Z)$. By (\ref{ex-seq-Hn}), there exists a unique homology class $D_{\Lambda}\in H_n(Y,\Lambda;\Z)$ such that $\partial (D_{\Lambda}) = [\Lambda]$.
We take $\epsilon>0$ such that $\varphi^t_{\alpha}(\Lambda) \subset Y\setminus \Lambda$ for every $t\in (0,\epsilon]$. Then, the \textit{Thurston-Bennequin number} of  $\Lambda$ is defined by
\[\tb(\Lambda) \coloneqq D_{\Lambda} \bullet [\varphi^{\epsilon}_{\alpha}(\Lambda) ] \in \Z . \]
This is independent of $\epsilon>0$ and the orientation of $\Lambda$ if $\Lambda$ is connected.
\end{defi}

The next proposition is a standard fact about the invariance of the Maslov class and the Thurston-Bennequin number.
As a terminology, a $C^{\infty}$ family $(f_t\colon N\to Y)_{t\in [0,1]}$ of immersions of an ($n-1$)-dimensional manifold $N$ such that $f_t^*\alpha=0$ for every $t\in [0,1]$ is called a \textit{Legendrian regular homotopy}.
\begin{prop}
Let $N$ be an ($n-1$)-dimensional manifold. Let $(f_t\colon N\to Y)_{t\in [0,1]}$ be a Legendrian regular homotopy and $f_t$ is an embedding for $t\in \{0,1\}$. Then
$f_0^*(\mu_{f_0(N)}) = f_1^*(\mu_{f_1(N)}) $
in $H^1(N;\Z)$.
Moreover, suppose that $N$ is a compact oriented manifold and $(f_0)_*([N])=0 \in H_{n-1}(Y;\Z)$. If $f_t$ is an embedding for every $t\in [0,1]$, then
$\tb(f_0(N)) = \tb (f_1(N))$.
\end{prop}

\subsection{Conley-Zehnder index of Reeb chord}

\subsubsection{Definition via capping path}\label{subsubsec-capping}

We assume that $2c_1(\xi)=0$ and $H_1(Y;\Z)=0$ as in Definition \ref{def-Maslov-class}.
Let us consider a Legendrian submanifold $\Lambda$ of $Y$ such that
$\Lambda$ is  connected and $\mu_{\Lambda}=0\in H^1(\Lambda;\Z)$.
For any Reeb chord $(a\colon [0,T]\to Y)\in \mathcal{R}(\Lambda)$,
we choose a continuous path $\gamma_a\colon [0,1]\to \Lambda$ such that $\gamma_a(0)= a(T)$ and $\gamma_a(1)=a(0)$.
Since $H_1(Y;\Z)=0$, there exists a continuous map $f \colon \Sigma \to Y$ on a compact oriented surface $\Sigma$ which bounds the loop
\[ l_a\colon \R/(T+1)\Z \to Y \colon t\mapsto \begin{cases}  \gamma_a(t)& \text{ if }0\leq t\leq 1, \\ a(t-1) & \text{ if }1\leq t\leq T+1. \end{cases}\]
We also choose a symplectic trivialization $\Psi\colon (l_a)^* \xi \to \C^{n-1}$ over $\R/ (T+1)\Z$ such that:
\begin{itemize}
\item $\Psi_{t+1} = \Psi_0 \circ d\varphi_{\alpha}^{T-t} \colon \xi_{a(t)} \to \C^{n-1}$
for every $t\in [0,T]$.
\item $\Psi$ extends to a symplectic trivialization of $f^* \xi$ on $\Sigma$.
\end{itemize}
Then, we obtain a path
\[  \Gamma_a\colon [0,1]\to \mathcal{L}(\C^{n-1}) \colon t\mapsto  \Psi_t(T_{\gamma_a(t)}\Lambda) .\]
from $\Gamma_a(0)=\Psi_0(T_{a(T)}\Lambda)$ to $\Gamma_a(1)= \Psi_0\circ d\varphi^T_{\alpha}(T_{a(0)}\Lambda)$.
Note that the Reeb chord $a$ is non-degenerate if and only if $\Gamma_a(0)$ is transversal to $\Gamma_a(1)$.
\begin{defi}\label{def-ind-cap}
For any $a\in \mathcal{R}(\Lambda)$, 
we define its \textit{Conley-Zehnder index} $\mu (a) \in \frac{1}{2} \Z$ by 
\[ \mu(a) \coloneqq \mu ( \Gamma_a, \Gamma_a(0)) + \frac{n-1}{2}. \]
By \cite[Theorem 2.4]{RS}, $\mu(a)$ is an integer if $a$ is non-degenerate.
From the condition that $2c_1(\xi)=0$ and $\mu_{\Lambda}=0$,
$\mu(a)$ is independent of the choice of $\gamma_a$, $f$ and $\Psi$.  
\end{defi}

\begin{rem}
Take a path $\lambda$ from $\Gamma_a(1)$ to $\Gamma_a(0)$ by a positive rotation as in \cite[Remark A.1]{CEL} (see also \cite[Subsection 2.2]{EES-R}). Its Maslov index is $\mu(\lambda,\Gamma_a(0))= \frac{n-1}{2}$. Let $\Gamma_a * \lambda$ be the loop obtained by concatenating $\Gamma_a$ and $\lambda$. Then, $\mu(a)$ is equal to
\[ \mu(\Gamma_a,\Gamma_a(0)) + \mu(\lambda,\Gamma_a(0)) = \mu(\Gamma_a*\lambda)\]
by \cite[Theorem 2.3]{RS}.
The right-hand side coincides with the definition of the Conley-Zehnder index in \cite[Subsection 2.2]{EES-R}.
\end{rem}

\subsubsection{Example: Index of Reeb chords and Morse index of binormal chords}\label{subsubsec-Morse index}

We focus on an example which is of main interest in this paper. The Euclidean space $\R^n$ with the coordinate $(q_1,\dots ,q_n)$ is equipped with the standard metric $\la\cdot,\cdot\ra \coloneqq \sum_{i=1}^n dq_i\otimes dq_i$.
We identify its cotangent bundle $T^*\R^n$ with $\R^n \times \R^n$ by the diffeomorphism
\begin{align}\label{identify-TRn}
 \R^n\times \R^n \to T^*\R^n \colon (q, (p_1,\dots ,p_n)) \mapsto \sum_{i=1}^n p_i (dq_i)_{q}.
 \end{align}
We denote the canonical Liouville form on $T^*\R^n$ by $\lambda \coloneqq \sum_{i=1}^np_i dq_i$.
Consider the unit cotangent bundle of $\R^n$
\[U^*\R^n \coloneqq \{(q,p)\in T^*\R^n \mid |p|=1\}.\]
Then, $\alpha \coloneqq \rest{\lambda}{U^*\R^n}$ is a canonical contact form on $U^*\R^n$.

Let $K$ be a compact $C^{\infty}$ submanifold of $\R^n$. The conormal bundle of $K$ is defined by
\[L_K\coloneqq \{(q,p) \in T^*\R^n \mid q\in K ,\ \la p ,v \ra =0 \text{ for every }v\in T_qK\}.\]
It is a Lagrangian submanifold of $T^*\R^n$ with $\rest{\lambda}{L_K}=0$.
We call its unit sphere bundle
\[\Lambda_K\coloneqq L_K\cap U^*\R^n\]
the \textit{unit conormal bundle} of $K$. Since $\rest{\alpha}{\Lambda_K}=0$, $\Lambda_K$ is a Legendrian submanifold of $U^*\R^n$.

Next, consider a $C^{\infty}$ function
\begin{align}\label{energy-fcnl}
E\colon K\times K\to \R\colon (q,q') \mapsto \frac{1}{2}|q-q'|^2.
\end{align}
Obviously, $\min E=0$ and $E^{-1}(0)$ coincides with the diagonal
\[\Delta_K\coloneqq\{(q,q')\in K\times K \mid q=q'\}\]
Let $\mathcal{C}(K)$ denote the set of all critical points of $E$ in $(K\times K) \setminus \Delta_K$. For every $x\in \mathcal{C}(K)$, we denote its Morse index with respect to $E$ by $\ind x\in \Z_{\geq 0}$.
\begin{rem}\label{rem-sp-emb}
Given any compact submanifold $K$ of $\R^n$, consider the normal bundle $(TK)^{\perp}$ and fix $r_K>0$ such that $\{q+v \mid (q,v)\in (TK)^{\perp},\ |v|<r_K\}$ is a tubular neighborhood of $K$. We define $\mathcal{O}_K$ to be the set of $C^{\infty}$ maps $\sigma \colon K\to \R^n  $ such that $\sigma(q) \in (T_q K)^{\perp}$ and $|\sigma(q)| \leq \frac{r_K}{2}$ for every $q\in K$. $\mathcal{O}_K$ is equipped with the $C^{\infty}$ topology. For every $\sigma \in \mathcal{O}_K$, we define
\[K_{\sigma} \coloneqq \{q+\sigma(q) \mid q\in K\}. \]
Then, the subset $\{\sigma \in \mathcal{O}_K \mid \text{every }x\in \mathcal{C}(K_{\sigma}) \text{ is a non-degenerate critical point of }E\}$ is open dense in $\mathcal{O}_K$.
\end{rem}

$\mathcal{C}(K)$ coincides with the set
\[ \{ (q,q')\in (K\times K)\setminus \Delta_K \mid (q-q') \perp T_qK \text{ and }  (q-q') \perp T_{q'}K \}.\]
Here, we regard $q-q'$ as a vector in $\R^n$.
From this description, it is easy to see that we have a bijection
\begin{align}\label{bijection-RC}
\mathcal{R}(\Lambda_K) \to \mathcal{C}(K) \colon (a\colon[0,T]\to U^*\R^n) \mapsto x_a \coloneqq (\pi_{\R^n}\circ a(0), \pi_{\R^n} \circ a(T)) ,
\end{align}
where $\pi_{\R^n} \colon T^*\R^n \to \R^n$ is the bundle projection.
Let $d$ be the codimension of $K$ in $\R^n$.
\begin{prop}\label{prop-correspond}
$a\in \mathcal{R}(\Lambda_K)$ is non-degenerate if and only if $x_a\in \mathcal{C}(K)$ is a non-degenerate critical point of $E$.
Moreover, in such case, $\mu(a)=\ind x_a  + d-1 $.
\end{prop}
For the proof, we refer to \cite[Corollary 4.2]{APS} about a relation between the Conley-Zehnder index and the Morse index. There is a gap between Definition \ref{def-ind-cap} and the definition in \cite{APS} of the Conley-Zehnder index of Hamiltonian chords of $L_K$. We fill this gap and prove Proposition \ref{prop-correspond} in Appendix \ref{subsec-A1}.

\subsection{Legendrian contact homology}\label{subsec-LCH}

Throughout this subsection, we require the following condition on $(Y,\alpha)$:
There exist a $(2n-2)$-dimensional Liouville manifold $(P,\lambda_P)$, where $\lambda_P\in \Omega^1(P)$ is a Liouville $1$-form, and a diffeomorphism $\varphi_P\colon Y\to \R \times P$ such that:
\begin{itemize}
\item $(P, \lambda_P)$ is the completion of a compact Liouville domain $(\bar{P},\lambda_P)$.
\item $2 c_1(TP)=0$ and $H_1(P;\Z)=0$.
\item $(\varphi_P)^*(dz+\lambda_P) =\alpha$. Here, $z$ is the $\R$-coordinate of $\R\times P$.
\end{itemize}
We note that $\partial_z = (\varphi_P)_*(R_{\alpha})$.
In particular, 
there is no periodic orbit along $R_{\alpha}$.

\subsubsection{Pseudo-holomorphic curves with boundary in $\R \times \Lambda$}\label{subsubsec-J-hol}

Let $D \coloneqq \{z\in \C \mid |z| \leq 1\}$. For any $m\in \Z_{\geq 0}$, fix boundary points $p_0,p_1,\dots ,p_m\in \partial D$ ordered clockwise  and define
\[D_{m+1}\coloneqq D\setminus \{p_0,p_1,\dots ,p_m\}.\]
Let $\mathcal{C}_{m+1}$ be the space of conformal structures on $D_{m+1}$ which can be smoothly extended to the disk $D$.
For any $\kappa \in \mathcal{C}_{m+1}$, let $j_{\kappa} \colon TD_{m+1} \to TD_{m+1}$ be a complex structure on $D_{m+1}$ associated to $\kappa$.
We take holomorphic strip coordinates
\[\begin{array}{cc}
\psi_0\colon [0,\infty)\times [0,1] \to D_{m+1}, & \psi_k\colon  (-\infty,0] \times [0,1] \to D_{m+1} \end{array}\]
near $p_0$ and $p_k$ for $k=1,\dots ,m$ respectively. Here, the complex structures on the domain of $\psi_0$ and $\psi_k$ are determined so that the map $ \R \times [0,1] \to \C \colon (s,t)\mapsto s+ \sqrt{-1} t$ is holomorphic.
For $m=0,1$, $\mathcal{C}_{m+1}$ consists of a single element and we denote by $\mathcal{A}_{m+1}$ the group of automorphisms on $D_{m+1}$.

For every $s\in \R$ , we define a transition map
\[\tau_s\colon  \R \times Y \to  \R \times Y \colon (r,x)\mapsto (r+s,x).\]
We consider almost complex structures $J$ of class $C^{\infty}$ on $\R\times Y$ satisfying the following conditions:
\begin{enumerate}
\item $J$ is compatible with $d (e^r \alpha)$, namely $d(e^r \alpha)(\cdot, J\cdot)$ is a Riemannian metric on $\R\times Y$.
\item $J(\xi)=\xi$ and $J(\partial_r)=R_{\alpha}$. Here, $r$ is the $\R$-coordinate of $\R\times Y$.
\item $\tau_s^*J=J$ for every $s\in \R$.
\end{enumerate}
Suppose that there exists $J_{\mathrm{fix}}$ satisfying the above three  conditions and
\begin{itemize}
\item There exists compact subsets $\bar{Y}_{l}$ ($l=1,2,\dots$) of $Y$ such that $\bar{Y}_1 \subset \bar{Y}_2 \subset \dots \subset Y= \bigcup_{l=1}^{\infty} \bar{Y}_l$ with the property as follows: For any $J_{\mathrm{fix}}$-holomorphic curve $u\colon \Sigma \to \R\times Y$ on a compact Riemann surface $(\Sigma,j)$ with $u(\partial \Sigma)\subset \R\times \bar{Y}_l$, its image $u(\Sigma)$ must be contained in $\R\times \bar{Y}_l$.
\end{itemize}
We define $\mathcal{J}_{\bar{Y}_l}$ to be the set of almost complex structures $J$ which satisfy the above three condition and $J=J_{\mathrm{fix}}$ outside $\bar{Y}_l$.
$\mathcal{J}_{\bar{Y}_l}$ is endowed with the $C^{\infty}$ topology.

Let $\Lambda$ be a compact connected Legendrian submanifold of $(Y,\alpha)$ with $\mu_{\Lambda}=0$. Take $l\in \Z_{\geq 1}$ such that $\Lambda$ is contained in the interior of $\bar{Y}_l$. Suppose that every Reeb chord of $\Lambda$ is non-degenerate.
For $J\in \mathcal{J}_{\bar{Y}_l}$, we will define a moduli space of $J$-holomorphic curves in $\R\times Y$.
Referring to \cite[Definition 3.7]{DR} (especially the case where the almost complex structure and the Lagrangian boundary condition are cylindrical), we first define a space
\begin{align}\label{moduli-hat}
\hat{\mathcal{M}}_{J}(a;b_1,\dots ,b_m)
\end{align}
for any Reeb chords $(a\colon [0,T]\to Y)$ and $(b_k\colon [0,T_k]\to Y) \in \mathcal{R}(\Lambda)$ for $k=1,\dots ,m$.
This consists of pairs $(u,\kappa)$ of $\kappa \in \mathcal{C}_{m+1}$ and a $C^{\infty}$ map
\[u\colon D_{m+1} \to \R\times Y\]
satisfying the following conditions:
\begin{itemize}
\item $du + J\circ du \circ j_{\kappa} =0$.
\item $u(\partial D_{m+1} ) \subset \R\times \Lambda$.
\item There exist $s_0, s_1,\dots ,s_m\in \R$ such that for $k=1,\dots ,m$,
\begin{align*}
& \tau_{-Ts} \circ u\circ \psi_0(s,\cdot) \to (s_0, a(T\cdot)) \text{ if }s\to \infty,  \\
& \tau_{-T_ks}\circ u\circ \psi_k(s,\cdot)\to (s_k, b_k(T_k\cdot) ) \text{ if }s\to -\infty ,
\end{align*}
in the $C^{\infty}$ topology on $[0,1]$.
\end{itemize}
The group $\R$ acts on (\ref{moduli-hat}) by $s\cdot (u,\kappa) \coloneqq (\tau_s\circ u, \kappa)$ for every $s\in \R$. In addition,
when $m=0,1$, the groups $\mathcal{A}_{m+1}$ acts on (\ref{moduli-hat}) by $(u,\kappa )\cdot \varphi\coloneqq (u\circ \varphi, \kappa)$ for every $\varphi\in \mathcal{A}_{m+1}$. The moduli space of $J$-holomorphic curves up to the $\R$-action and conformal equivalence is defined by
\begin{align}\label{moduli-cobordism}
\mathcal{M}_{J}(a;b_1,\dots ,b_m) \coloneqq \begin{cases}
\R \backslash \hat{\mathcal{M}}_{J}(a;b_1,\dots ,b_m) & \text{ if }m\geq 2, \\
\R \backslash \hat{\mathcal{M}}_{J}(a;b_1,\dots ,b_m)/\mathcal{A}_{m+1} & \text{ if }m=0,1.
\end{cases}
\end{align}
When $m=0$, we denote it by $\M_J(a;\emptyset)$.

The transversality for this moduli space can be achieved by a perturbation of $J$ as below. We refer to \cite[Proposition 3.13]{DR} and sketch the proof. In addition, let us define
\[|a| \coloneqq \mu(a) -1 \]
for every $a\in \mathcal{R}(\Lambda)$ in order to compute the dimension of the moduli space.
\begin{prop}\label{prop-transverse}
Suppose that $(b_1,\dots ,b_m)\neq (a)$.
For  generic $J\in \mathcal{J}_{\bar{Y}_l}$,  the moduli space $\mathcal{M}_{J}(a;b_1,\dots ,b_m) $ is cut out transversely and its dimension as a $C^{\infty}$ manifold is
\[|a| -1 - \sum_{k=1}^m |b_k|.\]
\end{prop}
\begin{proof}
Let $\pr_{\xi}\colon TY = \R\cdot R_{\alpha} \oplus \xi \to \xi$ be the projection. For any $[(u,\kappa)] \in \mathcal{M}_{J}(a;b_1,\dots ,b_m)$, let
 $\bar{u} \colon D_{m+1}\to Y$ be the $Y$-component of $u$
and $C_u$ be the subset of $D_{m+1}$ consisting of $z\in D_{m+1}\setminus (\partial D_{m+1}\cup \bar{u}^{-1}(\partial \bar{Y}_l))$ such that:
\begin{itemize}
\item $\pr_{\xi}\circ (d\bar{u})_z\colon T_z D_{m+1} \to \xi_{\bar{u}(z)}$ is a linear injective map.
\item $(\bar{u})^{-1}(\bar{u}(z)) =\{z\}$.
\end{itemize}
Then, $C_u \cap \psi_0([0,\infty) \times [0,1])$ is not the empty set. See the proof of \cite[Proposition 3.13]{DR}. From the existence of such injective points, the transversality of the moduli space for generic $J\in \mathcal{J}_{\bar{Y}_l}$ is proved by modifying the argument in \cite[Subsection 8.3]{W} to the present case with Lagrangian boundary condition.
For the dimension of the moduli space, see \cite[Section 4.2.4]{DR-lift} and  \cite[Proposition 2.3]{EES}.
\end{proof}

The moduli space $\M_J(a;b_1,\dots,b_m)$ admits a compactification, which we denote by
\[ \overline{\M}_J(a;b_1,\dots ,b_m) , \]
from the Gromov-Hofer compactness \cite[Section 3.3.2]{DR}.
It consists of $J$-holomorphic buildings as defined in \cite[Definition 3.12]{DR} with boundary in $\R\times \Lambda$. 
In particular, when $|b_1|+\dots +|b_m| = |a|-1$, $\M_J(a;b_1,\dots ,b_m)$ is a compact $0$-dimensional manifold for generic $J\in \mathcal{J}_{\bar{Y}_l}$. 

\subsubsection{Definition and invariance of Legendrian contact homology}\label{subsubsec-DGA}

We continue to consider a compact connected Legendrian submanifold $\Lambda$ such that $\mu_{\Lambda}=0$ and every Reeb chord is non-degenerate. 
We define $\mathcal{A}_*(\Lambda)$ to be the unital non-commutative graded $\Z/2$-algebra freely generated by $\mathcal{R}(\Lambda)$. Here, $a\in \mathcal{A}_{|a|}(\Lambda)$ for every $a\in \mathcal{R}(\Lambda)$.

For generic $J$ of Proposition \ref{prop-transverse},
we define a derivation $\partial_J \colon \mathcal{A}_*(\Lambda)\to \mathcal{A}_{*-1}(\Lambda)$  of degree $-1$  by
\[ \partial_J (a) \coloneqq \sum_{|b_1|+\dots +|b_m| =|a|-1} \#_{\Z/2} \mathcal{M}_{J} (a;b_1,\dots ,b_m) \cdot b_1\cdots b_m \]
for every $a\in \mathcal{R}(\Lambda)$
and extend it by the Leibniz rule.
The pair $(\mathcal{A}_*(\Lambda), \partial_J)$ has the following properties.
\begin{thm}
$(\mathcal{A}_*(\Lambda), \partial_J)$ is a differential graded $\Z/2$-algebra. Namely, $\partial_J \circ \partial_J =0$.
Moreover, the isomorphism class of its homology $\ker \partial_J/\Im \partial_J$ as a graded $\Z/2$-algebra is independent of the choice of $J$ of Proposition \ref{prop-transverse} and invariant under a Legendrian isotopy of $\Lambda$.
\end{thm}
 For the proof that $\partial_J\circ \partial_J=0$, see \cite[Subsection 3.2.2]{DR}. For the proof of invariance, we refer to \cite[Theorem 3.5]{DR} using exact Lagrangian cobordisms associated to Legendrian isotopies.
$(\mathcal{A}_*(\Lambda), \partial_J)$ is called the \textit{Chekanov-Eliashberg differential graded algebra (DGA)} of $\Lambda$ with coefficients in $\Z/2$.
We denote its homology by
\[ \LCH_*(\Lambda) \coloneqq \ker \partial_J/\Im \partial_J.\]
It is called the \textit{Legendrian contact homology} of $\Lambda$ with coefficients in $\Z/2$.


\subsubsection{Strip Legendrian contact homology and coproduct operation}\label{subsubsec-degree-bound}
In this subsection, we observe the Chekanov-Eliashberg DGA of $\Lambda$ under the following condition:
\begin{itemize}
\item[($\bigstar$)] $\Lambda$ is a compact connected Legendrian submanifold of $(Y,\alpha)$ such that $\mu_{\Lambda}=0$ and every Reeb chord of $\Lambda$ is non-degenerate. Moreover, $|a|\geq 2$ for every $(a\colon [0,T] \to Y)\in \mathcal{R}(\Lambda)$ which is homotopic with end points in $\Lambda$ to a constant path in $\Lambda$.
\end{itemize}
See Remark \ref{rem-ind-pos} and \ref{rem-ind-pos-example} below for related notions in the preceding works.

We consider a filtration $(\mathcal{F}^m)_{m=0}^{\infty}$ of $\mathcal{A}_*(\Lambda)$ by $\Z/2$-linear subspaces
\[ \mathcal{F}^m \coloneqq \bigoplus_{m'=m}^{\infty} \left( \bigoplus_{a_1,\dots ,a_{m'}\in \mathcal{R}(\Lambda)} \Z/2 \cdot (a_1\cdots a_{m'}) \right).\]
It gives a decreasing sequence $\mathcal{A}_*(\Lambda) = \mathcal{F}^0 \supset \mathcal{F}^1 \supset \mathcal{F}^2 \supset \cdots$.
\begin{prop}\label{prop-filtration}
If $\Lambda$ satisfies the condition ($\bigstar$), then
the differential $\partial_J$ on $ \mathcal{A}_*(\Lambda)$ preserves the filtration $(\mathcal{F}^m)_{m=0}^{\infty}$.
\end{prop}
\begin{proof}

 For every $a\in \mathcal{R}(\Lambda)$, $\dim \M_{J}(a,\emptyset)=|a|-1\geq 1$ by the condition that $|a| \geq 2$. Therefore, the moduli space $\M_{J}(a,\emptyset)$ does not contribute to $\partial_J(a)$, and thus  $\partial_J(a)\in \mathcal{F}^1$. Since $\partial_J$ satisfies the Leibniz rule, 
 $\partial_J(a_1\cdots a_m) \in \mathcal{F}^m$ for any $a_1,\dots ,a_m \in \mathcal{R}(\Lambda)$.
\end{proof}
The quotient complex $\mathcal{F}^1/\mathcal{F}^2$ is naturally isomorphic to a chain complex $(CL_*(\Lambda), d_J)$ defined by
\[\begin{array}{cc}
\displaystyle{
CL_p(\Lambda) \coloneqq \bigoplus_{|a|=p} \Z/2 \cdot a ,} & \displaystyle{ d_J (a) \coloneqq \sum_{|b|=|a|-1} \#_{\Z/2}\M_J (a;b) \cdot b }
\end{array}\]
for every $p\in \Z$ and $a\in \mathcal{R}(\Lambda)$.
Likewise, the quotient complex $\mathcal{F}^2/\mathcal{F}^3$ is naturally isomorphic to
\[ ( CL_*(\Lambda)^{\otimes 2}, d_J\otimes \id_{CL_*}+\id_{CL_*}\otimes d_J)\]
by the isomorphism $CL_*(\Lambda)^{\otimes 2} \to \mathcal{F}^2/\mathcal{F}^3 \colon a_1\otimes a_2 \mapsto a_1a_2 + \mathcal{F}^3$.
Moreover, the quotient complex $\mathcal{F}^1/\mathcal{F}^3$ is naturally isomorphic to
\begin{align}\label{C1-C2}
 \left( CL_*(\Lambda)\oplus CL_*(\Lambda)^{\otimes 2}, \begin{pmatrix}d_J & 0 \\
\delta_J &  d_J\otimes \id_{CL_*}+\id_{CL_*}\otimes d_J
\end{pmatrix} \right) ,
\end{align}
where  $\delta_J \colon CL_*(\Lambda )\to CL_*(\Lambda)^{\otimes 2}[-1]$ is a chain map defined by
\[
\delta_J(a) =  \sum_{|b_1|+|b_2|=|a|-1} \#_{\Z/2}\M_J (a;b_1,b_2) \cdot b_1\otimes b_2 \]
for every $a\in \mathcal{R}(\Lambda)$.

We denote the homology of $(CL_*(\Lambda),d_J)$ by $HL_*(\Lambda)$. Let us call it the \textit{strip Legendrian contact homology} of $\Lambda$.
In addition, we denote the homology of $(CL_*(\Lambda)^{\otimes 2} ,d_J \otimes \id_{CL_*} + \id_{CL_*} \otimes d_J) $ by $HL^2_*(\Lambda)$.
Moreover, $\delta_J$ induces a $\Z/2$-linear map
\[(\delta_J)_*\colon HL_*(\Lambda) \to HL^2_{*-1}(\Lambda). \]

\begin{rem}\label{rem-ind-pos}
The condition in ($\bigstar$) on the degree of Reeb chords corresponds to the \textit{index-positivity} condition in \cite[Definition 3.10]{U} on a contact manifold with a contact form (see also \cite[page 2105]{CO}). The fact here that we can define $d_J$ and prove $d_J\circ d_J=0$ is analogous to the case (i) in \cite[Corollary 3.7]{U}. \cite{Al} studied the strip Legendrian contact homology for a pair of disjoint Legendrian submanifolds $(\Lambda, \widehat{\Lambda})$ when there is no Reeb chord of $\Lambda$ and $\widehat{\Lambda}$ contractible in $\pi_1(Y,\Lambda)$ and $\pi_1(Y,\widehat{\Lambda})$ respectively. See \cite[Definition 1]{Al}.
In addition, using the terminology of \cite[Section 4.1]{EES}, the differential $\partial_J\colon \mathcal{A}_* \to \mathcal{A}_{*-1}$ is \textit{augmented} under the condition ($\bigstar$) and $HL_*(\Lambda)$ is the \textit{linearized homology} of the DGA $(\mathcal{A}_*(\Lambda),\partial_J)$.
\end{rem}

\begin{rem}\label{rem-ind-pos-example}
The condition ($\bigstar$) is a strict condition on $\Lambda$.
Aside from our example to be explained in Section \ref{subsec-LCH-conormal}, we can find the following examples:
\begin{itemize}
\item Let $(Q,g)$ be a Riemannian manifold whose sectional curvature is non-positive.
In its unit cotangent bundle $U^*Q$, consider $\Lambda \coloneqq T^*_q Q \cap U^*Q$ for any $q\in Q$.
Then, there is no Reeb chord of $\Lambda$ which is contractible with end points in $\Lambda$, since $(Q,g)$ has no non-constant contractible geodesic $c \colon [0,T] \to Q$ such that $c(0)=c(T)=q$ by the Cartan-Hadamard theorem.
\item
An index-positivity condition generalized to Morse-Bott cases (see \cite[Definition 2.14, Remark 2.17]{BK}) is satisfied for the real Legendrian submanifold $\mathcal{L}$ in the $(2n+1)$-dimensional $A_k$-type Brieskorn manifold when $n\geq 3$. See \cite[Section 3.5]{BK}. 
\end{itemize}
\end{rem}

The next theorem shows that $HL_*(\Lambda)$ together with $(\delta_J)_*$ is invariant under Legendrian isotopies between two Legendrian submanifolds satisfying ($\bigstar$).

\begin{thm}\label{thm-Leg-invariant}
Let $\Lambda_0$ and $\Lambda_1$ be Legendrian submanifolds of $Y$ satisfying the condition ($\bigstar$). 
For $i=0,1$, we take $J_i \in \mathcal{J}_{\bar{Y}_{l(i)}}$ of Proposition \ref{prop-transverse} for $\Lambda=\Lambda_i$ and $\bar{Y}_l = \bar{Y}_{l(i)}$ whose interior contains $\Lambda_i$.
If $\Lambda_0$ is Legendrian isotopic to $\Lambda_1$ in $Y$, then there exists a quasi-isomorphic chain map
\[ \varphi \colon CL_*(\Lambda_0 ) \to CL_*(\Lambda_1)\]
such that the following diagram commutes up to chain homotopy:
\begin{align}\label{diagram-isotopy}
\begin{split}
\xymatrix{
CL_{*+1}(\Lambda_0) \ar[r]^-{\delta_{J_0}} \ar[d]^-{\varphi} & CL_*(\Lambda_0)^{\otimes 2} \ar[d]^-{\varphi \otimes \varphi } \\
CL_{*+1}(\Lambda_1) \ar[r]^-{\delta_{J_1}} & CL_*(\Lambda_1)^{\otimes 2} .
}
\end{split}
\end{align}
\end{thm}
\begin{proof}
Suppose that there exists a Legendrian isotopy $(\Lambda_t)_{t\in [0,1]}$ from $\Lambda_0$ to $\Lambda_1$. Take $l\in \Z_{\geq 1}$ such that $l\geq \max \{l(0),l(1)\}$ and $\bigcup_{t\in [0,1]} \Lambda_t $ is contained in the interior of $\bar{Y}_l$.
We take an almost complex structure $\tilde{J}$ on $\R\times Y$ which is compatible with $d(e^r \alpha)$ and coincides with $J_0$ (resp. $J_1$) on $\R_{\geq r_0}\times Y$ (resp. $\R_{\leq r_1} \times Y$) for some $r_0,r_1\in \R$ with $r_1<r_0$ and moreover coincides with $J_{\mathrm{fix}}$ outside $\bar{Y}_l$.
Associated to $(\Lambda_t)_{t\in [0,1]}$, we can construct an exact Lagrangian cobordism, say $L$, in $\R\times Y$ from $\Lambda_1$ to $\Lambda_0$.
For the proof, we refer to \cite[Section 3.2]{Cha}. (Although \cite[Section 3.2]{Cha} only deals with the case $\dim Y=3$, its proof, using the embedding $\R\times Y \to T^*Y \colon (r,x) \mapsto e^r\cdot \alpha_x$, can be  generalized naturally to $Y$ of an arbitrary dimension.) For generic $\tilde{J}$, a DGA map
\[\Phi_L \colon \mathcal{A}_*(\Lambda_0) \to \mathcal{A}_*(\Lambda_1) \]
is defined by counting the modulo $2$ number of rigid $\tilde{J}$-holomorphic curves with boundary in $L$. See \cite[Section 3.2.3]{DR}.

We claim that $\Phi_L$ preserves the filtrations on $\mathcal{A}_*(\Lambda_0)$ and $\mathcal{A}_*(\Lambda_1)$ for a similar reason as Proposition \ref{prop-filtration}. It is proved as follows: For any $(a\colon [0,T] \to Y) \in \mathcal{R}(\Lambda)$, consider the moduli space of $\tilde{J}$-holomorphic curves $u \colon D_1\to \R \times Y$ such that $u(\partial D_1) \subset L$ and for some $s_0\in \R$,
\[\tau_{-Ts}\circ u\circ \psi_0(s,\cdot) \to (s_0, a(T\cdot)) \text{ if } s\to \infty\]
in the $C^{\infty}$ topology on $[0,1]$. For a regular $\tilde{J}$, it is a $C^{\infty}$ manifold of dimension $|a|\geq 2$. Therefore, this moduli space does not contribute to $\Phi_L(a)$.

It follows that $\Phi_L$ induces chain maps on the quotient complexes of the following forms:
\begin{align*} \varphi &\colon CL_*(\Lambda_0) \to CL_*(\Lambda_1)  ,\\
 \begin{pmatrix} \varphi & 0 \\ h & \varphi \otimes \varphi \end{pmatrix} &\colon CL_*(\Lambda_0) \oplus CL_*(\Lambda_0)^{\otimes 2} \to CL_*(\Lambda_1) \oplus CL_*(\Lambda_1)^{\otimes 2} .
 \end{align*}
 From the fact that the second one is a chain map, 
\[ h\circ d_{J_0} + (\varphi \otimes \varphi) \circ \delta_{J_0}=  \delta_{J_1} \circ \varphi + (d_{J_1}\otimes \id + \id \otimes d_{J_1}) \circ h . \]
This means that the diagram (\ref{diagram-isotopy}) commutes modulo the chain homotopy $h$.

It remains to show that $\varphi$ is a quasi-isomorphism.
There exists exact Lagrangian cobordisms $L'$ and $L''$ from $\Lambda_0$ to $\Lambda_1$ induced by the isotopy $(\Lambda_{1-t})_{t\in [0,1]}$ from $\Lambda_1$ to $\Lambda_0$ such that:
\begin{itemize}
\item  The concatenation of $L$ and $L'$ (see the proof of \cite[Theorem 3.5]{DR}) is an exact Lagrangian cobordsim which is isotopic to $\R\times \Lambda_1$ by a compactly supported Hamiltonian isotopy.
\item The concatenation of $L''$ and $L$ is an exact Lagrangian cobordsim which is isotopic to $\R\times \Lambda_0$ by a compactly supported Hamiltonian isotopy.
\end{itemize}
They induces a DGA map $\Phi_{L'}, \Phi_{L''} \colon \mathcal{A}_*(\Lambda_1) \to \mathcal{A}_*(\Lambda_0)$ which preserves the filtrations.
In the same way as $\varphi$, we obtain a chain map $\varphi', \varphi''\colon CL_*(\Lambda_1) \to CL_*(\Lambda_0)$ respectively.

From the proof of \cite[Theorem 3.5]{DR} and \cite[Lemma 3.14]{EHK},
$\Phi_{L}\circ \Phi_{L'}$ is chain homotopic to the identity map on $\mathcal{A}_*(\Lambda_1)$. Moreover,
the chain homotopy $\Theta \colon \mathcal{A}_*(\Lambda_1) \to \mathcal{A}_{*+1}(\Lambda_1)$ preserves the filtration. 
The reason is the same as above, that is, the moduli space of $\tilde{J}$-holomorphic curves on $D_1$ does not contribute to $\Theta$.
On the quotient complex, $\Theta$ induces a chain homotopy $\theta \colon CL_*(\Lambda_1)\to CL_{*+1}(\Lambda_1)$ between $\varphi \circ \varphi'$ and $\id_{CL_*(\Lambda_1)}$. Likewise, we can show that $\varphi'' \circ \varphi$ is  chain homotopic to $\id_{CL_*(\Lambda_0)}$. Hence $\varphi$ is a quasi-isomorphism.
\end{proof}


We continue to assume that $\Lambda$ satisfies the condition ($\bigstar$).
Let $e \geq 2$ be an integer such that $|a|\geq e$ for every $a\in \mathcal{R}(\Lambda)$. 
Then, $\mathcal{A}_p(\Lambda)$ in lower degree $p$ is written as
\begin{align*}
\mathcal{A}_0(\Lambda)&= (\Z/2)\cdot 1, \\
\mathcal{A}_p(\Lambda) &= CL_p\oplus (CL^{\otimes 2})_p \text{ if }1\leq p\leq 3e-1, \\
\mathcal{A}_{3e}(\Lambda)&= CL_{3e}\oplus (CL^{\otimes 2})_{3e} \oplus \bigoplus_{|a_1|=|a_2|=|a_3|=e} (\Z/2) \cdot a_1a_2a_3.
\end{align*}
We note that $\partial_J(a_1a_2a_3) =0$ if $|a_1|=|a_2|=|a_3|= e$.
It follows that if $1\leq p\leq 3e-1$, then $\LCH_p(\Lambda)$ is isomorphic as a $\Z/2$-vector space to the $p$-th degree part of the homology of the chain complex (\ref{C1-C2}).
From the short exact sequence of chain complexes
\[ \xymatrix{
0 \ar[r] & CL_*(\Lambda)^{\otimes 2} \ar[r]^-{i} & CL_*(\Lambda) \oplus CL_*(\Lambda)^{\otimes 2} \ar[r]^-{j} & CL_*(\Lambda) \ar[r] & 0,
}\]
where $i(y)=(0,y)$ and $j(x,y)=x$, 
we obtain an exact sequence of $\Z/2$-vector spaces
\begin{align}\label{ex-seq-LCH}
\begin{split}
\xymatrix@R=5pt{
 HL_{p+1}(\Lambda) 
 \ar[r]^-{(\delta_J)_*}   & HL^2_p(\Lambda) \ar[r] & \LCH_{p}(\Lambda) \ar[r] &  HL_{p}(\Lambda)
\ar[r]^-{(\delta_J)_*}   &  HL_{p-1}^2(\Lambda)
}
\end{split}
\end{align}
if $1\leq p\leq 3e-1$. As a consequence, we obtain the following computation.
\begin{prop}\label{prop-LCH-dim}
For $p\in \Z$ with $1\leq p\leq 3e-1$,
\begin{align*}
\dim_{\Z/2} \LCH_p(\Lambda) =  & \dim_{\Z/2} \ker \left( (\delta_J)_*\colon  HL_{p}(\Lambda) \to  HL^2_{p-1}(\Lambda) \right) \\
  &+  \dim_{\Z/2} \coker \left( (\delta_J)_*\colon  HL_{p+1}(\Lambda) \to  HL^2_p(\Lambda) \right) .
\end{align*}
\end{prop}

\section{Coproduct on $H_{*}(K\times K, \Delta_K)$ from string topology}\label{sec-Morse}


\noindent
In this section, the coefficient of every singular homology group is $\Z/2$, unless otherwise specified.

\subsection{Morse homology of $(K\times K,\Delta_K)$}\label{subsubsec-Morse-complex}

Let $K$ be a compact submanifold of $\R^n$ of codimension $d$.

\subsubsection{Definition of Morse chain complex}
We define the Morse homology associated to the $C^{\infty}$ function $E$ of (\ref{energy-fcnl}).
Suppose that every $x \in \mathcal{C}(K)$ is non-degenerate.
We fix $\epsilon_0>0$ such that
$E(x) >\epsilon_0$ for every $x\in \mathcal{C}(K)$.
Let $\mathcal{G}_K$ denote the set of Riemannian metrics on $K\times K$ of class $C^{\infty}$. It is equipped with the $C^{\infty}$ topology.

For any $g\in \mathcal{G}_K$, we define the negative gradient vector field $V_g$ on $K\times K$ with respect to $E$ by $-dE= g(V_g,\cdot)$.
Let $(\varphi^t_g)_{t\in \R}$ denote the flow of $V_g$. For every $x,x'\in \mathcal{C}(K)$, we define
\begin{align*}
\mathcal{W}^s_g(x) & \coloneqq \{ y\in K\times K \mid \lim_{t\to \infty}\varphi^t_g(y)=x \}, \\
\mathcal{W}^u_g(x) & \coloneqq \{ y\in K\times K \mid \lim_{t\to -\infty}\varphi^t_g(y) =x\}, \\
\mathcal{T}_g(x;x') & \coloneqq (\mathcal{W}^u_g(x) \cap \mathcal{W}^s_g(x')) /\R . 
\end{align*}
Here, the group $\R$ acts on $\mathcal{W}^u_g(x) \cap \mathcal{W}^s_g(x')$ by $y\cdot t \coloneqq \varphi^{-t}_g(y)$ for every $t\in \R$.
For generic $g\in \mathcal{G}_K$, the \text{Morse-Smale condition} on $V_g$ is achieved. Namely, $\mathcal{W}^u_g(x) $ intersects $\mathcal{W}^s_g(x')$ transversely for every $x,x'\in \mathcal{C}(K)$. In such case, $\mathcal{T}_g(x;x')$ is a $C^{\infty}$ manifold of dimension $\ind x -\ind x'-1$.
We note that for every $x\in \mathcal{C}(K)$,
\begin{align}\label{stable-contained}
 \mathcal{W}^s_g(x) \subset E^{-1}([\epsilon_0,\infty)) .
 \end{align}

We define $CM_*$ to be the graded $\Z/2$-vector space spanned by $\mathcal{C}(K)$. Here, $x\in CM_{\ind x}$ for every $x\in \mathcal{C}(K)$. We define a $\Z/2$-linear map $d_g \colon CM_*\to CM_{*-1}$ by
\[d_g(x) \coloneqq \sum_{x'\in \mathcal{C}(K) \colon \ind x' = \ind x -1}  \#_{\Z/2} \mathcal{T}_g(x,x') \cdot x'\]
for every $x\in \mathcal{C}(K)$.
It is a standard fact that $d_g \circ d_g =0$ \cite[Theorem 7]{S-Morse}. Therefore, we get a chain complex $(CM_*, d_g)$. 
We denote its homology by $HM_*(g)$.
In addition, we denote the homology of $(CM_*^{\otimes 2} , d_g \otimes \id_{CM_*} + \id_{CM_*}\otimes d_g )$ by $HM^2_*(g)$.

Let us introduce several notations.
For each $x\in \mathcal{C}(K)$, we define two sets 
\begin{align*}
 \overline{\W}^s_g(x) & \coloneqq \W^s_g(x) \sqcup \coprod_{k\geq 1} \coprod_{x_1,\dots x_k \in \mathcal{C}(K)} \W^s_g(x_k) \times \mathcal{T}_g(x_k;x_{k-1})\times \dots \times \mathcal{T}_g(x_2;x_1) \times \mathcal{T}_g(x_1;x)  ,\\
\overline{\W}^u_g(x) & \coloneqq \W^u_g(x) \sqcup \coprod_{k\geq 1} \coprod_{x_1,\dots ,x_k\in \mathcal{C}(K)} \mathcal{T}_g(x;x_1)\times \mathcal{T}_g(x_1;x_2) \times \dots \times \mathcal{T}_g(x_{k-1};x_k) \times \W^u_g(x_k) .
\end{align*}
We refer to \cite[(11.42)]{Ab} for their topology.
We define two maps
\[\begin{array}{cc}
i^s \colon \overline{\W}^s_g(x)  \to K \times K, &  i^u \colon  \overline{\W}^u_g(x) \to K\times K ,
\end{array}
\]
as the continuous extensions of the inclusion maps $\W^s_g(x) \to K\times K$ and $ \W^u_g(x) \to K\times K$. Explicitly, $i^s$ maps
$(y, [y_k],\dots ,[y_1]) \in  \W^s_g(x_k) \times \mathcal{T}_g(x_k;x_{k-1})\times \dots \times \mathcal{T}_g(x_1;x)$ to $y$, and $i^u$ is given similarly.
In addition, for any $x_1,x_2\in \mathcal{C}(K)$, we denote the product of $i^s \colon \overline{\W}^s_g(x_j) \to K\times K$ for $j=1,2$ by
\[ i^s_2 \colon \overline{\W}^s_g(x_1) \times \overline{\W}^s_g(x_2) \to (K\times K) \times (K\times K) .\]
\begin{rem} $\overline{\W}^s_g(x)$ is compact and $i^s(\overline{\W}^s_g(x)) \subset E^{-1}([\epsilon_0,\infty))$.
On the other hand, $\overline{\W}^u_g(x)$ can be non-compact since there can exist a trajectory $(\varphi^t_g(y))_{t\in \R}$ in $\W^u_g(x)$ which converges to a point in $\Delta_K= E^{-1}(0)$ when $t\to \infty$. However,
\[ \overline{\W}^u_g(x)^{\geq \epsilon_0} \coloneqq (E\circ i^u)^{-1}([\epsilon_0,\infty)) \]
is a compact subset of $\overline{\W}^u_g(x)$ which contains $\W^u_g(x)^{\geq \epsilon_0} \coloneqq \W^u_g(x) \cap E^{-1}([\epsilon_0,\infty))$ as an open subset.
\end{rem}


Let us also consider 
\begin{align*}
\overline{\W}^{s,1}_g(x) & \coloneqq \W^s_g(x) \sqcup \coprod_{\ind x' = \ind x +1}   \W^s_g(x') \times \mathcal{T}_g(x',x) , \\
  \overline{\W}^{u,1}_g(x) & \coloneqq \W^u_g(x) \sqcup \coprod_{\ind x' = \ind x-1 } \mathcal{T}_g(x,x') \times  \W^u_g(x') ,
 \end{align*}
 for every $x\in \mathcal{C}(K)$.
$\overline{\W}^{s,1}_g(x) $ is a union of some strata in $\overline{\W}^s_g(x)$ of  codimension $0$ and $1$, and similar for $\overline{\W}^{u,1}_g(x)$.
$\overline{\W}^{s,1}_g(x)$ is an open subset of $\overline{\W}^s_g(x)$ and has a structure of $C^{\infty}$ manifold with the boundary $\coprod_{\ind x' = \ind x+1}   \W^s_g(x') \times \mathcal{T}_g(x',x)$, and $i^s \colon \overline{\W}^{s,1}_g(x) \to K\times K$ is a $C^{\infty}$ map.
Similarly, $\overline{\W}^{u,1}_g(x)$ is an open subset of $\overline{\W}^u_g(x)$ and has a structure of $C^{\infty}$ manifold with the boundary $ \coprod_{\ind x' = \ind x-1} \mathcal{T}_g(x,x') \times  \W^u_g(x')$, and the map $i^u \colon \overline{\W}^{u,1}_g(x) \to K\times K$ is a $C^{\infty}$ map.
For the proof, see, for instance, \cite[Lemma 4.2]{Sc-eq}.


\subsubsection{Isomorphism from singular homology to $HM_*$}\label{subsubsec-singular-Morse}
Let us fix some notations. For $p\in \Z_{\geq 1}$, we define the standard $p$-simplex by $\Delta^p\coloneqq \{(x_0,\dots ,x_p)\in (\R_{\geq 0})^{p+1} \mid \sum_{i=0}^px_i= 1\}$.
For any subset $I\subset \{0,\dots ,p\}$, we define a subcomplex $\Delta^p(I)\coloneqq \{ (x_0,\dots ,x_p) \mid x_i=0 \text{ if }i\in I\}$.
For any topological space $X$, let $(C_*(X),\partial^{\sing})$ be the singular chain complex with coefficients in $\Z/2$. For the pair of spaces $(K\times K, E^{-1}([0,\epsilon_0]))$, let us denote the quotient complex by
\[ (CS_* \coloneqq C_*(K\times K) / C_*(E^{-1}([0,\epsilon_0]))  ,\partial^{\sing} ) . \]
Its homology is the relative singular homology $H_*(K\times K , E^{-1}([0,\epsilon_0]))$.

Let $CS_p(g)$ denote the $\Z/2$-linear subspace of $CS_p$ spanned by all chains represented by singular simplices
\begin{align*}
c \colon \Delta^p\to K\times K 
\end{align*}
satisfying the following condition:
\begin{itemize}
\item[(T)] For any subset $I\subset \{0,\dots ,p\}$, the restriction of $c$ on $\Delta^p(I)$
\[\rest{c}{\Delta^p(I)} \colon \Delta^p(I) \to K\times K \]
is a $C^{\infty}$ map which is transversal to $\W^s_g(x)$ for every $x\in \mathcal{C}(K)$.
\end{itemize}
Then, $(CS_*(g),\partial^{\sing})$ is a subcomplex of $(CS_*,\partial^{\sing})$. Since this condition is satisfied for generic chains, the inclusion map $I_g\colon CS_*(g) \to CS_*$ induces an isomorphism
\[ (I_g)_* \colon H_*(CS_*(g),\partial^{\sing}) \to H_*(K\times K , E^{-1}([0,\epsilon_0])). \]

For any chain $c\colon \Delta^p \to K\times K$ satisfying (T), the fiber product
\[ \Delta^p \ftimes{c}{i^s} \W^s_g(x)\]
is a $C^{\infty}$ manifold of dimension $p-\ind x$ with boundaries and corners.
It has a compactification
\[\Delta^p\ftimes{c}{i^s}\overline{\W}^s_g(x).\]
If $\ind x=p$,  $\Delta^p\ftimes{c}{i^s}\overline{\W}^s_g(x) = \Delta^p \ftimes{c}{i^s} \W^s_g(x)$ and it is a compact $0$-dimensional manifold. If $\ind x= p-1$, $\Delta^p\ftimes{c}{i^s}\overline{\W}^s_g(x)$ is a compact $1$-dimensional manifold with boundary
\begin{align}\label{boundary-Delta-stable}
\begin{split}
& \coprod_{j\in \{0,\dots ,p\}} \Delta^p(\{j\})\ftimes{c}{i^s}\W^s_g(x)\\
\sqcup & \coprod_{\ind x'= p} (\Delta^p \ftimes{c}{i^s} \W^s_g(x')) \times \mathcal{T}_g(x';x)  .
\end{split}
\end{align}
See \cite[Lemma 3.14]{Ab}.

Now we define a $\Z/2$-linear map
$F_g \colon CS_*(g)\to CM_*$ which maps those chains represented by $c\colon \Delta^p \to K \times K$ satisfying (T) to
\[ \sum_{\ind x=p}  \#_{\Z/2} \left( \Delta^p \ftimes{c}{i^s} \W^s_g(x) \right) \cdot x . \]
This is well-defined since $\Delta^p \ftimes{c}{i^s} \W^s_g(x)=\emptyset$ if  $c \in C_*(E^{-1}([0,\epsilon_0]))$.
Since the modulo $2$ number of (\ref{boundary-Delta-stable}) is equal to $0$, we have
\[F_g \circ \partial^{\sing} +d_g\circ F_g =0. \]
Therefore, $F_g$ is a chain map and it induces
\[(F_g)_*\colon H_*(CS_*(g),\partial^{\sing}) \to HM_*(g) .\]
Hereafter, we abbreviate
\[\Psi_g \coloneqq (F_g)_* \circ (I_g)^{-1}_* \colon H_*(K\times K, E^{-1}([0,\epsilon_0])) \to HM_*(g).\]

We substitute $(K\times K, E^{-1}([0,\epsilon_0]))$ by the pair $(K\times K,E^{-1}([0,\epsilon_0]) )^{\times 2} = ((K\times K)^{\times 2} ,U_{\epsilon_0})$, where
\[U_{\epsilon_0} \coloneqq  \left( (K\times K)\times E^{-1}([0,\epsilon_0]) \right) \cup \left( E^{-1}([0,\epsilon_0]) \times (K\times K) \right) , \]
and consider the intersection of singular chains with $\W^s_g(x_1)\times \W^s_g(x_2)$ for every $x_1,x_2\in \mathcal{C}(K)$. Then, in a parallel way as $\Psi_g$, we define
\[ \Psi_g^2 \colon H_*((K \times K)^{\times 2},U_{\epsilon_0}) \to HM^2_*(g).\]
\begin{prop}\label{prop-Morse-singular}
The following hold:
\begin{itemize}
\item $\Psi_g$ and $\Psi_g^2$ are isomorphisms.
\item Let $P$ be a compact $p$-dimensional manifold with boundary and $c\colon P\to K\times K$ be a $C^{\infty}$ map such that $c(\partial P)\subset E^{-1}([0,\epsilon_0])$. For the homology class $c_*([P])\in H_p(K\times K, E^{-1}([0,\epsilon_0]))$, if $c$ is transversal to $\W^s_{g}(x)$ for every $x\in \mathcal{C}(K)$, then $\Psi_g(c_*([P]))\in HM_p(g)$ is represented by
\[ \sum_{\ind x=p}\#_{\Z/2}\left( P\ftimes{c}{i^s}\W^s_g(x) \right) \cdot x.\] 
\item Let $Q$ be a compact $q$-dimensional manifold with boundary and $s\colon Q\to (K\times K)^{\times 2}$ be a $C^{\infty}$ map such that $c(\partial Q) \subset U_{\epsilon_0} $. For the homology class $s_*([Q])\in H_q((K\times K)^{\times 2},U_{\epsilon_0})$, if $s$ is transversal to $\W^s_{g}(x_1)\times \W^s_{g}(x_2)$ for every $x_1,x_2\in \mathcal{C}(K)$, then $\Psi^2_g(s_*([Q]))\in HM^2_q(g)$ is represented by
\[ \sum_{\ind x_1+ \ind x_2=q}\#_{\Z/2}\left( Q\ftimes{s}{i^s_2} (\W^s_g(x_1) \times \W^s_g(x_2)) \right) \cdot x_1\otimes x_2 .\] 
\end{itemize}
\end{prop}
The proof is standard and not the heart of this paper, so we leave it to 
Appendix \ref{subsec-A2}.

\subsection{Coproduct}\label{subsec-coprod}
To simplify the discussion, we assume that $K$ has the codimension $d\geq 2$.
In this subsection, we define $\Z/2$-linear maps $(\delta_{g,g'})_* \colon HM_*(g) \to HM_{*+1-d}^2(g')$ on Morse homology and $\delta_K\colon H_*(K\times K,\Delta_K) \to H_{*+1-d}((K\times K,\Delta_K)^{\times 2})$ on singular homology. These are derived from the same idea in string topology.
We will also check their coincidence.

\subsubsection{Definition of coproduct on Morse chain complex}\label{subsec-coproduct-Morse}

Let us consider two maps:
\[
\ev \colon  (K\times K) \times [0,1] \to \R^n \colon ((q,q'),\tau) \mapsto (1-\tau )\cdot q+\tau \cdot q' \]
and
\[ \sp \colon  \ev^{-1}(K) \to (K\times K)\times (K\times K) \colon ((q,q'),\tau) \mapsto ((q,\ev((q,q'),\tau)), (\ev((q,q'),\tau),q') ) . 
\]
($\ev$ stands for ``evaluation'' and $\sp$ stands for ``splitting''.)
They are derived from an operation in string topology explained in the Remark \ref{rem-string} below.

\begin{figure}
\centering
\begin{overpic}[width=15cm]{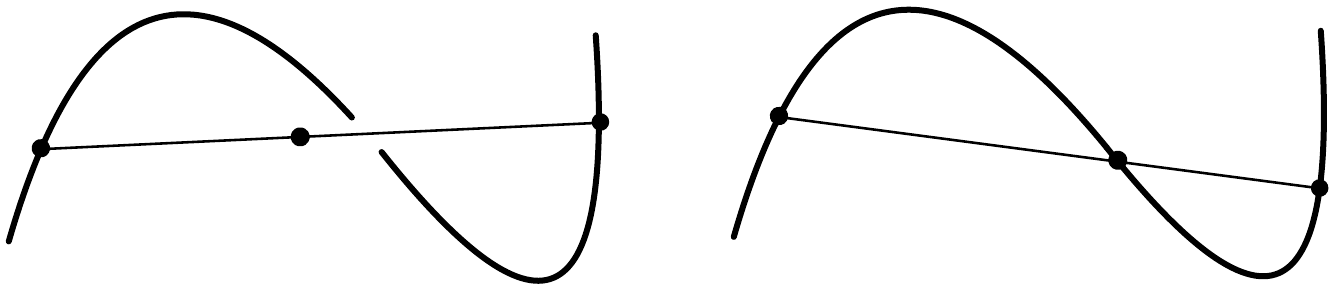}
\put(12,17.5){$K$}
\put(67,18){$K$}
\put(1,11){$q'$}
\put(42,14){$q$}
\put(15,8){$\ev((q,q'),\tau)$}
\put(56.5,14){$q'$}
\put(96,9){$q$}
\put(70,7){$\ev((q,q'),\tau)$}
\end{overpic}
\caption{$\ev((q,q'),\tau)$ is a point on the linear path from $q$ to $q'$ (left-hand side). If $((q,q'),\tau) \in \ev^{-1}(K)$ as in the right-hand side, then we split the linear path at $\tau\in [0,1]$.}\label{figure-evaluation}
\end{figure}

\begin{rem}\label{rem-string}
Let $\mathcal{P}_K\coloneqq \{ \gamma\colon [0,1] \to \R^n \colon C^{\infty} \mid \gamma(0),\gamma(1)\in K \}$. We have two maps:
\begin{align*}
\ev_{\mathcal{P}} &\colon \mathcal{P}_K \times [0,1] \to \R^n \colon (\gamma,\tau) \mapsto \gamma(\tau), \\
\sp_{\mathcal{P}} &\colon (\ev_{\mathcal{P}})^{-1}(K) \to \mathcal{P}_K\times \mathcal{P}_K \colon (\gamma,\tau) \mapsto (\gamma^1,\gamma^2),
\end{align*}
where $\gamma^1(t) \coloneqq \gamma (\tau t)$ and $\gamma^2(t)\coloneqq \gamma (\tau+ (1-\tau)t )$ for every $t\in [0,1]$. The first map is the evaluation map. The second map splits $\gamma\in \mathcal{P}_K$ into two paths $\gamma^1,\gamma^2$ at $\tau \in [0,1]$ where $\gamma$ intersects $K$. 
We identify $K\times K$ with the subspace of $\mathcal{P}_K$ consisting of linear paths by the map
\[K\times K \to \mathcal{P}_K \colon (q,q') \mapsto ([0,1] \to \R^n \colon t\mapsto (1-t)\cdot q+ t \cdot q'). \]
Then, $K\times K$ is a strong deformation retract of $\mathcal{P}_K$.
Moreover, $\ev$ can be seen as the restriction of $\ev_{\mathcal{P}}$ on $(K\times K)\times [0,1]$, and $\sp$ can be seen as the restriction of $\sp_{\mathcal{P}}$ on $\ev^{-1}(K)$. See Figure \ref{figure-evaluation}.
\end{rem}
\begin{rem}
The idea of the splitting operation as above for chains in $\mathcal{P}_K$ can be found in \cite[Section 1]{Su}. A chain-level definition of this operation was used in \cite[Section 5]{CELN} to define the \textit{string homology} for framed knots in oriented $3$-dimensional manifolds. When $K$ is a knot in $\R^3$, \cite[Section 2.3]{An} considered Morse theory similar to this subsection in order to define the cord algebra of $K$, which is isomorphic to the $0$-th degree part of the string homology.
In the same spirit, the \textit{Goresky-Hingston coproduct} $\vee \colon H_*(\mathcal{L}(M),M;\Z) \to H_{*+1-n}((\mathcal{L}(M),M)^{\times 2};\Z)$ was defined by \cite[\S 8.4]{GH} for a closed oriented $n$-dimensional manifold $M$ and its free loop space $\mathcal{L}(M)$. Here, $M$ is identified with the subspace of $\mathcal{L}(M)$ consisting of constant loops.
\end{rem}

Suppose that every $x\in \mathcal{C}(K)$ is non-degenerate.
Choose $g,g' \in \mathcal{G}_K$.
For every $x\in \mathcal{C}(K)$, we consider the fiber product over $K\times K$
\[ \W^u_g(x) \ftimes{i^u}{\pr}\ev^{-1}(K),\]
where $\pr \colon \ev^{-1}(K) \to K\times K \colon (y,\tau) \mapsto y$ is the projection map.
 We define a map on it 
\[ \sp_x \colon \W^u_g(x) \ftimes{i^u}{\pr}\ev^{-1}(K) \to (K\times K)^{\times 2} \colon (y, (y,\tau)) \mapsto \sp(y,\tau).\]
Then, for every $x_1,x_2\in \mathcal{C}(K)$, we define the fiber product over $(K\times K)^{\times 2}$
\[  \mathcal{T}_{g,g'}(x;x_1,x_2)\coloneqq \left( \W^u_g(x) \ftimes{i^u}{\pr}\ev^{-1}(K) \right) \ftimes{\sp_x}{i^s_2} (\W^s_{g'}(x_1) \times \W^s_{g'}(x_2)). \]
Let us denote the restriction of $\ev$ on the open subset $((K\times K)\setminus \Delta_K) \times (0,1)$ by
\[ \mathring{\ev} \coloneqq \rest{\ev}{((K\times K)\setminus \Delta_K) \times (0,1)}.\]
For any $(y,\tau)\in (\Delta_K\times [0,1])\cup ((K\times K)\times \{0,1\})$, $\sp(y,\tau) $ is contained in $(\Delta_K \times (K\times K))\cup ((K\times K)\times \Delta_K)$, which is disjoint from $\W^s_{g'}(x_1)\times \W^s_{g'}(x_2)$. Therefore,
\[ \mathcal{T}_{g,g'}(x;x_1,x_2) = \left( \W^u_g(x) \ftimes{i^u}{\pr}\mathring{\ev}^{-1}(K) \right) \ftimes{\sp_x}{i^s_2} (\W^s_{g'}(x_1) \times \W^s_{g'}(x_2)).  \]

\begin{defi}\label{def-admissible}
We call a triple $(K,g,g')$ \textit{admissible} if $K$, $g$ and $g'$ satisfy the following conditions:
\begin{enumerate}
\item Every $x \in \mathcal{C}(K)$ is non-degenerate.
\item For any $x=(q,q')\in \mathcal{C}(K)$, $(1- \tau)\cdot q + \tau \cdot q' \notin K$ for every $\tau \in \R\setminus \{0,1\}$.
\item $\mathring{\ev} \colon ((K\times K)\setminus \Delta_K) \times (0,1) \to \R^n$ is transversal to $K$.
\item $V_g$ and $V_{g'}$ satisfy the Morse-Smale condition.
\item For every $x\in \mathcal{C}(K)$, $\pr \colon \mathring{\ev}^{-1}(K) \to K\times K$ is transversal to  $\W^u_g(x)$.
\item For every $x, x_1,x_2 \in \mathcal{C}(K)$,
 $\sp_x$ is transversal to  $\W_{g'}^s(x_1)\times \W_{g'}^s(x_2)$.
\end{enumerate}
\end{defi}

Recall the space $\mathcal{O}_K$ in Remark \ref{rem-sp-emb} and the perturbation $K_{\sigma}$ of $K$ for every $\sigma\in \mathcal{O}_K$.

\begin{lem}\label{lem-admissible}
For generic $\sigma \in \mathcal{O}_K$ and generic $g,g'\in \mathcal{G}_{K_{\sigma}}$, $(K_{\sigma},g,g')$ is admissible.
\end{lem}
To prove this, let us prepare a lemma.
In addition, recall the condition that $\codim K \geq 2$ imposed in the beginning of Section \ref{subsec-coprod}. In particular, the dimension $n$ of $\R^n$ is greater than $2$.

We define $\overline{\mathcal{C}}(K) \coloneqq \mathcal{C}(K)/\sim$, where $\sim$ is the equivalence relation generated by $(q,q')\sim (q',q)$ for every $(q,q') \in \mathcal{C}(K)$.
For any $x= (q,q') \in \mathcal{C}(K)$
we let $\overline{x} \in \overline{\mathcal{C}}(K)$ denote the equivalence class of $x\in \mathcal{C}(K)$ and define
\[ l_{x} \coloneqq \{ (1-\tau) \cdot q + \tau \cdot q' \in \R^n \mid \tau \in \R\}. \]
\begin{lem}\label{lem-disjoint-line}
Let $\mathcal{O}'_K$ be a subset of $\mathcal{O}_K$ consisting of $\sigma \in \mathcal{O}_K$ such that:
\begin{itemize}
\item Every $x\in \mathcal{C}(K_{\sigma})$ is non-degenerate.
\item For any  $x, y\in \mathcal{C}(K_{\sigma})$, $l_{x} \cap  l_{y}= \emptyset $ if $\overline{x}\neq \overline{y}$.
\end{itemize}
Then, $\mathcal{O}'_K$ is an open dense subset of $\mathcal{O}_K$.
\end{lem}
\begin{proof}
Let $\mathcal{O}'' \coloneqq \{\sigma \in \mathcal{O}_K \mid \text{every }x\in \mathcal{C}(K) \text{ is non-degenerate} \}$. It is an open dense subset of $ \mathcal{O}_K$.
Fix $\sigma_0 \in \mathcal{O}''$ and order the set $\mathcal{C}(K_{\sigma_0})$ to be $\mathcal{C}(K_{\sigma_0}) = \{x_1,\dots ,x_{2N}\}$ such that $\overline{x_i} = \overline{x_{i+N}}$ in $\overline{\mathcal{C}}(K)$ for $i=1,\dots ,N$.
There exists a neighborhood $U_0\subset \mathcal{O}''$ of $\sigma_0$ such that  for any $\sigma \in U_0$,
\[\mathcal{C}(K_{\sigma}) = \{x_i(\sigma) \mid i=1,\dots ,2N \},\]
where $x_i(\sigma)$  varies  continuously on $\sigma \in U_0$ and $x_i(\sigma_0)=x_i$  for $i=1,\dots ,2N$, and $\overline{x_i(\sigma)} = \overline{x_{i+N}(\sigma)}$ for $i=1,\dots ,N$.
Let us write $x_i(\sigma) =(q_i(\sigma),q'_i(\sigma))\in K\times K$.

If $\sigma_0$ satisfies the second condition, any $\sigma \in U_0$ sufficiently close to $\sigma_0$ also satisfies the second condition. Hence, $\mathcal{O}'_K$ is open in $\mathcal{O}''$.
It remains to show that $\mathcal{O}'_K$ is dense in $\mathcal{O}''$.

Let $\sigma_0$ be an arbitrary element of $\mathcal{O}''$.
We may assume that $N\geq 2$.
Consider
\[x_1=(q_1(\sigma_0),q'_1(\sigma_0)),\ x_2=(q_2(\sigma_0),q'_2(\sigma_0)) \in \mathcal{C}(K)\]
and the intersection $l_{x_1} \cap l_{x_2}$.
Since $\overline{x_1} \neq \overline{x_2}$, either $q_1(\sigma_0)$ or $q'_1(\sigma_0)$ are not in $\{q_2(\sigma_0),q'_2(\sigma_0)\}$.
Let us consider the case where $q_1(\sigma_0) \notin \{q_2(\sigma_0),q'_2(\sigma_0)\}$ (the remaining case is parallel). 
We consider a deformation from $\sigma_0$ to $\sigma_1\in U_0$ such that
$K_{\sigma_1}$ agrees with $K_{\sigma_0}$ outside a neighborhood of $q_1(\sigma_0)$ disjoint from $\{ q'_1(\sigma_0), q_2(\sigma_0) ,  q'_2(\sigma_0)  \}$.
Then,
$q_2(\sigma_0) = q_2(\sigma_1)$ and $q'_2(\sigma_0) = q'_2(\sigma_1)$, and
in particular, $l_{x_2}= l_{x_2(\sigma_1)}$.
For any line $l$ in $\R^n$ which is sufficiently close to $l_{x_1}$ and perpendicular to $K_{\sigma_0}$ at a point near $q'_0(\sigma_0)$, we can take $\sigma_1$ so that $l_{x_1(\sigma_1)} = l$.
See Figure \ref{figure-binormal}. 
Noting that $n\geq 3$, we can choose $\sigma_1$ arbitrarily close to $\sigma_0$ such that $l_{x_1(\sigma_1)} \cap l_{x_2}= \emptyset$.

\begin{figure}
\centering
\begin{overpic}[height=3cm]{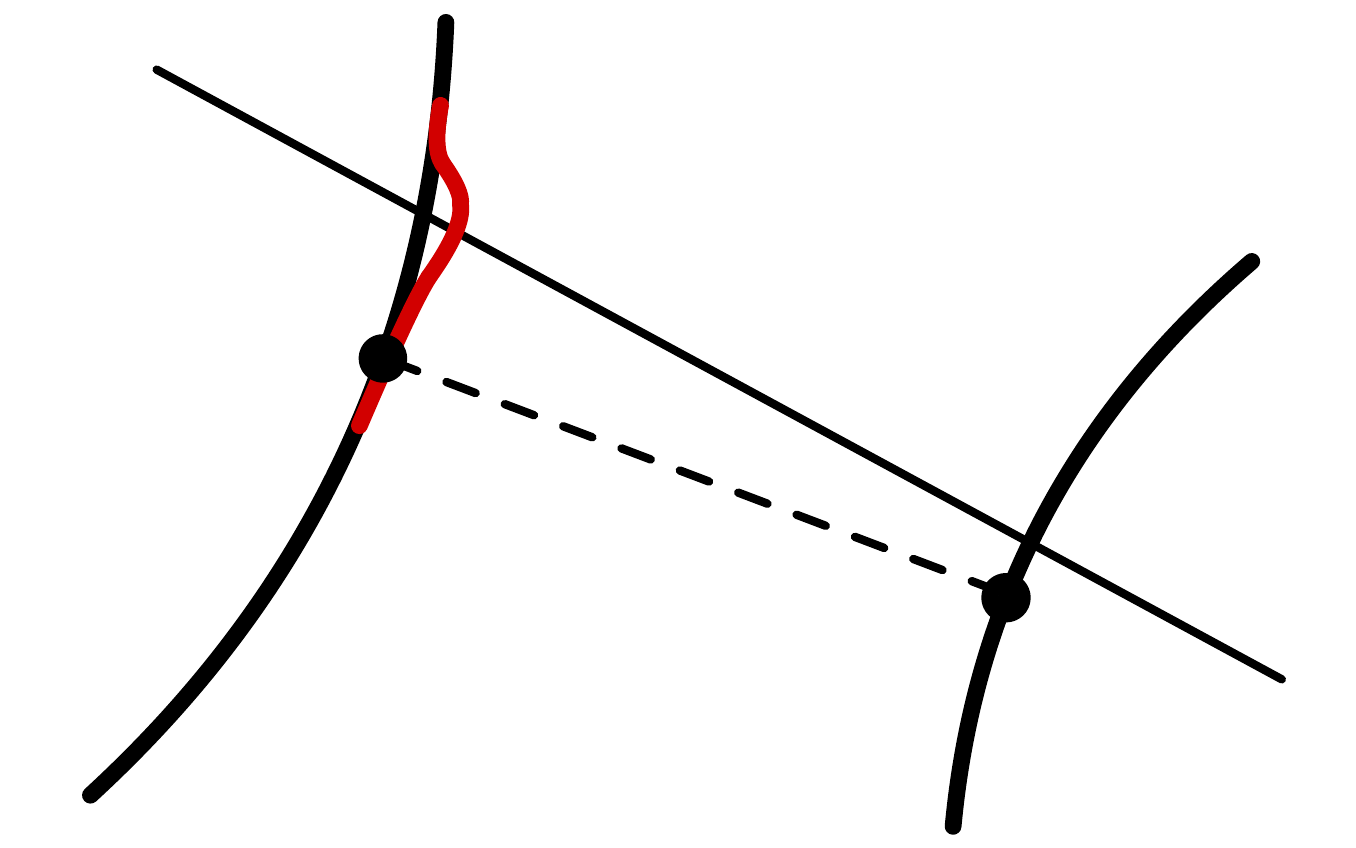}
\put(5.5,35){$q_1(\sigma_0)$}
\put(51,13.5){$q_1'(\sigma_0)$}
\put(98,10){$l$}
\put(13.5,4){$K_{\sigma_0}$}
\put(34,49){$\color{red} K_{\sigma_1}$}
\end{overpic}
\caption{$l$ is a line in $\R^n$ which is close to $l_{x_1(\sigma_0)}$ and perpendicular to $K_{\sigma_0}$ at a point near $q'_1(\sigma_0)$.
$K_{\sigma_1}$ is obtained from $K_{\sigma_0}$ by a perturbation in a neighborhood of $q_1(\sigma_0)$ and perpendicular to $l$ in this neighborhood.
If $\sigma_1$ is sufficiently close to $\sigma_0$, $l = l_{x_1(\sigma_1)}$.}\label{figure-binormal}
\end{figure}

We take $\sigma_1,\sigma_2,\dots ,\sigma_{N-1}\in U_0$ inductively.
For $\sigma_{k-1}$ with $1\leq k\leq N-2$, suppose that
\[l_{x_i(\sigma_{k-1})}\cap l_{x_j(\sigma_{k-1})}=\emptyset  \text{ for any } i,j \in \{1,\dots ,k \}\text{ with }i\neq j.\]
Then, by a deformation similar to the above one from $\sigma_0$ to $\sigma_1$, there exists $\sigma_k\in U_0$ such that:
\begin{itemize}
\item $x_i(\sigma_k)$ is so close to $x_i(\sigma_{k-1})$ for every $i\in\{1,\dots ,k\}$ that 
\[l_{x_i(\sigma_{k})}\cap l_{x_j(\sigma_{k})}=\emptyset  \text{ for any } i,j\in \{1,\dots ,k\} \text{ with }i\neq j.\]
\item $x_{k+1}(\sigma_k) = x_{k+1}(\sigma_{k-1})$ and thus, $ l_{x_{k+1}(\sigma_k)} =  l_{x_{k+1}(\sigma_{k-1})}$.
\item For every $i\in \{1,\dots ,k\}$, $ l_{x_i(\sigma_k)} \cap l_{x_{k+1}(\sigma_k)} = \emptyset$.
\end{itemize}
When $k=N-1$, $\sigma_{N-1}$ belongs to $\mathcal{O}'_K$.
Since $\sigma_1,\sigma_2,\dots ,\sigma_{N-1}$ can be chosen arbitrarily close to $\sigma_0$, this shows that $\mathcal{O}'_K$ is dense in $\mathcal{O}''$.
\end{proof}

Let us set $c(K) \coloneqq \bigcup_{(q,q')\in\mathcal{C}(K)} \{ q , q' \} \subset K$.

\begin{proof}[Proof of Lemma \ref{lem-admissible}]
We use the open dense subset $\mathcal{O}'_K$ of Lemma \ref{lem-disjoint-line}.
Suppose that $\sigma \in \mathcal{O}'_K$. Then, the first condition of Definition \ref{def-admissible} is satisfied for the submanifold $K_{\sigma}$.
Moreover, for any $x=(q,q')\in \mathcal{C}(K_{\sigma})$, $l_x \cap c(K_{\sigma}) = \{q,q'\}$. Therefore, the set
\[l_x\setminus \{q,q'\} = \{ (1-\tau) \cdot q + \tau \cdot q' \in \R^n \mid \tau \in \R\setminus \{0,1\}\}\]
can intersect $K_{\sigma}$ only outside $c(K_{\sigma})$.
Noting that $\codim K \geq 2$, by a generic perturbation of $\sigma$ which fixes a neighborhood of $c(K_{\sigma})$, the second condition of Definition \ref{def-admissible} is achieved.
It is also standard to show that for generic $\sigma \in \mathcal{O}'_K$, the third condition of Definition \ref{def-admissible} are satisfied for  $K_{\sigma}$.

If $K$ satisfies the first three conditions of Definition \ref{def-admissible}, then from the second condition of  Definition \ref{def-admissible}, the image of $\pr \colon \ev^{-1}(K) \to K\times K$ is disjoint from $\mathcal{C}(K)$ and moreover, the image of the map $\sp \colon \mathring{\ev}^{-1}(K) \to (K\times K)\times (K\times K)$ is disjoint from $(\mathcal{C}(K)\times (K\times K) ) \cup ( (K\times K) \times \mathcal{C}(K) )$.
In this case, by generic perturbations of $g,g'\in \mathcal{G}_K$ outside $\mathcal{C}(K)$, we can achieve the latter three conditions of Definition \ref{def-admissible}.
\end{proof}


If $(K,g,g')$ is admissible, $ \W^u_g(x) \ftimes{i^u}{\pr} \mathring{\ev}^{-1}(K) $ is a $C^{\infty}$ manifold of dimension
\[(\dim \W_g^u(x) +1)-\codim K = \ind x +1-d .\]
Therefore, the dimension of $\mathcal{T}_{g,g'}(x;x_1,x_2)$ is computed by
\[\ind x + 1 -d - (\ind x_1 + \ind x_2 ).\]

In addition,
for $\epsilon_0>0$ fixed in Subsection \ref{subsubsec-Morse-complex},
we choose $0<\tau_0<\frac{1}{2}$ such that
\[ \sup_{(q,q')\in K\times K} \frac{1}{2}\tau_0^2|q-q'|^2<\epsilon_0.
\]
Take  $(y, (y,\tau)) \in \W^u_g(x) \ftimes{i^u}{\pr} \ev^{-1}(K)$ arbitrarily. Suppose that either $E(y)<\epsilon_0$ or $\tau \in [0,\tau_0]\cup [1-\tau_0,1]$ holds. Then, $\sp_x(y,(y,\tau))= \sp(y,\tau)$ is contained in the subset
$ U_{\epsilon_0} = \{(y_1,y_2) \mid \min \{ E(y_1), E(y_2)\}\leq \epsilon_0 \}$. 
By (\ref{stable-contained}), $\W^s_{g'}(x_1)\times \W^s_{g'}(x_2)$ is disjoint from this subset.
Therefore, if we define $\ev_{[\tau_0,1-\tau_0]}$ to be the restriction of $\ev$ on $(K\times K)\times [\tau_0,1-\tau_0]$, then
\begin{align}\label{T-VV'-another}
 \mathcal{T}_{g,g'}(x;x_1,x_2) =  \left( \W^u_g(x)^{\geq \epsilon_0} \ftimes{i^u}{\pr} (\ev_{[\tau_0,1-\tau_0]})^{-1}(K) \right) \ftimes{\sp_x}{i^s_2} (\W^s_{g'}(x_1) \times \W^s_{g'}(x_2)) .
 \end{align}
We note that $(\ev_{[\tau_0,1-\tau_0]})^{-1}(K)$ is compact.
This description of $\mathcal{T}_{g,g'}(x;x_1,x_2)$ leads us to the following compactness theorem.
\begin{lem}\label{lem-boundary-T}
If $\ind x_1+\ind x_2 = \ind x -d+1$, then $\mathcal{T}_{g,g'}(x;x_1,x_2)$ is a compact $0$-dimensional manifold. If $\ind x_1+\ind x_2 = \ind x -d$, $\mathcal{T}_{g,g'}(x;x_1,x_2)$ is compactified to a compact $1$-dimensional manifold whose boundary is
\begin{align}\label{boundary-T}
\begin{split}
&\coprod_{\ind x'=\ind x-1} \mathcal{T}_g(x,x')\times \mathcal{T}_{g,g'}(x';x_1,x_2) \\
\sqcup & \coprod_{\ind x'_1=\ind x_1+1} \mathcal{T}_{g,g'}(x;x'_1,x_2)\times \mathcal{T}_{g'}(x'_1,x_1) \\
\sqcup & \coprod_{\ind x'_2=\ind x_2+1}\mathcal{T}_{g,g'}(x;x_1,x'_2)\times \mathcal{T}_{g'}(x'_2,x_2),
\end{split}
\end{align}
\end{lem}
\begin{proof}
From the equation (\ref{T-VV'-another}), the compactification of $\mathcal{T}_{g,g'}(x;x_1,x_2)$ is given by
\begin{align}\label{compact-T}
\begin{split}
 \overline{\mathcal{T}}_{g,g'}(x;x_1,x_2) 
\coloneqq & \left( \overline{\W}^u_g(x) \ftimes{i^u}{\pr} \ev^{-1}(K) \right) \ftimes{\sp_x}{i^s_2} (\overline{\W}^s_{g'}(x_1) \times \overline{\W}^s_{g'}(x_2)) \\
= & \left( \overline{\W}^u_g(x)^{\geq \epsilon_0} \ftimes{i^u}{\pr} (\ev_{[\tau_0,1-\tau_0]})^{-1}(K) \right) \ftimes{\sp_x}{i^s_2} (\overline{\W}^s_{g'}(x_1) \times \overline{\W}^s_{g'}(x_2)) .
 \end{split}
\end{align}
Its compactness follows from the second line.
Let us observe the case where the dimension of $\mathcal{T}_{g,g'}(x;x_1,x_2)$ is $0$ or $1$.
If $\ind x_1+\ind x_2 =\ind x -d +1$, $\overline{\mathcal{T}}_{g,g'}(x;x_1,x_2)  =\mathcal{T}_{g,g'}(x;x_1,x_2)$.
Next, suppose that $\ind x_1+\ind x_2 = \ind x -d$.
By using the $C^{\infty}$ manifolds  $\overline{\W}^{u,1}_g(x)$, $\overline{\W}^{s,1}_{g'}(x_1)$ and $\overline{\W}^{s,1}_{g'}(x_2)$ with boundary,
$ \overline{\mathcal{T}}_{g,g'}(x;x_1,x_2)  $ is given as the union of its open subsets
\begin{align*}
  \left( \overline{\W}^{u,1}_g(x) \ftimes{i^u}{\pr} \ev^{-1}(K) \right) \ftimes{\sp_x}{i^s_2} \left( \W^s_{g'}(x_1) \times \W^s_{g'}(x_2)\right) ,\\
 \left( \W^{u}_g(x) \ftimes{i^u}{\pr} \ev^{-1}(K) \right) \ftimes{\sp_x}{i^s_2} \left( \overline{\W}^{s,1}_{g'}(x_1) \times \W^s_{g'}(x_2)\right) , \\
  \left( \W^{u}_g(x) \ftimes{i^u}{\pr} \ev^{-1}(K) \right) \ftimes{\sp_x}{i^s_2} \left( \W^s_{g'}(x_1) \times \overline{\W}^{s,1}_{g'}(x_2)\right) ,
\end{align*}
each of which is a $1$-dimensional manifold with boundary.
Here, note that $i^u$ and $i^s_2$ are of class $C^{\infty}$.
Moreover, the union of their boundaries agrees with (\ref{boundary-T}).
%
\end{proof}

Now we define the coproduct operation on the Morse chain complex.
\begin{defi}
We define a $\Z/2$-linear map $\delta_{g,g'}\colon CM_{*+d-1} \to CM_*^{\otimes 2}$ by
\[ \delta_{g,g'}(x) \coloneqq \sum_{x_1,x_2 \in \mathcal{C}(K) \colon \ind x_1 + \ind x_2= \ind x -d+1} \#_{\Z/2} \mathcal{T}_{g,g'}(x;x_1,x_2) \cdot x_1\otimes x_2 \]
for every $x\in \mathcal{C}(K)$.
\end{defi}
Lemma \ref{lem-boundary-T} implies that for every $x\in \mathcal{C}(K)$,
\[ \delta_{g,g'}\circ d_g(x) + (d_{g'}\otimes \id_{CM_*} +  \id_{CM_*}\otimes d_{g'})\circ \delta_{g,g'}(x) =0.\]
Therefore, $\delta_{g,g'}$ is a chain map from $CM_{*+d-1}$ to $CM_*^{\otimes 2}$. This induces a map on the homology
\[ (\delta_{g,g'})_* \colon HM_{*+d-1}(g) \to HM^2_*(g') . \]

\subsubsection{Definition of coproduct $\delta_K$ on singular homology}\label{subsubsec-coproduct-singular}

In this subsection, we consider $K$ in a general position and do not require the conditions of Definition \ref{def-admissible}. We denote $I\coloneqq [0,1]$.

Consider the normal bundle $(TK)^{\perp} \to K$ of $K$ in $\R^n$ and its open disk bundle of radius $r>0$. If $r>0$ is sufficiently small, the disk bundle is identified with a tubular neighborhood of $K$ defined by
\[ N\coloneqq \{ q+v \in \R^n \mid (q,v)\in (T K)^{\perp},\ |v|<r \}  \]
via the map $\{(q,v)\in (TK)^{\perp} \mid |v|<r\} \to N\colon (q,v) \mapsto q+v$.
Let us also set
\[ \textstyle{ N'\coloneqq \{ q+v \in \R^n \mid (q,v)\in (T K)^{\perp},\ |v|<\frac{r}{2} \}  } .\]
Then, we have the projection $\pi_K\colon N\to K\colon q+v \mapsto q$ and the cohomology class
\[\eta \in H^d(N,N\setminus N')\]
corresponding to the Thom class of $(TK)^{\perp}$.

On the open subset $\ev^{-1}(N)$ of $(K\times K)\times I$, we define
\[\begin{array}{rcl}
\widehat{\sp} \colon \ev^{-1}(N) & \to & (K\times K)\times (K\times K) \\
((q,q'),\tau) & \mapsto & ((q, \pi_{K}\circ \ev((q,q'),\tau)) , (\pi_K\circ \ev ((q,q'),\tau) , q') ) .
\end{array}
\]

\begin{defi}\label{def-delta-K}
We define
\[ \delta_K \colon H_*(K\times K,\Delta_K ) \to H_{*+1-d}((K\times K,\Delta_K)^{\times 2})\]
to be the composite map
\[\xymatrix@C=40pt@R=10pt{
H_*(K\times K,\Delta_K ) \ar[r]^-{  \times [I]} & H_{*+1}((K\times K) \times I, Z_K) & \\
\ar[r]^-{\ev^*(\eta) \cap  } & H_{*+1-d} (\ev^{-1}(N) , Z_K) 
\ar[r]^-{\widehat{\sp}_*} & H_{*+1-d} ((K\times K, \Delta_K)^{\times 2}) .
}\]
Here, $Z_K\coloneqq (\Delta_K\times I) \cup ((K\times K)\times \partial I)$.
Namely, for every $x\in H_*(K\times K,\Delta_K)$,
\[ \delta_K (x) = \widehat{\sp}_*(\ev^*(\eta) \cap (x\times [I]) ) .\]
\end{defi}
\begin{rem}
Precisely, the second map $\ev^*(\eta)\cap$ is defined as follow: Let us abbreviate
$X =(K\times K)\times I$, $V= \ev^{-1}(\R^n \setminus N')$ and $W = \ev^{-1}(N\setminus N')$. Since $\ev(Z_K)= K$, $Z_K$ is disjoint from the closed subset $V$. Then, using the cohomology class $\ev^*(\eta) \in H^d(\ev^{-1}(N), W)$, we obtain a composite map
\[\xymatrix@C=30pt{
H_p(X,Z_K) \ar[r] & H_p(X, Z_K \sqcup V) \cong H_p(\ev^{-1}(N) , Z_K\sqcup W) \ar[r]^-{\ev^*(\eta) \cap } & H_{p+1-d}(\ev^{-1}(N),Z_K) ,
} \]
where the left map is induced by the inclusion map for pairs and the middle isomorphism is defined by the excision of $\ev^{-1}(\R^n \setminus N)$.
\end{rem}
It is easy to check that $\delta_K$ is independent of the choice of the radius $r>0$ of the disk bundle of $(TK)^{\perp}$.

\subsubsection{On invariance of $\delta_K$}\label{subsubsec-invariance}

Let $K_+$ and $K_-$ be compact submanifolds of $\R^n$ of codimension $d$.
Suppose that there exists a properly embedded submanifold $L$ of $\R^{n+1}=\R^n\times \R$ satisfying:
\begin{itemize}
\item $L\cap (\R^n\times \R_{\geq 1}) = K_+\times \R_{\geq 1}$ and $L\cap (\R^n\times \R_{\leq -1}) = K_-\times \R_{\leq -1}$.
\item $L\cap (\R^n\times [-1,1])$ is compact.
\end{itemize}
We denote the inclusion maps by $i_+\colon K_+ \to L\colon q\mapsto (q,1)$ and $i_-\colon K_-\to L\colon q\mapsto (q,-1)$.

For sufficiently small $r>0$, we define a tuple $(\tilde{N},\tilde{N}',\pi_L, \tilde{\eta})$ by
\begin{align*}
\tilde{N} & \coloneqq \{q+v \in \R^{n+1} \mid (q,v)\in (TL)^{\perp},\ |v|<r\},  \\
\tilde{N}' & \coloneqq \textstyle{ \{q+v \in \R^{n+1} \mid (q,v)\in (TL)^{\perp},\ |v|<\frac{r}{2} \} , }
\end{align*}
$\pi_L\colon \tilde{N} \to L\colon q+v\mapsto q$ and $\tilde{\eta} \in H^d(\tilde{N}, \tilde{N} \setminus \tilde{N}')$ corresponding to the Thom class of $(TL)^{\perp}$.
We also define
\begin{align*}
\tilde{\ev} &\colon (L\times L) \times I \to \R^{n+1} ((q,q'),\tau) \mapsto (1-\tau)\cdot q + \tau \cdot q, \\
\tilde{\sp} & \colon (\tilde{\ev})^{-1}(\tilde{N}) \to (L\times L)^{\times 2} \colon ((q,q'),\tau) \mapsto ( (q, \pi_L\circ \tilde{\ev} ((q,q'),\tau)), (\pi_L\circ \tilde{\ev} ((q,q'),\tau) , q'))  .
\end{align*}
Note that $N_{+}\coloneqq \tilde{N}\cap (\R^n\times \{ 1\})$ is a tubular neighborhood of $K_{+}$ in $\R^n \cong \R^n\times \{1\}$ and $(i'_{+})^*(\tilde{\eta})$ corresponds to the Thom class of $(TK_{+})^{\perp}$, where $i'_{+}\colon N_{+} \to \tilde{N}$ is the inclusion map. Likewise, we define $N_-\coloneqq \tilde{N}\cap (\R^n\times \{-1\})$ and $i'_-\colon N_-\to \tilde{N}$ to be the inclusion map.

Replacing $(K,Z_K, N , \eta )$ and $(\ev,\widehat{\sp})$ in Definition \ref{def-delta-K} by
\[(L , Z_L\coloneqq (\Delta_L \times I) \cup ( (L\times L)\times \partial I ) , \tilde{N} , \tilde{\eta} ) \text{ and } (\tilde{\ev},\tilde{\sp})\]
respectively, we define a $\Z/2$-linear map of degree $(1-d)$
\[\delta_L\colon H_*(L\times L,\Delta_L) \to H_{*+1-d}((L\times L,\Delta_L)^{\times 2}) \colon x \mapsto \tilde{\sp}_*( \tilde{\ev}^*(\tilde{\eta}) \cap (x\times [I]) ).\]

We claim that the following diagram commutes: 
\[\xymatrix{
H_*(K_+\times K_+,\Delta_{K_+}) \ar[r]^-{\delta_{K_+}} \ar[d]^-{(i_+^{\times 2})_*} & H_{*+1-d}((K_+\times K_+,\Delta_{K_+})^{\times 2}) \ar[d]^-{(i_+^{\times 4})_*} \\
H_*(L\times L,\Delta_L) \ar[r]^-{\delta_L} & H_{*+1-d}((L\times L,\Delta_L)^{\times 2}).
}\]
Indeed, for the map $\ev \colon (K_+\times K_+ ) \times I \to \R^n$, we have
\begin{align*}
& \tilde{\ev}\circ (i_+^{\times 2}\times \id_I) = i'_+ \circ \ev \colon\ev^{-1}(N_+) \to \tilde{N} ,\\
& \tilde{\sp}\circ  (i_+^{\times 2}\times \id_I) = i_+^{\times 4} \circ \widehat{\sp}  \colon \ev^{-1}(N_+) \to (L\times L)^{\times 2},
\end{align*}
and from the naturality of the cap product,
\begin{align*}
\left( \delta_L \circ (i_+^{\times 2})_*\right) (x) & =\tilde{\sp}_* (\tilde{\ev}^*(\tilde{\eta}) \cap (i_+^{\times 2}\times \id_I)_* (x \times [I])) \\
& = \left( \tilde{\sp}_* \circ (i_+^{\times 2}\times \id_I)_* \right) \left(  (i_+^{\times 2}\times \id_I)^*(\tilde{\ev}^* (\tilde{\eta})) \cap (x \times [I]) \right) \\
&= \left( (i_+^{\times 4})_*\circ \widehat{\sp}_* \right) \left( \ev^*((i'_+)^* \tilde{\eta}) \cap (x \times [I])\right) .
\end{align*}
for every $x \in H_*(K_+\times K_+ ,\Delta_{K_+})$. Since $(i'_+)^*(\tilde{\eta})$ corresponds to the Thom class of $(TK_{+})^{\perp}$, the commutativity is proved. The same assertion holds for $K_-$ and $i_-\colon K_- \to L$.
From this observation, we immediately obtain the following result.
\begin{thm}\label{thm-cobordism}
Assume that $i_+$ is a homotopy equivalence and let $p_+\colon L\to K_+$ be a homotopy inverse of $i_+$. Then the following diagram commutes:
\begin{align}\label{diagram-delta-K-pm}
\begin{split}
\xymatrix{
H_*(K_-\times K_-,\Delta_{K_-}) \ar[r]^-{\delta_{K_-}} \ar[d]^-{((p_+\circ i_-)^{\times 2})_*} & H_{*+1-d}((K_-\times K_-,\Delta_{K_-})^{\times 2}) \ar[d]^-{((p_+\circ i_-)^{\times 4})_*} \\
H_*(K_+\times K_+,\Delta_{K_+}) \ar[r]^-{\delta_{K_+}} & H_{*+1-d}((K_+\times K_+,\Delta_{K_+})^{\times 2}) .
}
\end{split}
\end{align}
\end{thm}

\begin{ex}[Invariance under $C^{\infty}$ isotopy]\label{ex-isotopy}
If $K_+$ is $C^{\infty}$ isotopic to $K_-$ in $\R^n$, then there exists a isotopy $(f_s\colon K_- \to \R^n)_{s\in [0,1]}$ such that $f_0(x) =x$ for every $x\in K_-$ and $f_1(K_-)=K_+$. Take a $C^{\infty}$ function $\sigma \colon \R\to [0,1]$ such that $\sigma(z) = 0 $ if $z\leq -1$ and $\sigma(z)=1$ if $z\geq 1$. Let us consider
$L\coloneqq \{(f_{\sigma(z)}(q),z) \mid q\in K_- , \ z\in \R \}$.
Then $p_+\colon L\to K_+\colon (f_{\sigma(z)}(q),z)\mapsto f_1(q)$ is a homotopy inverse of $i_+$. Since $p_+\circ i_- =f_1\colon K_-\to K_+$, the diagram (\ref{diagram-delta-K-pm}) shows the $C^{\infty}$ isotopy invariance of the coproduct.
\end{ex}


\subsection{Coincidence of two definitions}


Let $P$ be a $p$-dimensional compact $C^{\infty}$ manifold without boundary  and $c \colon P \to K\times K\colon z\mapsto (q_z,q'_z)$ be a $C^{\infty}$ map.
Consider a $C^{\infty}$ function
\[h\colon P\times [0,1] \to \R\colon (z,\tau) \mapsto \tau(1-\tau)\cdot E(c(z))\]
and take $\epsilon_1>0$ such that
\begin{align}\label{ineq-h}
 \textstyle{\min \{\frac{1}{2}\tau^2|q_z-q'_z|^2 , \frac{1}{2}(1-\tau)^2|q_z-q'_z|^2\} <\epsilon_0 }
 \end{align}
for every $(z,\tau)\in h^{-1}([0,\epsilon_1])$.

Suppose that $(K,g,g')$ is admissible.
We assume that $c$ is transversal to $\W^s_g(x)$ for every $x\in \mathcal{C}(K)$ and moreover transversal to the map $\pr \colon \mathring{\ev}^{-1}(K)\to K\times K$. Then, the map
\[ \ev_c\colon (P\times I) \setminus h^{-1}(0) =\{ z\in P \mid c(z) \notin \Delta_K \} \times (0,1)\to \R^n\colon (z,\tau) \mapsto \mathring{\ev} ( c(z),\tau)\]
is transversal to $K$. We assume that $\epsilon_1>0$ is a regular value of $h\colon \ev^{-1}_c(K)\to \R$. Then, $\ev_c^{-1}(K)^{\geq \epsilon_1} \coloneqq \ev^{-1}_c(K) \cap h^{-1}([\epsilon_1,\infty))$ is a compact $C^{\infty}$ manifold with boundary $\ev_c^{-1}(K)\cap h^{-1}(\epsilon_1)$.
On the manifold $\ev_c^{-1}(K)^{\geq \epsilon_1}$, we define a $C^{\infty}$ map
\[ \sp_c \colon \ev_c^{-1}(K)^{\geq \epsilon_1} \to (K\times K)^{\times 2} \colon (z,\tau) \mapsto \sp (c (z) , \tau) .\]
We can check from (\ref{ineq-h}) that $\sp_c$ maps the boundary of $\ev_c^{-1}(K)^{\geq \epsilon_1}$ to $U_{\epsilon_0}$. Hence, we obtain a homology class $(\sp_c)_*([\ev_c^{-1}(K)^{\geq \epsilon_1}]) \in H_{p+1-d}((K\times K)^{\times 2} , U_{\epsilon_0})$.
\begin{lem}\label{lem-geometric-intersection}
By generic perturbations of $g$ and $g'$ in $\mathcal{G}_K$ and $c$ if necessary, we have an equation
\begin{align}\label{chain-equal}
 (\delta_{g,g'})_* \circ \Psi_g(c_*([P])) = \Psi^2_{g'} ((\sp_c)_* ([\ev_c^{-1}(K)^{\geq \epsilon_1}])) 
\end{align}
in $HM^2_{p+1-d}(g')$.
\end{lem}

\begin{proof}

By the second assertion of Proposition \ref{prop-Morse-singular} and the definition of $\delta_{g,g'}$, the left-hand side of (\ref{chain-equal}) is represented by the chain
\begin{align}\label{chain-FIf}
\begin{split}
 &\delta_{g,g'} \left( \sum_{\ind x=p} \#_{\Z/2}\left( P\ftimes{c}{i^s} \W^s_g(x) \right) x \right) \\
 = &\sum_{ \ind x_1 +\ind x_2=p+1-d}  \left(\sum_{\ind x=p}  \#_{\Z/2}  \left( (P \ftimes{c}{i^s} \W^s_g(x)) \times \mathcal{T}_{g,g'}(x;x_1,x_2) \right) \right) x_1\otimes x_2  .
 \end{split}
\end{align}

We give another chain which is homologous to (\ref{chain-FIf}).
We consider the flow $\{\varphi^{\rho}_g\}_{\rho\geq 0}$ of $V_g$ and the fiber product over $K\times K$
\[P \ftimes{\varphi^{\rho}_g \circ c}{\pr} \ev^{-1}(K) \]
for every $\rho \in [0,\infty)$.
On this fiber product, we define a map
\[ \sp_{c,\rho}\colon P \ftimes{\varphi^{\rho}_g \circ c}{\pr} \ev^{-1}(K) \to (K \times K)^{\times 2} \colon (z, (y,\tau)) \mapsto \sp(y,\tau) .\]
Then, for any $x_1,x_2\in \mathcal{C}(K)$,
we define a set
\begin{align*}
\mathcal{S}_{c}(x_1,x_2) \coloneqq 
\coprod_{\rho \in [0,\infty)} \{\rho\}\times \left( (P \ftimes{\varphi^{\rho}_g \circ c}{\pr} \ev^{-1}(K) ) \ftimes{\sp_{c,\rho}}{i^s_2} (\W^s_{g'}(x_1) \times \W^s_{g'}(x_2) ) \right).
\end{align*}
It has the topology as a subspace of $[0,\infty) \times \left( P\times \ev^{-1}(K) \times (\W^s_{g'}(x_1) \times \W^s_{g'}(x_2) ) \right) $.
For generic $g$ and $g'$ and $c$, $\mathcal{S}_{c}(x_1,x_2)$ is transversely cut out and it is a $C^{\infty}$ manifold of dimension
\[p+ 2-d- (\ind x_1+\ind x_2). \]

We claim that $\mathcal{S}_{c}(x_1,x_2)$ admits a compactification as follows when its dimension is $0$ or $1$:
If $\ind x_1+\ind x_2 = p-d+2$, $\mathcal{S}_{c}(x_1,x_2)$ is a compact $0$-dimensional manifold.
If $\ind x_1+\ind x_2=p-d+1$, $\mathcal{S}_{c}(x_1,x_2)$ is compactified to a compact $1$-dimensional manifold whose boundary is 
\begin{align}\label{boundary-S}
\begin{split}
& \mathcal{S}_{c}(x_1,x_2) \cap \{ \rho=0\} \\
\sqcup & \coprod_{\ind x'_1=\ind x_1 +1}  \mathcal{S}_c(x'_1,x_2) \times \mathcal{T}_{g'}(x'_1,x_1) \\
\sqcup &  \coprod_{\ind x'_2=\ind x_2 +1} \mathcal{S}_c(x_1,x'_2) \times \mathcal{T}_{g'}(x'_2,x_2) \\
\sqcup & \coprod_{\ind x= p} ( P \ftimes{c}{i^s} \W^s_g(x)) \times \mathcal{T}_{g,g'}(x;x_1,x_2).
\end{split}
\end{align}
The proof is similar to Lemma \ref{lem-boundary-T}, except that
we have the boundary component in the fourth line of (\ref{boundary-S}) which consists of the limit of sequences $(\rho_n, Y_n)_{n=1,2,\dots }$ in $\mathcal{S}_{c}(x_1,x_2)$ such that $\lim_{n\to \infty} \rho_n = \infty$.
%
Moreover,
\begin{align}\label{S0}
\begin{split}
 \mathcal{S}_{c}(x_1,x_2) \cap \{ \rho=0\} & =  (P \ftimes{c}{\pr} \ev^{-1}(K) ) \ftimes{\sp_{c,0}}{i^s_2} (\W^s_{g'}(x_1) \times \W^s_{g'}(x_2) ) \\
 & \cong \ev_{c}^{-1}(K)^{\geq \epsilon_1} \ftimes{\sp_{c}}{i^s_2} (\W^s_{g'}(x_1) \times \W^s_{g'}(x_2) ) .
 \end{split}
 \end{align}
 
By (\ref{boundary-S}) and (\ref{S0}), the chain
\begin{align}\label{chain-FIf-homologous}
  \sum_{\ind x_1+\ind x_2=p-d+1}  \#_{\Z/2} \left( \ev_{c}^{-1}(K)^{\geq \epsilon_1} \ftimes{\sp_{c}}{i^s_2} (\W^s_{g'}(x_1) \times \W^s_{g'}(x_2) ) \right) \cdot x_1\otimes x_2 
 \end{align}
is equal to
\[ (\ref{chain-FIf}) + (d_{g'}\otimes \id_{CM_*} + \id_{CM_*} \otimes d_{g'})  \left( \sum_{\ind x'_1+\ind x'_2=p-d+2}
\#_{\Z/2} \mathcal{S}_c(x'_1,x'_2) \cdot  x'_1\otimes x'_2 \right),
\]
which is homologous to (\ref{chain-FIf}). Moreover, by the third assertion of Proposition \ref{prop-Morse-singular}, the chain (\ref{chain-FIf-homologous}) represents
\[  \Psi^2_{g'} ( (\sp_{c})_*([\ev_{c}^{-1}(K)^{\geq \epsilon_1}]) ) \in HM^2_{p+1-d}(g') \]
This proves the equality (\ref{chain-equal}) in $HM_{p+1-d}^{2}(g')$.
\end{proof}



We may take $\epsilon_1>0$ to be a regular value of $h$ so that $G\coloneqq h^{-1}([\epsilon_1,\infty))$ is a compact manifold with smooth boundary $ \partial G = h^{-1}(\epsilon_1)$. Then, its fundamental class $[G] \in H_{p+1}(G,\partial G)$ is defined.
In addition, let
\[j \colon (K\times K,\Delta_K) \to (K\times K, E^{-1}([0,\epsilon_0]))\]
denote the inclusion map for pairs. This is a homotopy equivalence. The product $j^{\times 2} \colon  (K\times K,\Delta_K)^{\times 2} \to ((K\times K)^{\times 2} ,U_{\epsilon_0}) $ is also a homotopy equivalence.
\begin{lem}\label{lem-homology-geometry}
By the map
\[ (j^{\times 2})_* \circ \delta_K\colon H_*(K\times K,\Delta_K) \to H_*((K\times K)^{\times 2}, U_{\epsilon_0} ),\]
$c_*([P])$ is mapped to $(\sp_c)_* ([\ev_c^{-1}(K)^{\geq \epsilon_1} ] )$.
\end{lem}

\begin{proof}
Since $\ev_c$ is transversal to $K$, if we take $r>0$ sufficiently small, then $\ev^{-1}(N)\cap G$ with boundary $\ev^{-1}(N)\cap \partial G $ is a tubular neighborhood of $\ev_c^{-1}(K)^{\geq \epsilon_1} = \ev^{-1}(K)\cap G$ with boundary $\ev_c^{-1}(K)\cap \partial G$. Moreover, the cohomology class
\[(\rest{\ev_c}{G})^*(\eta) \in H^d( \ev_c^{-1}(N)\cap G , \ev^{-1}_c(N\setminus N') \cap \partial G ) \]
corresponds to the Thom class of the normal bundle of $\ev_c^{-1}(K)^{\geq \epsilon_1}$. Let us define
\[\widehat{\sp}_c \colon \ev_c^{-1}(N) \cap G\to (K\times K)^{\times 2}\colon (z,\tau) \mapsto \widehat{\sp}(c(z),\tau).\]

We claim that $\left( (j^{\times 2})_*\circ \delta_K\right) (c_*([P]))$ is equal to the image of $[P]\in H_p(P)$ by the following composite map:
\[\xymatrix@C=70pt@R=10pt{
H_p(P ) \ar[r]^-{  \times [I]} & H_{p+1}(P\times I , h^{-1}([0,\epsilon_1])) \cong H_{p+1}(G,\partial G)  \\
\ar[r]^-{(\rest{\ev_c}{G})^*(\eta)  \cap  } & H_{p+1-d} (\ev^{-1}_c(N) \cap G , \ev^{-1}_c(N) \cap \partial G) \\
\ar[r]^-{(\widehat{\sp}_c)_*} & H_{p+1-d} ((K\times K)^{\times 2},U_{\epsilon_0}) ,
}\]
where the isomorphism $H_{p+1}(P\times I , h^{-1}([0,\epsilon_1])) \cong H_{p+1}(G,\partial G)$ is given by the excision of $h^{-1}([0,\epsilon_1))$.
Indeed, $[P\times I]$ is mapped to $[G]$ by this excision and
\begin{align*}
\left( (j^{\times 2})_*\circ \delta_K\right) (c_*([P]))  & = (j^{\times 2} \circ \widehat{\sp})_*(\ev^*(\eta) \cap (c \times \id_{I})_*([P\times I])) \\
 &= \left( (j^{\times 2} \circ \widehat{\sp})_* \circ (c\times \id_I)_* \right) \left( (c\times \id_I)^*(\ev^*(\eta)) \cap [P\times I]\right) \\
 &=(\widehat{\sp}_c)_* ( (\rest{\ev_c}{G})^*(\eta) \cap [G] ).
 \end{align*}
Here, the second equality follows from the naturality of cap product.
By Poincar\'{e} duality,
\[
 (\rest{\ev_c}{G})^{*} (\eta) \cap [G] =  [\ev_c^{-1}(K)^{\geq \epsilon_1}]  \in H_{p+1-d} (\ev^{-1}_c(N) \cap G , \ev^{-1}_c(N) \cap \partial G ).
\]
Since $\widehat{\sp}_c$ coincides with $\sp_c$ on $\ev_c^{-1}(K)^{\geq \epsilon_1}$,
\[(\widehat{\sp}_c)_* ( (\rest{\ev_c}{G})^*(\eta) \cap [G] ) = (\sp_c)_* ([\ev_c^{-1}(K)^{\geq \epsilon_1} ] ).\]
This proves the lemma.
\end{proof}

We use Lemma \ref{lem-homology-geometry} to show the the coincidence of the two definitions of coproduct.
\begin{thm}\label{thm-morse-singular}
By generic perturbations of $g$ and $g'$ in $\mathcal{G}_K$ if necessary, the following diagram commutes:
\[\xymatrix{
H_*(K\times K,\Delta_K) \ar[r]^-{\delta_K} \ar[d]^-{\Psi_g \circ j_*}& H_{*+1-d}((K\times K,\Delta_K)^{\times 2}) \ar[d]^-{\Psi_{g'}^2 \circ (j^{\times 2})_*}\\
HM_*(g) \ar[r]^-{(\delta_{g,g'})_*} & HM^2_{*+1-d}(g')
}\]
\end{thm}
\begin{proof}
Since the inclusion map $\Delta_K\to K\times K$ has a retraction $K\times K\to \Delta_K\colon (q,q')\mapsto (q,q)$,  
we have an exact sequence for the pair $(K\times K,\Delta_K)$
\begin{align}\label{ex-seq-diagonal}
 0 \to H_*(\Delta_K) \to H_*(K\times K) \to H_*(K\times K,\Delta_K) \to 0. 
\end{align}
In particular, $H_*(K\times K) \to H_*(K\times K,\Delta_K)$ is surjective. By \cite[Th\'{e}or\`{e}me III.2]{T}, every homology class of $H_p(K\times K)$ with coefficients in $\Z/2$ is represented by $c_*([P])$ for some closed $p$-dimensional manifold $P$ and a $C^{\infty}$ map $c \colon P\to K\times K$. We choose a finite set $\mathcal{B}$ of pairs $(P,c)$ as above and an admissible triple $(K,g,g')$ such that
$\{ c_*([P]) \mid (P,c)\in \mathcal{B}\}$ is a basis of $H_*(K\times K,\Delta_K)$ and the equality (\ref{chain-equal}) of Lemma \ref{lem-geometric-intersection} holds for every $(P,c)\in \mathcal{B}$. Then
\begin{align*}
\left( (\delta_{g,g'})_* \circ \Psi_g \circ j_*\right) (c_*([P])) =&  \left( (\delta_{g,g'})_* \circ \Psi_g\right) (c_*[P]) \\
=& \left( \Psi_g^2\circ (\sp_c)_*\right) ([\ev_c^{-1}(K)^{\geq \epsilon_1}] )\\
=&\left(  \Psi_g^2 \circ (j^{\times 2})_*\circ \delta_K\right) (c_*[P]).
\end{align*}
The second and the third equality follows from Lemma \ref{lem-geometric-intersection} and Lemma \ref{lem-homology-geometry} respectively. 
\end{proof}

Let us give a simple example such that $K$ is disconnected and $\delta_K$ is a non-zero map.
\begin{ex}
Consider the Hopf link in $\R^3$
\[K \coloneqq \{ (\cos \theta,\sin \theta,0) \in \R^3 \mid \theta\in [0,2\pi]\} \sqcup \{(  0, 1+\cos \varphi , \sin \varphi) \mid \varphi \in [0,2\pi]\}.\]
Let $c\colon S^1 = \R/ 2\pi \Z \to K\times K \colon \theta \mapsto ((1,0,0),  (\cos \theta,\sin \theta,0))$.
Note that $\{\theta \mid c(\theta) \notin \Delta_K\} = (0,2\pi)$.
Using the notation as before, $\ev_c$ is given by
\[ \ev_c  \colon (0,2\pi) \times (0,1) \to \R^n \colon (\theta,\tau) \mapsto (1-\tau + \tau \cos \theta, \tau \sin \theta ,0 ) \]
and $(\ev_c)^{-1}(K)^{\geq \epsilon_1} = \{(\pi,\frac{1}{2})\}$ for sufficiently small $\epsilon_1>0$.
See Figure \ref{figure-hopf}. Then,
\[\sp_c \colon (\ev_c)^{-1}(K)^{\geq \epsilon_1} \to (K\times K)^{\times 2} \colon \textstyle{(\pi,\frac{1}{2})} \mapsto ((1,0,0), p_0), (p_0,(-1,0,0)) ,\]
where $p_0=(0,0,0)\in K$. It is clear that $\delta_K(c_*([S^1])) = [\{\sp_c(\pi, \frac{1}{2})\}]\in H_0 ((K\times K,\Delta_K)^{\times 2})$ is non-zero since $(\pm 1,0,0)$ and $p_0$ are in different components of $K$.
\end{ex}

\begin{figure}
\centering
\begin{overpic}[height=5cm]{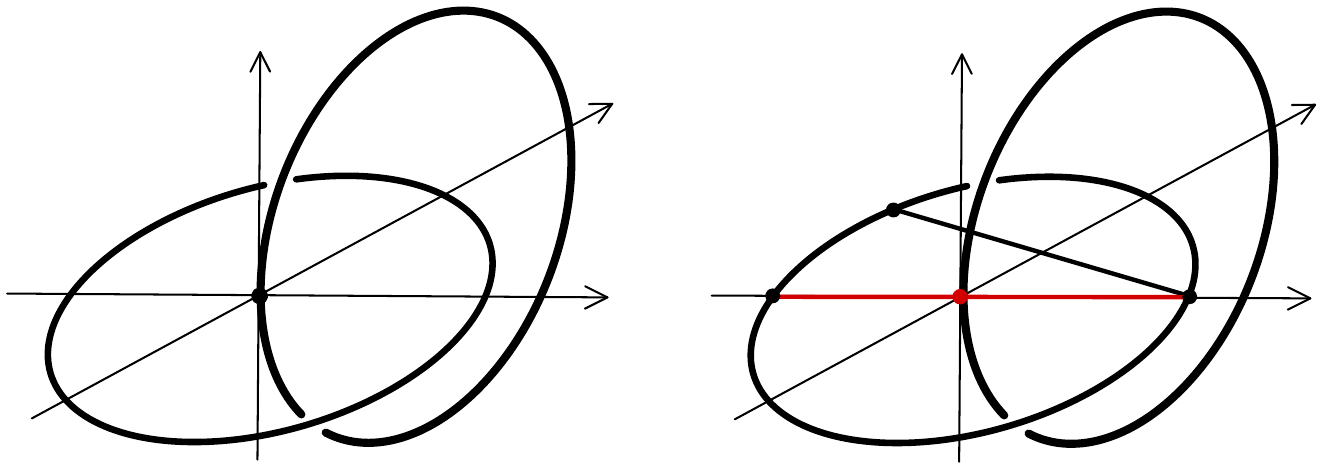}
\put(0,30){$\R^3$}
\put(38.5,16){$K$}
\put(44,10){$x_1$}
\put(45,24.5){$x_2$}
\put(15.5,30){$x_3$}
\put(16 ,14){$p_0$}
\put(53, 21){$(\cos \theta, \sin \theta ,0)$}
\put(48.5,14.5){$(-1,0,0)$}
\put(69,14){$p_0$}
\end{overpic}
\caption{$K$ is the Hopf link in $\R^3$ as in the left-hand side. Here, $(x_1,x_2,x_3)$ is the coordinate of $\R^3$ and $p_0=(0,0,0)$.
For every $\theta \in (0,2\pi)$, $\ev_c(\theta, \cdot) \colon (0,1) \to \R^3$ is extended to a linear path on $[0,1]$ from $(1,0,0)$ to $(\cos \theta,\sin \theta, 0)$, which is drawn by a black segment in the right-hand side.
The path $\ev_c(\theta, \cdot)$ intersects $K$ if and only if $\theta = \pi$ (see the red segment), and $\ev_c(\pi , \tau) \in K$ if and only if $\tau = \frac{1}{2}$.
}\label{figure-hopf}
\end{figure}

Lastly, let us show a variant of Lemma \ref{lem-homology-geometry} for $c\colon P\to K\times K\colon z\mapsto (q_z,q'_z)$,  which is used in Section \ref{sec-application} for concrete computations.
Here, we let $K$ be in general position and do not require the conditions in Definition \ref{def-admissible}.
Consider a $C^{\infty}$ section $\colon K\to (TK)^{\perp} \colon q\mapsto (q,\sigma(q))$ such that $|\sigma(q)|<\frac{r}{2}$ for every $q\in K$. Then, we obtain a submanifold
$K_{\sigma} = \{ q+\sigma(q) \in \R^n \mid q\in K\}$ in $N'$ as in Remark \ref{rem-sp-emb}. 
We define a map on $\ev_c^{-1}(K_{\sigma})$
\[ \sp_{c,\sigma}\colon \ev_c^{-1}(K_{\sigma}) \to (K\times K)^{\times 2} \colon (z,\tau) \mapsto \widehat{\sp}(c(z),\tau) . \]
\begin{prop}\label{prop-geometric-shift}
Assume that $\ev_c$ is transversal to $K_{\sigma}$ and $\{q_z ,q'_z\} \cap K_{\sigma}=\emptyset$ for every $z\in P$. Then
$\ev_c^{-1}(K_{\sigma})$ is a compact $C^{\infty}$ manifold without boundary. Moreover,
\[\delta_K(c_*([P])) = (\sp_{c,\sigma})_* ([\ev_c^{-1}(K_{\sigma})]) .\]
\end{prop} 
\begin{proof}
From the assumption on $K_{\sigma}$ and $c$, $\ev(c(z),\tau) $ is not contained in $K_{\sigma}$ for every $(z, \tau)\in h^{-1}(0) = (c^{-1}(\Delta_K)\times I) \cup (P\times \partial I )$. Therefore, if $\epsilon_1>0$ is sufficiently small, then the manifold $\ev_c^{-1}(K_{\sigma})$, whose boundary is empty, is contained in the compact set $h^{-1}([\epsilon_1,\infty))$. Hence $\ev_c^{-1}(K_{\sigma})$ is a compact $C^{\infty}$ manifold without boundary.

In order to prove the second assertion, 
we take $0< r_1 < \frac{r}{2}- \sup_{q\in K} |\sigma(q)|$ and define
\[N_{\sigma} \coloneqq \{ q + \sigma(q) + v \mid (q,v) \in (TK)^{\perp}, |v|<r_1\}\]
contained in $N'$.
Then, $N_{\sigma}$ is a tubular neighborhood of $K_{\sigma}$.
If $r_1>0$ is sufficiently small, then $\{q_z,q'_z\} \cap N_{\sigma} = \emptyset$ for every $z\in P$ and $\ev_c^{-1}(N_{\sigma})$ is a tubular neighborhood of $\ev_c^{-1}(K_{\sigma})$ in $P\times (0,1)$.
Let us also set
\[ \textstyle{ N'_{\sigma} \coloneqq \{ q + \sigma(q) + v \mid (q,v) \in (TK)^{\perp}, |v|<\frac{r_1}{2}\}. }\] 
We have the inclusion maps
$ (N,N\setminus N') \to (N,N\setminus N'_{\sigma}) \leftarrow (N_{\sigma},N_{\sigma}\setminus N'_{\sigma}) $
and both of them are homotopy equivalences. Via the induced isomorphisms on cohomology, $\eta \in H^d(N,N\setminus N')$ corresponds to the cohomology class
\[\eta_{\sigma} \in H^d(N_{\sigma},N_{\sigma}\setminus N'_{\sigma}),\]
which coincides with the Thom class of the normal bundle of $K_{\sigma}$.
Then, the following diagram commutes:
\[\xymatrix@C=45pt{
H_{p+1} (P\times I , P\times \partial I) \ar[r]^-{(\ev_c)^*(\eta) \cap} \ar[rd]_-{ (\ev_c)^*(\eta_{\sigma}) \cap}& H_{p+1-d}( \ev_c^{-1}(N) , P\times \partial I ) \ar[r]^-{(\widehat{\sp}_c)_*} & H_{p+1-d} ((K\times K,\Delta_K)^{\times 2}) \\
 & H_{p+1-d}( \ev_c^{-1}(N_{\sigma}) ,\emptyset ) , \ar[ur]_-{(\widehat{\sp}_{c,\sigma})_*} \ar[u] &
}\]
where $\widehat{\sp}_c \coloneqq \rest{\widehat{\sp} \circ (c\times \id_I)}{ \ev_c^{-1}(N) }$, $\widehat{\sp}_{c,\sigma} \coloneqq \rest{\widehat{\sp} \circ (c\times \id_I)}{ \ev_c^{-1}(N_{\sigma}) }$ and the middle map is induced by the inclusion map $(\ev_c^{-1}(N_{\sigma}),\emptyset) \to (\ev_c^{-1}(N),P\times \partial I)$.
For the map $(\ev_c)_*(\eta_{\sigma}) \cap$, we note that $(P\times \partial I) \cap (\ev_c)^{-1}( N_{\sigma}) =\emptyset$ since $\{q_z,q'_z\} \cap N_{\sigma} = \emptyset$ for every $z\in P$.

As in the proof of Lemma \ref{lem-homology-geometry},
we can check by the naturality of cap product that
\[\delta_K(c_*([P])) =(\widehat{\sp}_c)_*((\ev_c)^*(\eta) \cap [P\times I]) .\]
By Poincar\'{e} duality, $(\ev_c)^*(\eta_{\sigma}) \cap [P\times I] = [\ev_c^{-1}(K_{\sigma}) ]$. Since $\sp_{c,\sigma} = \rest{\widehat{\sp}_{c,\sigma}}{ \ev_c^{-1}(K_{\sigma})}$, the second assertion follows from the diagram.
\end{proof}

\section{Computation of $HL_*(\Lambda_K)$ and the coproduct $(\delta_J)_*$}\label{sec-Floer-Morse}

\subsection{Strip Legendrian contact homology of unit conormal bundles}\label{subsec-LCH-conormal}

Consider the diffeomorphism
\[U^*\R^n \to \R \times T^*S^{n-1} \colon (q,p) \mapsto (\la q, p\ra , (-p,q-\la q, p \ra \cdot p)).\]
Then, the contact form $\alpha= \rest{\sum_{i=1}^n p_idq_i}{U^*\R^n}$ coincides with $dz+\lambda_{\mathrm{can}}$ on $\R\times T^*S^{n-1}$, where $z$ is the coordinate of $\R$ and $\lambda_{\mathrm{can}}$ is the canonical Liouville form on $T^*S^{n-1}$.
The contact manifold $(U^*\R^n,\alpha)$ for $n\geq 3$ satisfies the condition imposed in Subsection \ref{subsec-LCH} for $(P,\lambda_P)=(T^*S^{n-1},\lambda_{\mathrm{can}})$.

Let $K$ be a compact connected submanifold of $\R^n$ of codimension $d$. After a perturbation by generic $\sigma \in \mathcal{O}_K$ if necessary, we may suppose that every $x\in \mathcal{C}(K)$ is non-degenerate (Remark \ref{rem-sp-emb}).
Its unit conormal bundle $\Lambda_K$ is a compact connected Legendrian submanifold of $U^*\R^n$. This satisfies $\mu_{\Lambda_K}=0$ since the Maslov class of the conormal bundle $L_K$ vanishes. 
By  Proposition \ref{prop-correspond},  every Reeb chord $a\in \mathcal{R}(\Lambda_K)$ is non-degenerate.
Moreover, if $d\geq 4$, then
\[  |a| = \ind (x_a) + (d-2) \geq 2 \]
for every $a\in \mathcal{R}(\Lambda_K)$.
Therefore, if $d\geq 4$, $\Lambda_K$ satisfies the condition ($\bigstar$) introduced in Subsection \ref{subsubsec-degree-bound}.

Consider the diffeomorphism
\begin{align}\label{diffeo-symp-cot}
F\colon \R\times U^*\R^n \to \{(q,p)\in T^*\R^n \mid p\neq 0\} \colon (r,(q,p)) \mapsto (q,e^r\cdot p).
\end{align}
We fix an almost complex structure $J_{\mathrm{fix}}$ on $\R\times U^*\R^n$ determined by $ (F_*J_{\mathrm{fix}}) (\partial_{q_i}) = - |p| \partial_{p_i}$ on $\{(q,p)\in T^*\R^n \mid p\neq 0\}$.
We define  $U^*B^n_l \coloneqq \{(q,p) \in U^*\R^n \mid |q| \leq l\}$ for every $l\in \Z_{\geq 1}$. It is easy to check the condition which we required on $J_{\mathrm{fix}}$ in Subsection \ref{subsubsec-J-hol} for $(\bar{Y}_l)_{l=1,2,\dots } = (U^*B^n_l)_{l=1,2,\dots }$. If we choose $R_K>0$ to be $K \subset \{q\in \R^n \mid |q|<R_K \}$,
then we can define the space $\mathcal{J}_{U^*B^n_l}$ for every $l \geq R_K$ as in Subsection \ref{subsubsec-J-hol}.

By Subsection \ref{subsubsec-DGA}, the DGA $(\mathcal{A}_*(\Lambda_K),\partial_J)$ is defined for $J\in \mathcal{J}_{U^*B^n_l}$ of Proposition \ref{prop-transverse}.
As we have discussed in Subsection \ref{subsubsec-degree-bound}, since $\Lambda_K$ satisfies ($\bigstar$), the chain complex
$ (CL_*(\Lambda_K), d_J)$
and the chain map
\[ \delta_J \colon CL_{*+1}(\Lambda_K)\to CL_*(\Lambda_K)^{\otimes 2} \]
are defined.
Throughout this section, we assume that $\codim K=d\geq 4$ and every $x\in \mathcal{C}(K)$ is non-degenerate, and abbreviate the chain complex $(CL_*(\Lambda_K),d_J)$ by
$ (CL_*,d_J)$.

\subsection{Pseudo-holomorphic curves with switching Lagrangian boundary conditions}\label{subsec-switching}


\subsubsection{Preliminaries}

First, let us fix an almost complex structure on $T^*\R^n$ explicitly.
We choose  a $C^{\infty}$ function $\rho \colon[0,\infty)\to (0,\infty)$ such that
$\rho(s)=1$ if $s \leq \frac{1}{2} $ and $\rho(s)=s$ if $s\geq 1$.
Define an almost complex structure $J_{\rho}$ on $T^*\R^n$ by
\[ (J_{\rho})_{(q,p)} = \begin{pmatrix}
0 & \rho(|p|)^{-1} I_n \\
- \rho(|p|) I_n & 0 
\end{pmatrix}\]
on $T_{(q,p)} (T^*\R^n )= \R^n \oplus \R^n$, where $I_n$ is the identity matrix of rank $n$.
We choose an almost complex structure $J'$ on $T^*\R^n$ satisfying the following conditions:
\begin{itemize}
\item $d\lambda (\cdot , J' \cdot)$ is a Riemannian metric on $T^*\R^n$.
\item  $J'$ coincides with $J_{\rho}$ on $\{(q,p)\in \R^n \mid |p|\leq \frac{1}{2} \text{ or }|q|\geq R_K\}$. 
\item
Via the diffeomorphism (\ref{diffeo-symp-cot}),
 $F^*J'$ preserves $\xi$ and $(F^*J)(\partial_r) = \frac{e^r}{ \rho(e^r)}  R_{\alpha}$. Moreover, there exists $J\in \mathcal{J}_{U^*B^n_l}$ such that $F^*J'=J$ on $U^*\R^n \times \R_{\geq 0}$.
\end{itemize}
It is straightforward to check that $J_{\rho}$ satisfies these conditions.

Let us prepare several notations about Riemann surfaces introduced in Subsection \ref{subsubsec-J-hol}. We rewrite $\partial_1 D_1 \coloneqq \partial D_1$. For $m\geq 1$, we write the boundary components of $D_{2m+1}$ so that
\[\partial D_{2m+1} = \coprod_{k=1}^{2m+1} \partial_k D_{2m+1} ,\]
where the closure of $\partial_k D_{2m+1}$ in $\partial D$ has $\{p_{k-1}, p_k \}$ as the boundary (When $k=2m+1$, we set $p_{2m+1} \coloneqq p_0$).
In addition, we fix a complex structure $j_0$ on $D_3$ associated to the unique element $\kappa_0 \in \mathcal{C}_3$. Then, there exists a biholomorphic map
\begin{align}\label{bihol}
  \psi \colon D_3 \to ([0,\infty) \times [0,1] )\setminus \{(0,0), (0,1)\} 
\end{align}
characterized by the following condition:
\begin{itemize}
\item We identify $ [0,\infty) \times [0,1]$ with a subset of $\C P^1=\C \sqcup \{z_{\infty}\}$ via the map $ [0,\infty) \times [0,1] \to \C \colon (s,t) \mapsto s+\sqrt{-1}t$. Then, $\psi$ is a biholomorphic map which extends continuously to a map $\hat{\psi} \colon D\to \C P^1$ such that $\hat{\psi}(p_0)=z_{\infty}$, $\hat{\psi}(p_1)=0$ and  $\hat{\psi}(p_2)=\sqrt{-1}$.
\end{itemize}

\subsubsection{Definition of the moduli space $\mathcal{M}_m(a)$}\label{subsubsec-def-Mm}
We identify the image of the zero section of $T^*\R^n$ with $\R^n$.
For any $(a\colon [0,T]\to U^*\R^n)\in \mathcal{R}(\Lambda_K)$ and $m\in \Z_{\geq 0}$, we define moduli spaces $\M_m(a)$ and $\tilde{\M}_{m,k}(a)$. We refer to \cite[Section 6.3]{CELN} and \cite[Section 4.2]{CEL} for more general cases.
First, we define the space $\hat{\mathcal{M}}_m(a)$ consisting of pairs $(u,\kappa)$ of  $\kappa \in \mathcal{C}_{2m+1}$ and a $C^{\infty}$ map
\[u \colon D_{2m+1} \to T^*\R^n\]
satisfying the following conditions:
\begin{itemize}
\item $du + J'\circ du \circ j_{\kappa} =0$.
\item For $k=1,\dots ,2m+1$, $u(\partial_k D_{2m+1}) \subset \begin{cases} L_K  & \text{ if }k \text{ is odd,} \\ \R^n & \text{ if }k \text{ is even.}\end{cases}$
\item  $u\circ \psi_0(s,t)$ is contained in $T^*\R^n \setminus \R^n$ if $s\in[0,\infty)$ is sufficiently large.
Moreover, there exist $s_0\in \R$ and $q_1,\dots ,q_{2m}\in K= \R^n \cap L_K$ such that
\[ \tau_{-Ts}\circ ( F^{-1} \circ u\circ \psi_0)(s,\cdot) \to (s_0 , a(T\cdot)) \text{ if }s\to \infty \]
in $\R\times U^*\R^n$ in the $C^{\infty}$ topology on $[0,1]$ and
\[ u\circ \psi_k(s,\cdot) \to q_k \text{ if }s\to -\infty \]
in $T^*\R^n$ in the $C^{\infty}$ topology on $[0,1]$ for every $k\in \{1,\dots ,2m\}$.
\end{itemize}
\begin{figure}
\centering
\begin{overpic}[height=5cm]{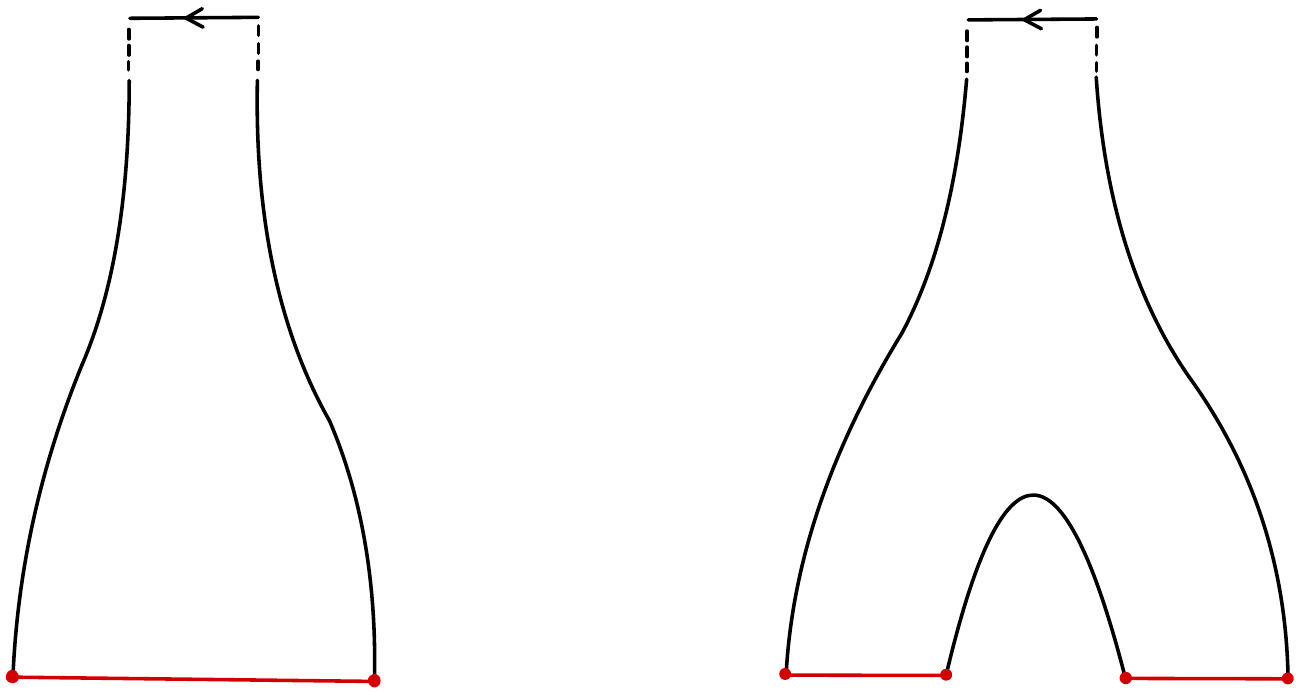}
\put(14,53.5){$a$}
\put(78.5,53){$a$}
\put(27,17){$L_K$}
\put(96.5,17){$L_K$}
\put(13,2.5){\color{red}$\R^n$}
\put(90,2.5){\color{red}$\R^n$}
\put(66,2.5){\color{red}$\R^n$}
\put(-4,-3){$u(p_2)$}
\put(25,-3){$u(p_1)$}
\put(95,-3){$u(p_1)$}
\put(82,-3){$u(p_2)$}
\put(69,-3){$u(p_3)$}
\put(55,-3){$u(p_4)$}
\end{overpic}
\caption{$J'$-holomorphic curves in $T^*\R^n$ with the red boundaries in  $\R^n$ and the black boundaries in $L_K$. The left-hand side belongs to $\M_1(a)$ and the right-hand side belongs to $\M_2(a)$.}\label{figure-switching}
\end{figure}
Figure \ref{figure-switching} describes the case where $m=1,2$.

In addition, for each $k\in \{1,\dots ,2m+1\}$, we consider the product space $\hat{\M}_{m}(a)\times \partial_k D_{2m+1}$.
When $m=0$, the action of $\mathcal{A}_1$ is defined by
$(u, \kappa) \cdot \varphi \coloneqq (u\circ \varphi, \kappa)$ on $\hat{\M}_0(a)$ and $((u,\kappa),z)\cdot \varphi \coloneqq ( (u\circ \varphi, \kappa) , \varphi^{-1}(z))$ on $\hat{\M}_0(a)\times \partial_1 D_1$ for every $\varphi \in \mathcal{A}_1$.
Now we define
\begin{align*}
 \M_m(a) & \coloneqq \begin{cases} \hat{\M}_m(a) & \text{ if }m\geq 1, \\
\hat{\M}_0(a)/\mathcal{A}_1 & \text{ if }m=0 ,\end{cases} \\
\tilde{\M}_{m,k}(a) & \coloneqq \begin{cases} \M_m(a) \times \partial_k D_{2m+1} & \text{ if }m\geq 1, \\
(\hat{\M}_0(a)\times \partial_1 D_1) /\mathcal{A}_1 & \text{ if }m=0 ,\end{cases}
\end{align*}
for every $k\in \{1,\dots ,2m+1\}$. 

For $(u,\kappa)\in \M_m(a)$, we denote $q_k=\lim_{s\to \infty} u\circ \psi_k(s,t)$ by $u(p_k)$ for every $k=1,\dots ,2m$. 
In particular, for $m\in \{1,2\}$, we define the following two maps:
\begin{align}\label{map-Gamma}
\begin{split}
\Gamma_1 &\colon \mathcal{M}_1(a) \to K\times K \colon (u,\kappa) \mapsto
 (u(p_1),u(p_2) ), \\
\Gamma_2 & \colon \mathcal{M}_2(a) \to (K\times K)^{\times 2} \colon (u,\kappa) \mapsto
 ((u(p_1),u(p_2))   , (u(p_3),u(p_4) ) ).
 \end{split}
\end{align}

Note that when $m=1$, $\M_1(a)$ contains
the trivial strip $u_a \colon D_3 \to T^*\R^n$ over $a$. It is defined as the composition of the biholomorphic map (\ref{bihol}) and the map
\begin{align}\label{eq-trivial-strip}
 [0,\infty)\times [0,1] \to T^*\R^n \colon (s,t) \mapsto (q_0+ Tt \cdot p_0, h(s) \cdot  p_0), 
 \end{align}
where $h\colon[0,\infty)\to [0,\infty)$ is a $C^{\infty}$ function characterized by $h(0)=0$ and $h'(s)= T \rho(h(s))$. It is a $J'$-holomorphic map from the third condition on $J'$.

\subsubsection{Moduli spaces of constant disks and evaluation map}

We define the moduli space $\M_{L_K,\R^n,L_K}$ (resp. $\M_{\R^n,L_K,\R^n}$) which consists of $C^{\infty}$ maps
\[v\colon D_3 \to T^*\R^n\]
satisfying the following conditions:
\begin{itemize}
\item $dv+J'\circ dv \circ j_{0}=0$.
\item For $k\in \{1,2,3\}$, $v(\partial_kD_{3}) \subset  \begin{cases} L_K  & \text{ if }k =1,3, \\ \R^n & \text{ if }k =2.\end{cases}  \left( \text{resp. } \begin{cases} \R^n  & \text{ if }k =1,3, \\ L_K & \text{ if }k =2.\end{cases} \right)$ 
\item There exits $q_0\in L_K$ (resp. $q_0\in \R^n$) and $q_1,q_2\in K=\R^n \cap L_K$ such that
\[ \begin{array}{cc}
 v\circ \psi_0(s,\cdot) \to q_0 \text{ if }s\to \infty, &
 v\circ \psi_k (s,\cdot) \to q_k \text{ if }s\to -\infty ,
\end{array}\]
in the $C^{\infty}$ topology on $[0,1]$ for $k=1,2$.
\end{itemize}
Let us denote $q_0=\lim_{s\to \infty} v \circ \psi_0(s,t)$ by $v(p_0)$ and $q_k= \lim_{s\to -\infty} v \circ \psi_k(s,t)$ by $v(p_k)$ for $k=1,2$.
We define the evaluation map at $p_0$ by
\begin{align*}
\ev_0 \colon & \M_{L_K,\R^n,L_K} \to L_K\colon v \mapsto v(p_0), \\
\ev_0 \colon & \M_{\R^n,L_K,\R^n} \to \R^n \colon v \mapsto v(p_0) .
\end{align*}
Since $\rest{\lambda}{\R^n}=0$ and $\rest{\lambda}{L_K}=0$, for any $v\in \M_{L_K,\R^n,L_K} \sqcup \M_{\R^n,L_K,\R^n}$,
\[\int_{D_3} v^* (d\lambda) =0\]
by Stokes' theorem.
This implies that both $ \M_{L_K,\R^n,L_K}$ and $ \M_{\R^n,L_K,\R^n} $ consist of constant disks in  $K=\R^n \cap L_K$, and thus the above evaluation maps are bijections onto $K$. Since
 $\M_{L_K,\R^n,L_K}$ and $\M_{\R^n,L_K,\R^n}$ are cut out transversely, each of them is a $C^{\infty}$ manifold of dimension $n-d$ and $\ev_0$ is a diffeomorphism onto $K$.

Next, for any $a\in \mathcal{R}(\Lambda_K)$, $m\in \Z_{\geq 0}$ and $k\in \{1,\dots , 2m+1\}$, we define the evaluation map
\[ \tilde{\ev}_k \colon \tilde{\M}_{m,k}(a) \to \begin{cases}
L_K & \text{ if }k \text{ is odd,} \\ \R^n & \text{ if }k \text{ is even.}\end{cases} \colon ((u,\kappa),z ) \mapsto u(z).\] 
%
We will later consider the fiber product over $L_K$
\begin{align}\label{curve-with-L_K-node}
  \tilde{\M}_{m,k}(a) \ftimes{\tilde{\ev}_k}{\ev_0} \M_{L_K,\R^n,L_K}
 \end{align}
when $k$ is an odd number and the fiber product over $\R^n$
\begin{align}\label{curve-with-Rn-node}
  \tilde{\M}_{m,k}(a) \ftimes{\tilde{\ev}_k}{\ev_0} \M_{\R^n,L_K,\R^n}
\end{align}
when $k$ is an even number. See Figure \ref{figure-bubble} for examples.
These are moduli spaces of $J'$-holomorphic curves with a boundary node and contained in the compactification of $\M_{m+1}(a)$. See \cite[Subsection 4.2.1]{CEL}.

\begin{figure}
\centering
\begin{overpic}[height=5cm]{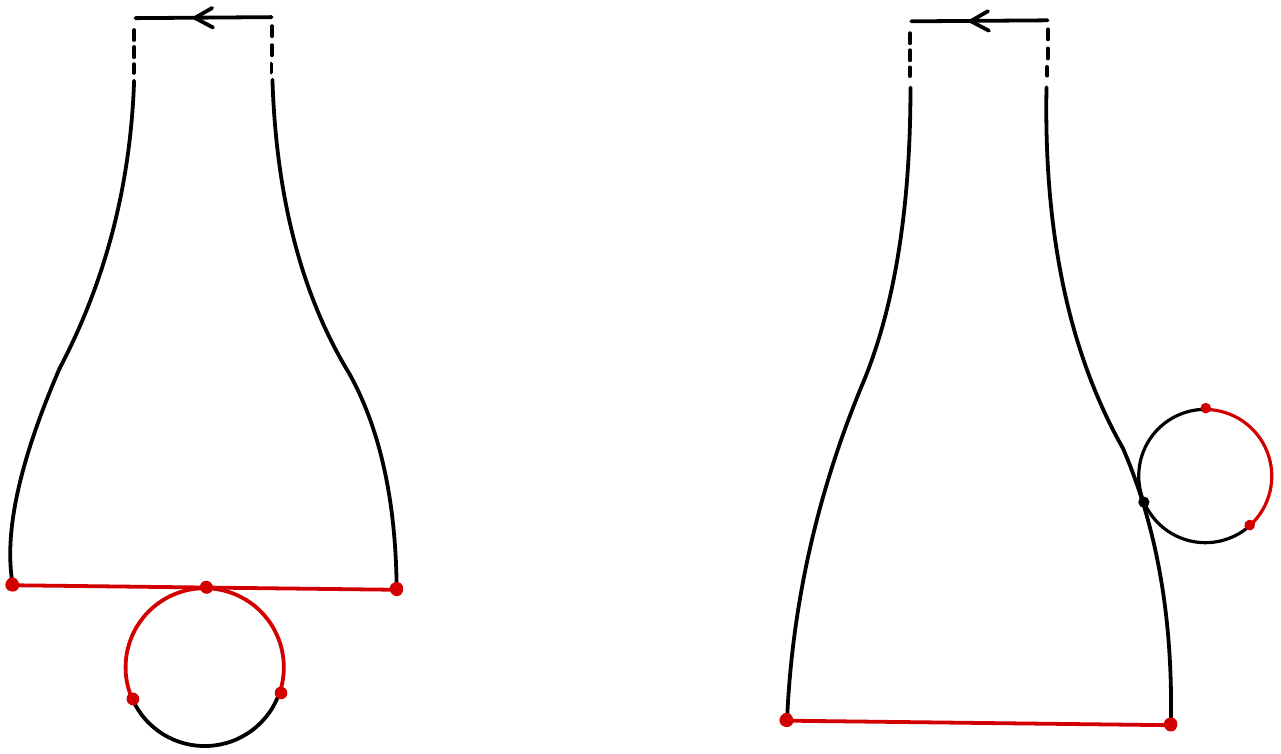}
\put(14.5,30){$u$}
\put(14.5,6){$v$}
\put(8,15){$ \tilde{\ev}_2(u,z) $}
\put(75,29){$u$}
\put(93,20){$v$}
\put(72.5,17){$\tilde{\ev}_1(u,z)$}
\put(14.5,59){$a$}
\put(75,59){$a$}
\end{overpic}
\caption{$J'$-holomorphic curves with a boundary node. The left-hand side belongs to $\tilde{\M}_{1,2}(a) \ftimes{\tilde{\ev}_2}{\ev_0} \M_{\R^n,L_K,\R^n}$ and the right-hand side belongs to $\tilde{\M}_{1,1}(a) \ftimes{\tilde{\ev}_1}{\ev_0} \M_{L_K,\R^n,L_K}$. Here, we omit writing $\kappa\in \mathcal{C}_3$.}\label{figure-bubble}
\end{figure}

\subsubsection{Basic properties}

Let us compute the formal dimension $\vdim \mathcal{M}_m(a)$ of $\mathcal{M}_m(a)$, which the Fredholm index of the linearlization of a section (the Cauchy-Riemann operator) of a Banach bundle whose zero locus is $\mathcal{M}_m(a)$. (More precisely, when $m=0$, we need to stabilize pseudo-holomorphic curves in $\M_0(a)$.)
\begin{prop}\label{prop-vdim}
$\vdim \mathcal{M}_m(a)= |a| - m(d-2)$. 
\end{prop}
\begin{proof}
We compute inductively on $m=0,1,2,\dots$. If $m=0$, then $\vdim \mathcal{M}_0(a) = |a|$. See, for instance, \cite[Section 4.2.4]{DR-lift}. Suppose that $\vdim \mathcal{M}_m(a) = |a|-m(d-2)$ for some $m\in \Z_{\geq 1}$. The boundary of the compactification of  $\M_{m+1}(a)$ contains the moduli space (\ref{curve-with-L_K-node}).
By a linear gluing argument, $\vdim \M_{m+1}(a) -1$ is equal to the formal dimension of (\ref{curve-with-L_K-node}). Therefore,
\begin{align*}
 \vdim \M_{m+1}(a) -1 &= (\vdim \M_{m}(a) +1) + \dim \M_{L_K,\R^n,L_K} -\dim \R^n \\
 &= |a|-(m+1)(d-2) -1.
 \end{align*}
This completes the induction on $m$.
\end{proof}

In the remaining part of this subsection, we observe the case of $m=0,1$.

\begin{prop}\label{prop-no-rigid}
For  generic $J$, the moduli space $\M_0(a)$ is cut out transversely and it is a $C^{\infty}$ manifold of dimension $|a|$. In particular, from the assumption that $d\geq 4$, $\M_0(a)$ is a $C^{\infty}$ manifold of positive dimension for every $a\in \mathcal{R}(\Lambda_K)$.
\end{prop}
The transversality is achieved by a generic perturbation of $J$ as in Proposition \ref{prop-transverse}. The positivity of the dimension of $\M_0(a)$, which follows from $|a|=\ind x_a +(d-2)$, is important in the later arguments.

Next, we consider a continuous function
\[ \len \colon K\times K \to \R \colon  (q,q') \mapsto |q-q'|.\]
Note that for any $C^{\infty}$ path $\gamma\colon [0,1]\to \R^n$ with $\gamma(0),\gamma(1)\in K$, its length $\int_{t=0}^1 \left| \frac{d \gamma}{dt}(t)\right|  dt$ is greater than or equal to $\len (\gamma(0),\gamma(1))$. The next proposition follows from \cite[Proposition 8.9]{CELN}.
\begin{prop}\label{prop-length}
For any $(a\colon [0,T]\to U^*\R^n)\in \mathcal{R}(\Lambda_K)$ and  $(u,\kappa_0) \in \mathcal{M}_1(a)$, 
\[  \len (\Gamma_1(u,\kappa_0)) \leq T .\]
The equality holds if and only if  $u$ is the trivial strip $u_a$.
\end{prop}

This theorem can be generalized to $\M_m(a)$ for $m\geq 2$ by \cite[Proposition 8.9]{CELN}, but the case of $m=1$ is sufficient for our purpose.

\subsection{Construction of an isomorphism of chain complex}\label{subsec-isom-chain-cpx}

For now,  we omit writing the conformal structure for each element of $\M_0(a)$ and $\M_1(a)$ for every $a \in \mathcal{R}(\Lambda_K)$, since both $\mathcal{C}_1$ and $\mathcal{C}_3$ consist of a single element.

\subsubsection{Moduli space cut out by stable manifold}
We take $g\in \mathcal{G}_K$. For any $a\in \mathcal{R}(\Lambda_K)$ and $x\in \mathcal{C}(K)$, we consider the fiber product over $K\times K$
\[ \mathcal{N}_g (a;x) = \mathcal{M}_1(a) \ftimes{\Gamma_1}{i^s} \W^s_g(x),\]
where $i^s \colon \W^s_g(x)\to K\times K$ is the inclusion map.
By Proposition \ref{prop-vdim}, its formal dimension is
\[|a|-d+2-\ind x .\]

Recall the bijection $\mathcal{R}(\Lambda_K) \to \mathcal{C}(K)\colon a \mapsto x_a$ of (\ref{bijection-RC}). 
\begin{lem}\label{lem-length-filtration}
Fix $(a\colon [0,T]\to U^*\R^n)\in \mathcal{R}(\Lambda_K)$. For any $x\in \mathcal{C}(K)$, the following hold:
\begin{itemize}
\item If $\len (x)>T$, then $\mathcal{N}_g (a;x) = \emptyset$.
\item Suppose that $\len (x)=T$. Then, $ \mathcal{N}_g (a;x) \neq \emptyset$ if and only if $x=x_a$. Moreover, if $x=x_a$, then $\mathcal{N}_1(a; x_a) $ consists only of $(u_a, x_a)$, where $u_a$ is the trivial strip over $a$.
\item For  generic $J$, $\mathcal{N}_g(a;x)$ is cut out transversely whenever  $\len (x)<T$.
\end{itemize}
\end{lem}
\begin{proof}
The first assertion follows from the inequality of Proposition \ref{prop-length} and that $\len (y)\geq \len (x)$ for every $y\in \W^s_g(x)$. 

For the second assertion, suppose that $T=\len (x)$. For an arbitrary element $(u, y)\in \mathcal{N}_g(a;x)$, by Proposition \ref{prop-length},
\[\len(y)=\len(\Gamma_1(u)) \leq  T=\len (x).\]
Since $y\in \W^s_g(x)$, this means that $y=x$ and the equality $\len(\Gamma_1(u)) = T$ holds. Therefore, by Proposition \ref{prop-length}, $u=u_a$ and $x=y=\Gamma_1(u_a)=x_a$. 

For the third assertion, suppose that $\len (x) <T$. Then, $u\neq u_a$ by Proposition \ref{prop-length}.
On $\Sigma \coloneqq u^{-1} (\{(q,p) \in T^* \R^n \mid |p| >  1\})$, let $\bar{u}\colon \Sigma \to U^*\R^n$ be the $U^*\R^n$-component of the map $F^{-1}\circ u $. Then, in the same way as the proof of Proposition \ref{prop-transverse}, it is shown that
\[ \{z \in \Sigma \setminus \partial D_{2m+1} \mid (\bar{u})^{-1} (\bar{u}(z))=\{z\},\ \pr_{\xi} \circ  (d\bar{u})_z \text{ is injective}\}\]
is not the empty set 
and the transversality can be achieved by a generic perturbation of $J \in \mathcal{J}_{U^*B^n_l}$.
\end{proof}

It remains to check the transversality for $\mathcal{N}_g(a;x_a)$ at the trivial strip $u_a$.

\begin{lem}\label{lem-trivial-strip}
The moduli space $\mathcal{N}_g(a; x_a)$ is cut out transversely.
\end{lem}
We will make some setup for analysis and give a detailed proof of this lemma in Appendix \ref{subsec-A3}. The proof is based on the ideas in \cite[Proposition 8.9]{CELN} and \cite[Proposition 2]{Asp}.

\begin{rem}\label{rem-cutout-transversely}
Using the map $\Gamma_1 \colon \M_1(a) \to K\times K$, we have
$\mathcal{N}_g(a;x) \cong \Gamma_1^{-1}(\W^s_g(x))$.
The condition that $\mathcal{N}_g(a;x)$ is cut out transversely is equivalent to satisfying the two conditions:
\begin{itemize}
\item $\M_1(a)$ is cut out transversely on a neighborhood $U_{a;x}$ of $\Gamma_1^{-1}(\W^s_g(x))$
\item The $C^{\infty}$ map $\Gamma_1\colon U_{a;x}\to K\times K$ is transversal to $\W^s_g(x)$.
\end{itemize}
\end{rem}

\subsubsection{Compactification of $\N_g(a;x)$ and construction of isomorphism}\label{subsubsec-compact-NV}
As a notation, for any $C^{\infty}$ path $\gamma \colon \partial_k D_{2m+1} \to \R^n$ with a finite length, let $\len(\gamma) \in \R_{\geq 0}$ denote the length of $\gamma$.
For any $u\in \M_1(a)$, if $\len \left( \rest{u}{\partial_2 D_3}\right)\leq \sqrt{2 \epsilon_0}$, then
\[ |u(p_1)-u(p_2)| \leq  \len \left( \rest{u}{\partial_2 D_3}\right)\leq \sqrt{2 \epsilon_0} , \]
and thus $E \circ \Gamma_1(u) \leq \epsilon_0$.
By (\ref{stable-contained}), this implies that $\Gamma_1(u) \notin \W^s_g(x)$ for every $x\in \mathcal{C}(K)$.
Therefore, if we define
\[\M_1(a)^{\geq \epsilon_0} \coloneqq \{u\in \M_1(a) \mid \len \left( \rest{u}{\partial_2 D_3} \right) \geq \sqrt{2\epsilon_0} \},\]
then
\[ \mathcal{N}_g(a;x) = \M_1(a)^{\geq \epsilon_0} \ftimes{\Gamma_1}{i^s} \W^s_g(x) . \]

By \cite[Theorem 1.1]{CEL}, $\M_m(a)$ for $m=0,1,\dots$ admits a compactification by adding pseudo-holomorphic buildings, possibly with boundary nodes.
See Figure \ref{figure-hol-bdry-1} in the case of $m=1$.
Let $\overline{\M}_1(a)^{\geq \epsilon_0}$ be the closure of $\M_1(a)^{\geq \epsilon_0}$ in this compactification.
In addition, let $\overline{\M}^1_1(a)^{\geq \epsilon_0}$ be the union of its strata of codimension $0$ and $1$.
Concretely, it is the disjoint union of $\M_1(a)^{\geq \epsilon_0}$ and
\begin{align}\label{building-M1}
\begin{split}
& \M_J(a;b_1,\dots ,b_m) \\
\times & \M_0(b_1)\times \dots \times  \M_0(b_{k-1}) \times \M_{1}(b_k)^{\geq \epsilon_0} \times \M_0(b_{k+1}) \times \dots \times \M_0(b_m)
\end{split}
\end{align}
for every $m,k\in \Z_{\geq 1}$ and $b_1,\dots ,b_m \in \mathcal{R}(\Lambda_K)$ such that $|a| - \sum_{i=1}^m|b_i| \geq 1$.

\begin{figure}
\centering
\begin{overpic}[height=4.5cm]{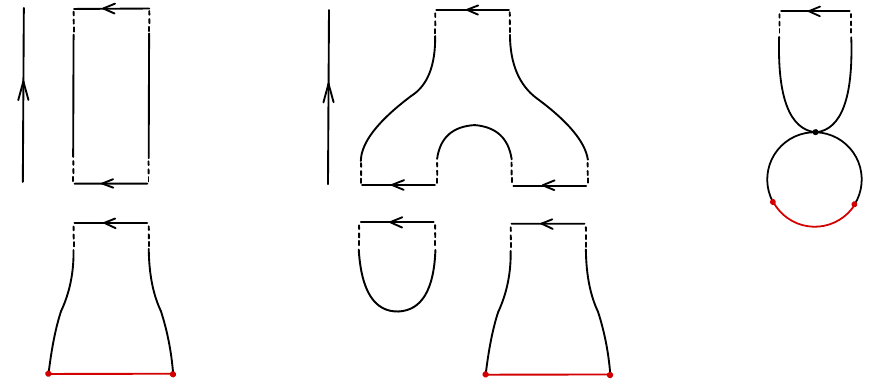}
\put(11,44){$a$}
\put(11,19.5){$b$}
\put(-15,33){$\R\times U^*\R^n$}
\put(-2,42){$\infty$}
\put(-5,22){$-\infty$}
\put(-7,10){$T^*\R^n$}
\put(52,44){$a$}
\put(44,19){$b'$}
\put(61.5,19){$b$}
\put(91,44){$a$}
\put(79,32){$T^*\R^n$}
\put(92,23){$v$}
\end{overpic}
\caption{$J'$-holomorphic buildings which appear in codimension $1$ strata of the compactification of $\M_1(a)$.
The red boundaries are mapped to $\R^n$ and the black boundaries are mapped to $\R\times \Lambda_K$ or $L_K$, depending on whether the target symplectic manifold is $\R\times U^*\R^n$ or $T^*\R^n$.
The right one describes a $J'$-holomorphic curve in $T^*\R^n$ with a boundary node. In fact, it is not contained in $\overline{\M}_1(a)^{\geq \epsilon_0}$ (see Remark \ref{rem-disjoint-component}). The middle one represents $J'$-holomorphic buildings which contain at least one $J'$-holomorphic curve in $\coprod_{b' \in \mathcal{R}(\Lambda_K)} \M_0(b')$.}\label{figure-hol-bdry-1}
\end{figure}

\begin{rem}\label{rem-disjoint-component}
The fiber product $\tilde{\M}_{0,1}(a)\ftimes{\tilde{\ev}_1}{\ev_0} \M_{L_K,\R^n,L_K}$, whose elements are illustrated by the right one in Figure \ref{figure-hol-bdry-1},
is contained in the codimension $1$ strata of the compactification of $\M_1(a)$.
However,  it is disjoint from $\overline{\M}_1(a)^{\geq \epsilon_0}$.
Indeed, the $J'$-holomorphic curve $v$ as in Figure \ref{figure-hol-bdry-1} is a constant disk in $T^*\R^n$ and its restriction on the red boundary component gives a constant path in $\R^n$.
Therefore, for any $u \in \M_1(a)$ sufficiently close to the strata $\tilde{\M}_{0,1}(a)\ftimes{\tilde{\ev}_1}{\ev_0} \M_{L_K,\R^n,L_K}$,
$\len (\rest{u}{\partial_2 D_3}) < \frac{1}{2} \epsilon_0$, and thus $\tilde{\M}_{0,1}(a)\ftimes{\tilde{\ev}_1}{\ev_0} \M_{L_K,\R^n,L_K}$ is disjoint from  $\overline{\M}_1(a)^{\geq \epsilon_0}$.
\end{rem}

We define $\overline{\Gamma}_1 \colon \overline{\M}_1(a)^{\geq \epsilon_0} \to K\times K$ to be the continuous extension of
$\Gamma_1\colon \M_1(a)^{\geq \epsilon_0} \to K\times K$.
Explicitly, it is defined on (\ref{building-M1}) as the composition of the projection map $ (\ref{building-M1}) \to \M_1(b_k)^{\geq \epsilon_0}$ and $\Gamma_1 \colon \M_1(b_k)^{\geq \epsilon_0} \to K\times K$.

\begin{lem}\label{lem-boundary-NV}
If $\ind x= |a|-d+2$, $\mathcal{N}_g(a;x)$ is a compact $0$-dimensional manifold. If $\ind x= |a|-d+1$,  $\mathcal{N}_g(a;x)$ is compactified to a compact $1$-dimensional manifold whose boundary is
\begin{align}\label{boundary-NV}
\coprod_{\ind x'=\ind x +1 }\mathcal{N}_g(a;x') \times \T_g(x';x) 
\sqcup \coprod_{|b|=|a|-1} \M_{J}(a;b) \times \mathcal{N}_g(b;x) .
\end{align}
\end{lem}

\begin{proof}
The compactification of $\mathcal{N}_g(a;x)$ is obtained by
\begin{align}\label{compact-NV}
 \overline{\N}_g(a;x) \coloneqq \overline{\M}_1(a)^{\geq \epsilon_0} \ftimes{\overline{\Gamma}_1}{i^s} \overline{\W}^s_g(x). 
 \end{align}
 We observe the case where the dimension of $\N_g(a;x)$ is $0$ or $1$.
If $\ind x = |a|-d+2$, $\overline{\N}_g(a;x)= \N_g(a;x)$ and it is a compact $0$-dimensional manifold.
Next, suppose that $\ind x = |a|-d+1$.
Then, $\overline{\N}_g(a;x)$ is the union of its open subsets
\[ \begin{array}{cc}  \M_1(a)^{\geq \epsilon_0} \ftimes{\Gamma_1}{i^s} \overline{\W}^{s,1}_g(x) , &  \overline{\M}^1_1(a)^{\geq \epsilon_0} \ftimes{\overline{\Gamma}_1}{i^s} \W^s_g(x)  . \end{array}\]
It is easy to see that
\[ \M_1(a)^{\geq \epsilon_0} \ftimes{\Gamma_1}{i^s} \overline{\W}^{s,1}_g(x) = \N_g(a;x) \sqcup \coprod_{ \ind x'=\ind x+1} \N_g(a;x')\times \T_g(x';x). \]
On the other hand,
\begin{align}\label{boundary-N-1}
\overline{\M}^1_1(a)^{\geq \epsilon_0} \ftimes{\overline{\Gamma}_1}{i^s} \W^s_g(x)
= \N_g(a;x) \sqcup \coprod_{|a|-1=|b|} \M_{J}(a;b) \times \mathcal{N}_g(b;x) .
\end{align}
To see this, note that among the codimension $1$ components (\ref{building-M1})  of $\overline{\M}^1_1(a)^{\geq \epsilon_0}$ for every $m\in \Z_{\geq 1}$, only the case of $m=1$ contributes to (\ref{boundary-N-1}).
The case of $m\geq 2$, which corresponds to the middle one of Figure \ref{figure-hol-bdry-1}, does not contribute to  (\ref{boundary-N-1}) since $\M_0(b_i)$ in  (\ref{building-M1})  for $i\neq k$ is a manifold of positive dimension by Proposition \ref{prop-no-rigid}.

It remains to show that if $\ind x = |a|-d+1$, $\overline{\mathcal{N}}_g(a;x)$ is a $1$-dimensional manifold with boundary (\ref{boundary-NV}).

Recall that $\overline{\W}^{s,1}_g(x)$ is a $C^{\infty}$ manifold with boundary and $i^s$ is a $C^{\infty}$ map.
Lemma \ref{lem-length-filtration} shows that the fiber product
$\M_1(a)^{\geq \epsilon_0} \ftimes {\Gamma_1}{i^s} \left(\partial \overline{\W}^{s,1}_g(x) \right)$ is transversal (see Remark \ref{rem-cutout-transversely}).
Therefore,
$\M_1(a)^{\geq \epsilon_0} \ftimes {\Gamma_1}{i^s} \overline{\W}^{s,1}_g(x)$
is a $1$-dimensional manifold with boundary
\[\M_1(a)^{\geq \epsilon_0} \ftimes {\Gamma_1}{i^s} \left(\partial \overline{\W}^{s,1}_g(x) \right) = \coprod_{\ind x'=\ind x +1 }\mathcal{N}_g(a;x') \times \T_g(x';x) .\]

On the other hand,
take a boundary point
\[([v^1],v^2) \in \M_J(a;b) \times \M_1(b)^{\geq \epsilon_0} \subset \overline{\M}_1^1(a)^{\geq \epsilon_0}\]
such that $|b|=|a|-1$ and  $\Gamma_1(v^2) \in \W^s_g(x)$. (Here, we omit writing the conformal structures.)
We argue as in \cite[Section 10.4]{CELN} replacing $\R^3$ by $\R^n$.
Let us sketch the analytic setup. (For details, see \cite[page 774, 775]{CELN}.)
First, from $v^1$ and $v^2$, we construct a pre-gluing $w_{\rho}$ in a Banach manifold, where $\rho \gg 0$ is a gluing parameter.
Next, we take $\mathscr{W}_{\rho}$ to be a neighborhood of $w_{\rho}$ which is diffeomorphic to a neighborhood of $0$ in $\mathscr{H}_{2,\rho}(w_{\rho}) \times B_{\rho} $.
Here $\mathscr{H}_{2,\rho}(w_{\rho})$ is a weighted Sobolev space of vector fields along $w_{\rho}$ with boundary conditions and $B_{\rho}$ is a certain finite dimensional factor.
Then, we consider the Cauchy-Riemann operator $\Dbar_{J'} \colon \mathscr{W}_{\rho} \to \mathscr{E}_{1,\rho}$ such that $(\Dbar_{J'})^{-1}(0) = \M_1(a)\cap \mathscr{W}_{\rho}$.
Here $\mathscr{E}_{1,\rho}$ is also a weighted Sobolev space of $w_{\rho}^* T (T^*\R^n)$-valued $(0,1)$-forms.

By the same formula as $\Gamma_1$ in (\ref{map-Gamma}),
we define $\Gamma_{1,\rho} \colon \mathscr{W}_{\rho} \to K\times K $
by evaluating maps in $\mathscr{W}_{\rho}$ at $p_1$ and $p_2$. It is a $C^{\infty}$ map transversal to $\W^s_g(x)$.
Moreover, since $v^2$ agrees with the pre-gluing $w_{\rho}$ near  $\{p_1,p_2\}$, $\Gamma_{1,\rho}(w_{\rho}) = \Gamma_1(v^2) \in \W^s_g(x)$.
Therefore, we obtain a Banach manifold $\mathscr{W}_{1,\rho} \ftimes{\Gamma_{1,\rho}}{i^s} \W^s_g(x) $ which contains $(w_{\rho},\Gamma_1(v^2))$ and
a $C^{\infty}$ map on the Banach manifold
\[\mathscr{F} \colon \mathscr{W}_{\rho} \ftimes{\Gamma_{1,\rho}}{i^s} \W^s_g(x)  \to \mathscr{E}_{1,\rho} \colon (w,y)  \mapsto \Dbar_{J'} (w) . \]

Let us denote the kernel of the linearization of the Cauchy-Riemann operator at $v^1 \in \hat{\M}_J(a;b)$ by $\mathscr{K}^1$ and the kernel of that at $v^2\in \M_1(b)$ by $\mathscr{K}^2$.
By \cite[Lemma 10.12]{CELN}, the linearization of $\Dbar_{J'}$ at $w_{\rho}$ 
\[D \Dbar_{J'} \colon \mathscr{H}_{2,\rho}(w_{\rho}) \times B_{\rho}  \to \mathscr{E}_{1,\rho}\]
admits a right inverse which is uniformly bounded as $\rho \to \infty$.
For the proof, we consider the $L^2$-orthogonal complement in $\mathscr{H}_{2,\rho}(w_{\rho}) \times B_{\rho} $ of a vector space spanned by elements from $\mathscr{K}^1$ and $\mathscr{K}^2$ (we need to cut off the elements of  $\mathscr{K}^1$ and $\mathscr{K}^2$. See \cite[page 776]{CELN}),
and then, we use the fact that the linearization of the Cauchy-Riemann operator at $v^1$ (resp. $v^2$) is an isomorphism on the $L^2$-orthogonal complement of $\mathscr{K}^1$ (resp. $\mathscr{K}^2$).
We remark that \cite[Lemma 10.12]{CELN} treats a different case (gluing constant disks \cite[Section 10.3]{CELN}), so we need to modify the argument to the present case as described in \cite[page 776]{CELN}.

Instead of $\mathscr{K}^2$, let us consider a fiber product
\[ \mathscr{K}^2 \ftimes{d\Gamma_1}{di^s} T_{\Gamma_1(v^2)}\W^s_g(x) ,\]
which is $0$-dimensional in the present case,
and use the fact from the third assertion of Lemma \ref{lem-length-filtration} that the linearization of the Cauchy-Riemann operator at
\[ (v_2 , \Gamma_1(v_2)) \in \M_1(b) \ftimes{\Gamma_1}{i^s}\W^s_g(x) = \mathcal{N}_{g}(b;x) \] is an isomorphism.
Then,  we can argue as in \cite[Lemma 10.12]{CELN} to show that the linearization of $\mathscr{F}$ at $(w_{\rho},\Gamma_1(v^2))$
\[ D\mathscr{F} \colon  \left( \mathscr{H}_{2,\rho}(w_{\rho}) \times B_{\rho} \right)    \ftimes{d\Gamma_{1,\rho}}{di^s} T_{\Gamma_1(v^2)}\W^s_g(x) \to \mathscr{E}_{1,\rho} \]
 admits a right inverse which is uniformly bounded as $\rho \to \infty$.
The rest of arguments, including using Floer’s Picard lemma \cite[Lemma 10.10]{CELN}, are parallel to those in \cite[Section 10.4, Proof of Theorem 10.3]{CELN}. (The present case corresponds to the boundary point of type (b) in \cite[Theorem 10.3]{CELN}.)
Then, it follows that $\overline{\M}_1^1(a)^{\geq \epsilon_0}  \ftimes {\overline{\Gamma}_1}{i^s} \W^s_g(x)$
is a $1$-dimensional manifold with boundary $\coprod_{|b|=|a|-1} \M_{J}(a;b) \times \mathcal{N}_g(b;x) $.
\end{proof}

Now we define a $\Z/2$-linear map
$ \Phi_g \colon CL_* \to CM_{*+d-2} $
by
\[ \Phi_g(a) \coloneqq \sum_{ \ind x = |a|- d+2} \#_{\Z/2}\mathcal{N}_g (a;x) \cdot x \]
for every $a\in \mathcal{R}(\Lambda_K).$
$\Phi_g$ is a chain map since
\[  d_g \circ \Psi_g (a) + \Psi_g \circ d_J(a) =0  \]
for every $a \in \mathcal{R}(\Lambda_K)$ by Lemma \ref{lem-boundary-NV}.

\begin{thm}\label{thm-isom-LM}
$\Psi_g$ is an isomorphism.
\end{thm}
\begin{proof}
 We fix an order on the set $\mathcal{C}(K)$ as $\mathcal{C}(K)=\{x_1,\dots ,x_N\}$ so that $E(x_i)\leq E(x_j)$ if $i\leq j$.
By using the bijection (\ref{bijection-RC}), we define $a_i\in \mathcal{R}(\Lambda_K)$ for $i=1,\dots ,N$ to be $x_i=x_{a_i}$. We consider the $\Z/2$-linear isomorphisms $CL_* \to (\Z/2)^N$ and $CM_*\to (\Z/2)^N$ which map $a_i$ and $x_i$ to $(0,\dots ,0,1,0,\dots ,0)\in(\Z/2)^N$ whose $i$-th component is $1$ and the other components are $0$. Through these isomorphisms, $\Phi_g$ is regarded as an $N\times N$-matrix with entries in $\Z/2$. The first two assertions of Lemma \ref{lem-length-filtration} imply that $\Phi_g$ is an upper triangular matrix all of whose diagonal entries are $1\in \Z/2$. Therefore, $\Phi_g$ is a linear isomorphism over $\Z/2$.
\end{proof}

\subsection{Compatibility with coproduct operations}\label{subsec-compatibility}

Choose an admissible triple $(K,g,g')$ and $J \in \mathcal{J}_{U^*B^n_l}$ satisfying the condition of Proposition \ref{prop-transverse}.
The purpose of this subsection is to prove the following theorem.

\begin{thm}\label{thm-diagram}
By generic perturbations of $J$ in $ \mathcal{J}_{U^*B^n_l}$ and $g'$ in $ \mathcal{G}_K$ if necessary,
 the following diagram commutes up to chain homotopy:
\begin{align}\label{diagram-coproduct}
\begin{split}
\xymatrix{
CL_{*+1} \ar[r]^-{\delta_J} \ar[d]^-{\Phi_g} & CL_{*}^{\otimes 2} \ar[d]^-{\Phi_{g'}\otimes \Phi_{g'}} \\
CM_{*-d+3} \ar[r]^-{\delta_{g,g'}} & CM_{*}^{\otimes 2}[-2d+4].
}\end{split}
\end{align}
\end{thm}

\subsubsection{Two chain maps from $CL_*$ to $CM_*^{\otimes 2}[-2d+3]$}\label{subsubsec-two-chain-map}
We fix an admissible triple $(K,g,g')$.
For every $a \in \mathcal{R}(\Lambda_K)$, 
we define an evaluation map $\ev^0_a$ and its variant $\ev^1_a$ by
\begin{align*}
\ev^0_a& \colon  \M_1(a)\times (0,1) \to \R^n \colon (u,\tau) \mapsto u\circ \psi(0,\tau) , \\
\ev^1_a& \colon  \M_1(a) \times (0,1) \to \R^n \colon (u,\tau) \mapsto (1-\tau) \cdot u(p_1) +\tau \cdot u(p_2),
\end{align*}
where $\psi$ is the biholomorphic map (\ref{bihol}). See Figure \ref{figure-homotopy}.
\begin{figure}
\centering
\begin{overpic}[height=5cm]{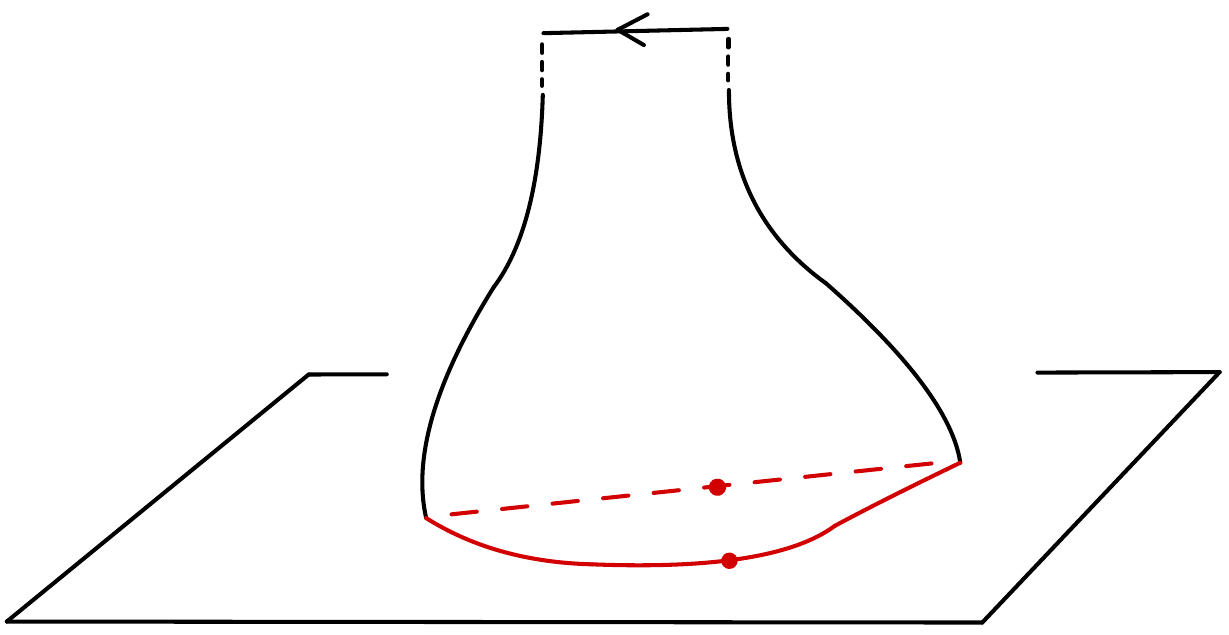}
\put(50,52){$a$}
\put(77,11){$u(p_1)$}
\put(27,6){$u(p_2)$}
\put(10,4){$\R^n$}
\put(10,40){$T^*\R^n$}
\put(57,2.5){$\ev^0_a(u,\tau)$}
\put(50,14){$\ev^1_a(u,\tau)$}
\end{overpic}
\caption{For $(u,\tau)\in \M_1(a)\times (0,1)$, $\ev^0_a(u,\tau)$ is a point on the curve $u(\partial_2 D_3)$ and $\ev^1_a(u,\tau)$ is a point on the linear path from $u(p_1)$ to $u(p_2)$.}\label{figure-homotopy}
\end{figure}
On $(\ev^j_a)^{-1}(K)$, we define
\[
\sp^j_a  \colon (\ev^j_a)^{-1}(K) \to (K\times K)^{\times 2} \colon (u,\tau) \mapsto ((u(p_1) , \ev^j_a(u, \tau) ) ,  (\ev^j_a(u, \tau)  , u(p_2))) .
\]
Let us consider the fiber product over $(K\times K)^{\times 2}$
\begin{align*}
\mathcal{S}^j_{g'}(a;x_1,x_2) \coloneqq (\ev^j_a)^{-1}(K) \ftimes{\sp^j_a}{i^s_2}  (\W^s_{g'}(x_1)\times \W^s_{g'}(x_2)).
 \end{align*}
Since $(K,g,g')$ is admissible,  $(1-\tau)\cdot q+\tau \cdot q'\notin K$ for any $(q,q')\in \mathcal{C}(K)$ and $0<\tau <1$.
Thus, $u\neq u_a$ for any $(u,\tau)\in (\ev_a^j)^{-1}(K)$ for $j=0,1$.
It should also be noted that for generic $J$, 
\[\begin{array}{cc} (u(p_0), u\circ \psi(0,\tau)) \in K\times \R^n , &  (u\circ \psi(0,\tau), u(p_1)) \in \R^n\times K \end{array}\]
are not contained in $\mathcal{C}(K)$ for every $(u,\tau) \in \M_1(a) \times (0,1) \cong \tilde{\M}_{1,2}(a)$ since $\dim \tilde{\M}_{1,2}(a) = \ind x_a + 1 < n + (n-d)$.
By similar arguments as in the proofs of Lemma \ref{lem-admissible} and the third assertion of Lemma \ref{lem-length-filtration}, we can show that by generic perturbations of $J$ and $g'$ outside $\mathcal{C}(K)$, $\mathcal{S}^j_{g'}(a;x_1,x_2)$ is cut out transversely and
its dimension is
\[ \dim \M_1(a) +1 -d - (\ind x_1+\ind x_2) = |a|-2d+3 - (\ind x_1+\ind x_2). \]
\begin{lem}\label{lem-boundary-S}
For $j=0,1$, the following hold:
If $\ind x_1+\ind x_2 = |a| - 2d+3$, then $\mathcal{S}^j_{g'}(a;x_1,x_2)$ is a compact $0$-dimensional manifold. If  $\ind x_1+\ind x_2 = |a|-2d+2$, then $\mathcal{S}^j_{g'}(a;x_1,x_2)$ is compactified to a compact $1$-dimensional manifold with boundary
\begin{align}\label{boundary-Sj}
\begin{split}
&\coprod_{\ind x'_1=\ind x_1+1} \mathcal{S}^j_{g'}(a;x'_1,x_2)\times \T_{g'} (x'_1;x_1) \\
\sqcup &\coprod_{\ind x'_2=\ind x_2+1} \mathcal{S}^j_{g'}(a;x_1,x'_2)\times \T_{g'} (x'_2;x_2) \\
\sqcup & \coprod_{|b|=|a|-1} \M_{J}(a;b) \times \mathcal{S}^j_{g'}(b;x_1,x_2).
\end{split}
\end{align}
\end{lem}

\begin{proof}
We first give a compactification of $\mathcal{S}^j_{g'}(a;x_1,x_2)$ in a way parallel to Lemma \ref{lem-boundary-T}.
Consider any $(u,\tau)\in (\ev^j_a)^{-1}(K)$. If  $u\notin \M_1(a)^{\geq \epsilon_0}$, then
\[\sp^j_a(u,\tau) \in \{(y_1,y_2)\in (K\times K)^{\times 2} \mid \min \{ E(y_1), E(y_2) \} \leq \epsilon_0 \} \]
and thus $\sp^j_a(u,\tau) \notin \W^s_{g'}(x_1) \times \W^s_{g'}(x_2)$ for any $x_1,x_2\in \mathcal{C}(K)$.
We define
\[ \overline{\ev}^j_a \colon \overline{\M}_1(a)^{\geq \epsilon_0} \times (0,1) \to \R^n \]
to be the continuous extension of $\ev^j_a\colon \M_1(a)^{\geq \epsilon_0} \times (0,1) \to \R^n$.
We extend $\sp^j_a$ on $(\overline{\ev}^j_a)^{-1}(K)$ continuously and denote it by $\overline{\sp}^j_a$.
From the compactness of $\overline{\M}_1(a)^{\geq \epsilon_0}$, there exists $\tau_1\in (0,\frac{1}{2})$ such that for every $(\bold{u},\tau) \in (\overline{\ev}^j_a)^{-1}(K)$ with $\tau \in [0,\tau_1] \sqcup [1-\tau_1,1]$, we have
\[\overline{\sp}^j_a(\bold{u},\tau) \in \{ (y_1,y_2) \in (K\times K)^{\times 2}\mid \min \{E(y_1),E(y_2)\} \leq \epsilon_0 \} \]
and thus $\overline{\sp}^j_a(\bold{u},\tau) \notin \W^s_{g'}(x_1) \times \W^s_{g'}(x_2)$ for any $x_1,x_2\in \mathcal{C}(K)$.

We define
\[\overline{\ev}^j_{a,[\tau_1,1-\tau_1]} \colon \overline{\M}_1(a)^{\geq \epsilon_0} \times [\tau_1,1-\tau_1] \to \R^n \]
to be the restriction of $\overline{\ev}^j_a$. It is clear that $(\overline{\ev}^j_{a,[\tau_1,1-\tau_1]})^{-1}(K)$ is a compact set.
From the above observations,
\begin{align}\label{compact-Sj}
\overline{\mathcal{S}}^j_{g'}(a;x_1,x_2) \coloneqq  (\overline{\ev}^j_{a,[\tau_1,1-\tau_1]})^{-1}(K) \ftimes{ \overline{\sp}^j_a}{i^s_2} (\overline{\W}^s_{g'}(x_1) \times \overline{\W}^s_{g'}(x_2) )
\end{align}
contains $\mathcal{S}^j_{g'}(a;x_1,x_2)$ and
it is a compactification of $\mathcal{S}^j_{g'}(a;x_1,x_2)$.

We observe the case where the dimension of $\mathcal{S}^j_{g'}(a;x_1,x_2)$ is $0$ or $1$. The argument is similar to the case of $\overline{\mathcal{N}}_g(a;x)$ in the proof of  Lemma \ref{lem-boundary-NV}.

If $\ind x_1 + \ind x_2 =|a| -2d+3$, then $\overline{\mathcal{S}}^j_{g'}(a;x_1,x_2) = \mathcal{S}^j_{g'}(a;x_1,x_2)$. If $\ind x_1 + \ind x_2 = |a| - 2d+2$,
$\overline{\mathcal{S}}^j_{g'}(a;x_1,x_2) $ is the union of its open subsets
\begin{align*}
& (\ev^j_a)^{-1}(K) \ftimes{\sp^j_a}{i^s_2}  \left(\overline{\W}^{s,1}_{g'}(x_1) \times \W^s_{g'}(x_2) \cup \W^s_{g'}(x_1) \times \overline{\W}^{s,1}_{g'}(x_2) \right) , \\
& \left( (\overline{\ev}^j_{a,[\tau_1,1-\tau_1]})^{-1}(K) \cap  \left( \overline{\M}^1_1(a)^{\geq \epsilon_0} \times (0,1) \right) \right) \ftimes{ \overline{\sp}^j_a}{i^s_2} (\W^s_{g'}(x_1) \times \W^s_{g'}(x_2) ) .
\end{align*}
It is easy to see that
\begin{align}\label{boundary-Sj-Morse}
\begin{split}
& (\ev^j_a)^{-1}(K) \ftimes{\sp^j_a}{i^s_2}  \left(\overline{\W}^{s,1}_{g'}(x_1) \times \W^s_{g'}(x_2) \cup \W^s_{g'}(x_1) \times \overline{\W}^{s,1}_{g'}(x_2) \right) \\
=\ & \mathcal{S}^j_{g'}(a;x_1,x_2) \\
 & \sqcup \coprod_{\ind x'_1=\ind x_1+1} \mathcal{S}^j_{g'}(a;x'_1,x_2)\times \T_{g'} (x'_1;x_1) 
\sqcup \coprod_{\ind x'_2=\ind x_2+1} \mathcal{S}^j_{g'}(a;x_1,x'_2)\times \T_{g'} (x'_2;x_2) .
\end{split}
\end{align}
On the other hand,
\begin{align}\label{boundary-Sj-hol}
\begin{split}
 & \left( (\overline{\ev}^j_{a,[\tau_1,1-\tau_1]})^{-1}(K) \cap  \left( \overline{\M}^1_1(a)^{\geq \epsilon_0} \times (0,1) \right) \right) \ftimes{ \overline{\sp}^j_a}{i^s_2} (\W^s_{g'}(x_1) \times \W^s_{g'}(x_2) ) \\
 =\ & \mathcal{S}^j_{g'}(a;x_1,x_2) \sqcup \coprod_{|b|=|a|-1} \M_{J}(a;b) \times \mathcal{S}^j_{g'}(b;x_1,x_2).
 \end{split}
 \end{align}
To see this, we note that among the components of (\ref{building-M1}) for $m\in \Z_{\geq 1}$, only the case of $m=1$ can contribute to (\ref{boundary-Sj-hol}).
The reason is the same as for (\ref{boundary-N-1}) in Lemma \ref{lem-boundary-NV}, that is, $\M_0(b')$ has a positive dimension for every $b'\in \mathcal{R}(\Lambda_K)$.
Therefore, $\overline{\mathcal{S}}^j_{g'}(a;x_1,x_2)$ is the disjoint union of $\mathcal{S}^j_{g'}(a;x_1,x_2)$ and (\ref{boundary-Sj}).

It remains to show that
when $\ind x_1 + \ind x_2 = |a| - 2d+2$,
$\overline{\mathcal{S}}^j_{g'}(a;x_1,x_2)$ is a $1$-dimensional manifold with boundary (\ref{boundary-Sj}).
We will refer to the case of $\overline{\mathcal{N}}_g(a;x)$ in the proof of Lemma \ref{lem-boundary-NV}.

From the fact that $\overline{\W}^{s,1}_{g'}(x_i)$ for $i=1,2$ is a $C^{\infty}$ manifold with boundary and $i^s_2$ is a $C^{\infty}$ map,
we can show that (\ref{boundary-Sj-Morse}) is a $1$-dimensional manifold with boundary 
\[ \coprod_{\ind x'_1=\ind x_1+1} \mathcal{S}^j_{g'}(a;x'_1,x_2)\times \T_{g'} (x'_1;x_1) 
\sqcup \coprod_{\ind x'_2=\ind x_2+1} \mathcal{S}^j_{g'}(a;x_1,x'_2)\times \T_{g'} (x'_2;x_2) . \]

On the other hand, we take any
\begin{align*}
(([v^1],v^2),\tau_2) &  \in (\overline{\ev}^j_{a,[\tau_1,1-\tau_1]})^{-1}(K) \cap \left( \left( \M_J(a;b) \times \M_1(b)^{\geq \epsilon_0} \right) \times (0,1) \right) \\
&\subset \overline{\M}_1^1(a)^{\geq \epsilon_0} \times [\tau_1,1-\tau_1]
\end{align*}
such that $|b|=|a|-1$ and $\overline{\sp}^j_a(([v^1],v^2),\tau_2) = \sp^j_b(v^2,\tau_2) \in \W^s_{g'}(x_1) \times \W^s_{g'}(x_2)$.
Let us use the notations introduced in the proof of Lemma \ref{lem-boundary-NV}: The pre-gluing $w_{\rho} \in \mathscr{W}_{\rho}$ for $\rho \gg 0$
and the Cauchy-Riemann operator
$ \Dbar_{J'} \colon \mathscr{W}_{\rho} \to \mathscr{E}_{1,\rho}$.
Here, let us use $L^p$-norm for $p>2$ instead of $L^2$-norm in \cite[Section 10.4]{CELN}  to define the weighted Sobolev spaces so that $\mathscr{W}_{\rho}$ is $C^1$-bounded.

By the same formula as $\ev^j_a$ and $\sp^j_{a}$, we define
\[\begin{array}{rl}
\ev^j_{\rho} & \colon \{ w\in \mathscr{W}_{\rho}\mid || \rest{w}{\partial_2 D_3} - \rest{v^2}{\partial_2 D_3} ||_{C^1}<\epsilon_2 \} \times (\tau_2-\epsilon_2, \tau_2+\epsilon_2) \to \R^n, \\
\sp^j_{\rho} & \colon (\ev^j_{\rho})^{-1}(K) \to  (K\times K)\times (K\times K) ,
\end{array}\]
for $\epsilon_2>0$.
Since $\mathcal{S}^j_{g'}(b;x_1,x_2)$ is cut out transversely at $((v^2, \tau_2),  \sp^j_b(v^2,\tau_2))$,
 there exists $\epsilon_2>0$ such that $\ev^j_{\rho} $ is transversal to $K$ and $\sp^j_{\rho}$ is transversal to $\W^s_{g'}(x_1)\times \W^s_{g'}(x_2)$.
We remark that the maps $\ev^j_{\rho}(w,\cdot)$ and $\sp^j_{\rho}(w, \cdot )$ are determined by the restriction of $w$ on $\partial_2D_3$ which is away from the gluing region, so we may take  $\epsilon_2$ independently on $\rho$.
%

Consider a Banach manifold
$(\ev^j_{\rho})^{-1}(K)\ftimes{\sp^j_{\rho}}{i^s_2} \left( \W^s_{g'}(x_1) \times \W^s_{g'}(x_2) \right)$
and a $C^{\infty}$ map on it
\[ \mathscr{F}\colon (\ev^j_{\rho})^{-1}(K) \ftimes{\sp^j_{\rho}}{i^s_2} \left( \W^s_{g'}(x_1) \times \W^s_{g'}(x_2) \right) \to \mathscr{E}_{1,\rho} \colon ((w,\tau), (y_1,y_2)) \mapsto \Dbar_{J'}(w). \]
Using the fact that $\M_J(a;b)$ and $S^j_{g'}(b;x_1,x_2)$ are cut out transversely,
we can argue as in \cite[Lemma 10.10]{CELN} to show that the linearization of $\mathscr{F}$ at $((w_{\rho},\tau_2), \sp^j_b(v^2))$ admits a right inverse which is uniformly bounded as $\rho \to \infty$.
The rest of arguments are parallel to those in \cite[Section 10.4, Proof of Theorem 10.3]{CELN}. Then, it follows that (\ref{boundary-Sj-hol}) is a $1$-dimensional manifold with boundary
$\coprod_{|b|=|a|-1} \M_{J}(a;b) \times \mathcal{S}^j_{g'}(b;x_1,x_2)$.
\end{proof}

For $j=0,1$, we define a $\Z/2$-linear map
$\sigma^j \colon CL_{*+1}(\Lambda_K) \to CM_{*}^{\otimes 2}[-2d+4] $
by
\[\sigma^j(a) \coloneqq \sum_{\ind x_1+\ind x_2 = |a|-2d+3} \#_{\Z/2} \mathcal{S}^j_{g'}(a;x_1,x_2) \cdot x_1\otimes x_2\]
for every $a\in \mathcal{R}(\Lambda_K)$.
By Lemma \ref{lem-boundary-S}, we have
\[ (d_{g'} \otimes \id_{CM_*} +\id_{CM_*}\otimes d_{g'})\circ \sigma^j + \sigma^j \circ d_J=0. \]
This means that $\sigma^j$ is a chain map.

Assuming the following three claims, we prove Theorem \ref{thm-diagram}:
\begin{enumerate}
\item $(\Phi_{g'}\otimes \Phi_{g'}) \circ \delta_J$ is chain homotopic to $\sigma^0$.
\item $\sigma^0$ is chain homotopic to $\sigma^1$. 
\item $\sigma^1$ is chain homotopic to $\delta_{g,g'}\circ \Phi_g$.
\end{enumerate}
These claims will be proved in Subsection \ref{subsubsec-step1}, \ref{subsubsec-step2} and \ref{subsubsec-step3} in order.
\begin{proof}[(Proof of Theorem \ref{thm-diagram} assuming the three claims)]
We include the chain maps $\sigma^0$ and $\sigma^1$ into the diagram (\ref{diagram-coproduct}) as follows:
\[\xymatrix{
CL_{*+1}\ar[r]^-{\delta_J} \ar[d]^-{\Phi_g} \ar@<0.5ex>[rd]^-{\sigma^0} \ar@<-0.5ex>[rd]_-{\sigma^1}& CL_{*}^{\otimes 2} \ar[d]^-{\Phi_{g'}\otimes \Phi_{g'}} \\
CM_{*-d+3} \ar[r]_-{\delta_{g,g'}} & CM_{*}^{\otimes 2}[-2d+4].
}\]
Then, Theorem \ref{thm-diagram} follows from this diagram and the above three claims.
\end{proof}

\subsubsection{Chain homotopy from $(\Phi_{g'}\otimes \Phi_{g'}) \circ \delta_J$ to $\sigma^0$}\label{subsubsec-step1}

For $a\in \mathcal{R}(\Lambda)$ and $x,x'\in \mathcal{C}(K)$, we consider the fiber product over $(K\times K)^{\times 2}$
\[\mathcal{N}_{g'}(a;x_1,x_2) \coloneqq \mathcal{M}_2(a) \ftimes{\Gamma_2}{i^s_2} (\W^s_{g'}(x_1)\times \W^s_{g'}(x_2)).\]

For  generic $J$, $\Gamma_2$ is transversal to $\W^s_{g'}(x_1)\times \W^s_{g'}(x_2)$.
As  a $C^{\infty}$ manifold, the dimension of $\mathcal{N}_{g'}(a;x_1,x_2) $ is
\[\dim \M_2(a) - (\ind x_1+\ind x_2) = |a|-2d+4 - (\ind x_1 + \ind x_2) . \]

\begin{figure}
\centering
\begin{overpic}[width=16cm]{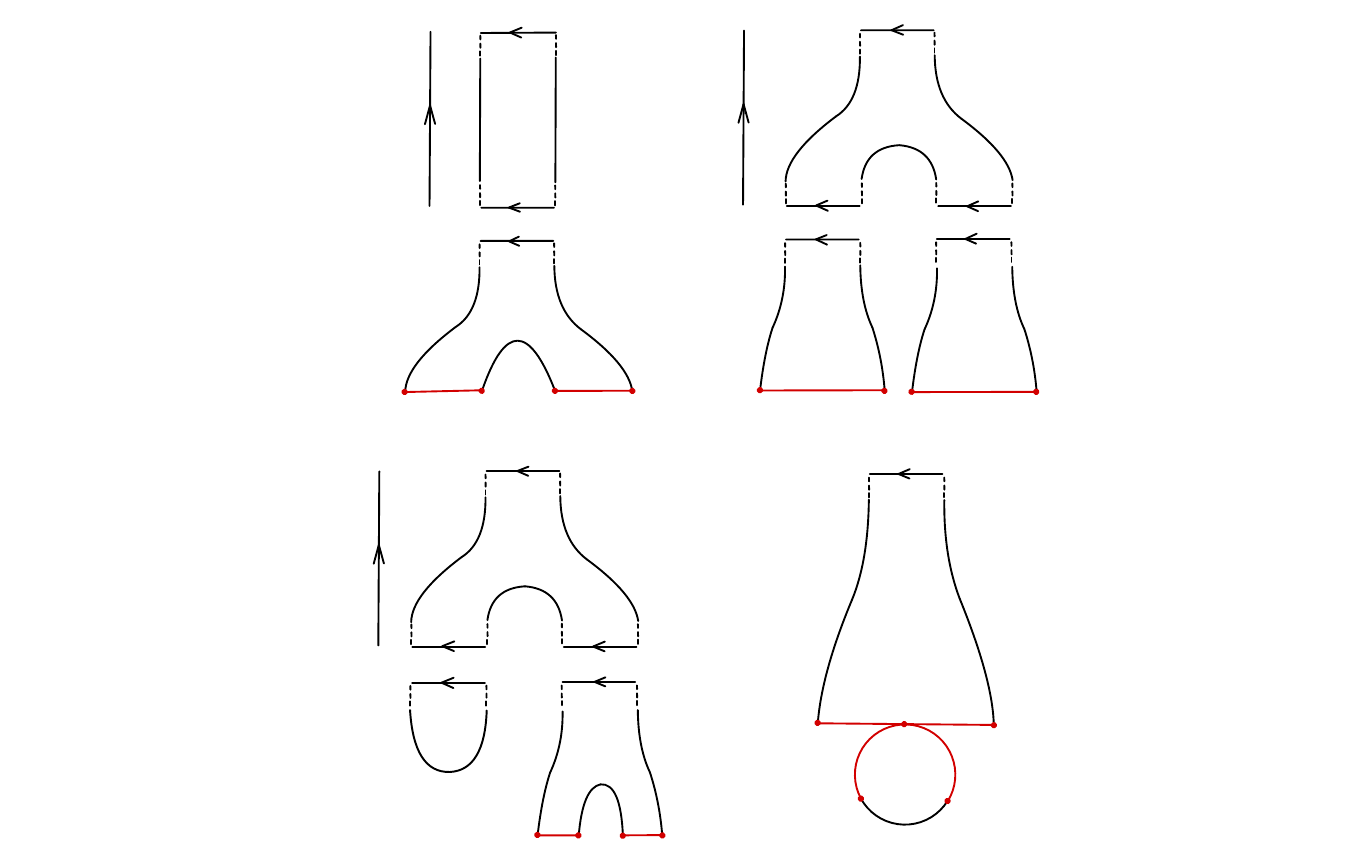}
\put(37,60.6){$a$}
\put(37,45){$b$}
\put(20,53){$\R\times U^*\R^n$}
\put(29,59.5){$\infty$}
\put(27,47){$-\infty$}
\put(24,40){$T^*\R^n$}
\put(64.7,60.6){$a$}
\put(58.5,45.4){$b_2$}
\put(69.5,45.4){$b_1$}
\put(37.2,28.1){$a$}
\put(31.5,12.5){$b'$}
\put(43,12.6){$b$}
\put(16.5,20.5){$\R\times U^*\R^n$}
\put(20,8){$T^*\R^n$}
\put(65.3,28){$a$}
\put(54,15){$T^*\R^n$}
\end{overpic}
\caption{$J'$-holomorphic buildings in the codimension $1$ strata of $\overline{\M}_2(a)^{\geq \epsilon_0}$.
The color of the boundary components are determined by the same rule as in Figure \ref{figure-hol-bdry-1}.
The the lower right one belongs to (\ref{curve-with-Rn-node}) for $m=1$ and already appeared in Figure \ref{figure-bubble}.
The the lower left one represents $J'$-holomorphic buildings which contain at least one $J'$-holomorphic curves in $\coprod_{b'\in \mathcal{R}(\Lambda_K)}\M_0(b')$.
}\label{figure-hol-bdry}
\end{figure}

\begin{lem}\label{lem-boundary-NV2}
If $\ind x_1+\ind x_2 = |a| - 2d+4$, $\N_{g'}(a;x_1,x_2)$ is a compact $0$-dimensional manifold. If $\ind x_1+\ind x_2= |a| -2d+3$, $\N_{g'}(a;x_1,x_2)$ is compactified to a compact $1$-dimensional manifold with boundary
\begin{align}\label{boundary-NV-2}
\begin{split}
&\coprod_{\ind x'_1=\ind x_1+1} \N_{g'}(a;x'_1,x_2)\times \mathcal{T}_{g'}(x'_1;x_1) \sqcup \coprod_{\ind x'_2=\ind x_2 +1}  \N_{g'}(a;x_1,x'_2)\times \mathcal{T}_{g'}(x'_2;x_2)  \\
\sqcup &\coprod_{|b|=|a|-1} \M_{J}(a;b) \times \N_{g'}(b;x_1,x_2) \\
\sqcup & \coprod_{|b_1|+|b_2|=|a|-1}\M_{J}(a;b_1,b_2)\times (\N_{g'}(b_1;x_1)\times \N_{g'}(b_2;x_2)) \\
 \sqcup  &\ \mathcal{S}^0_{g'}(a;x_1,x_2).
 \end{split}
\end{align}
\end{lem}
\begin{proof}
We define
\[ \M_2(a)^{\geq \epsilon_0} \coloneqq \{(u,\kappa)\in \mathcal{M}_2(a) \mid \len (\rest{u}{\partial_k D_5}) \geq \sqrt{ 2 \epsilon_0} \text{ for }k=2,4\} ,\]
then in the same way as in Subsection \ref{subsubsec-compact-NV}, we can show that
\[\N_{g'}(a;x_1,x_2) =  \M_2(a)^{\geq \epsilon_0} \ftimes{\Gamma_2}{i^s_2} (\W^s_{g'}(x_1) \times \W^s_{g'}(x_2)) .\]
Let $\overline{\M}_2(a)^{\geq \epsilon_0}$ be the closure of $\M_2(a)^{\geq \epsilon_0}$ in the compactification of $\M_2(a)$ from \cite[Theorem 1.1]{CEL}.
The $J'$-holomorphic buildings which appear in the codimension $1$ strata of $\overline{\M}_2(a)^{\geq \epsilon_0}$ are illustrated in Figure \ref{figure-hol-bdry}.

We note that in the codimension 1 strata of the compactification of $\M_2(a)$, there are several types of $J'$-holomorphic curves with a boundary node, including those of (\ref{curve-with-L_K-node}) and (\ref{curve-with-Rn-node}) for $m=1$ illustrated in Figure \ref{figure-bubble}.
However, these boundary strata except (\ref{curve-with-Rn-node})
are disjoint from $\overline{\M}_2(a)^{\geq \epsilon_0}$ (and (\ref{curve-with-Rn-node}) corresponds to the lower right one in Figure \ref{figure-hol-bdry}).
This can be checked in the same way as in Remark \ref{rem-disjoint-component}. 

We define $\overline{\Gamma}_2\colon \overline{\M}_2(a)^{\geq \epsilon_0}\to K\times K$ to be the continuous extension of $\Gamma_2$.
Then,
\[\overline{\N}_{g'}(a;x_1,x_2) \coloneqq \overline{\M}_2(a)^{\geq \epsilon_0} \ftimes{\overline{\Gamma}_2}{i^s_2} (\overline{\W}^s_{g'}(x_1) \times \overline{\W}^s_{g'}(x_2))\]
is a compact set and it is a compactification of $\N_{g'}(a;x_1,x_2)$.

Let us observe $\overline{\N}_{g'}(a;x_1,x_2)$ when the dimension of $\N_{g'}(a;x_1,x_2)$ is $0$ or $1$.
If $\ind x_1 + \ind x_2 = |a|-2d +4$, then $\overline{\N}_{g'}(a;x_1,x_2) = \N_{g'}(a;x_1,x_2)$.
If $ \ind x_1 + \ind x_2 = |a| -2d +3$, we note that the $J'$-holomorphic curves as in the lower left of Figure \ref{figure-hol-bdry} does not contribute to  $\overline{\N}_{g'}(a;x_1,x_2)$
since $\dim \M_0(b') >0$ for any $b' \in \mathcal{R}(\Lambda_K)$.
The remaining codimension $1$ strata of $\overline{\M}_2(a)^{\geq \epsilon_0}$ are 
\begin{align}\label{boundary-M2-1}
\begin{split}
&\coprod_{b\in \mathcal{R}(\Lambda_K)} \M_J(a;b) \times \M_2(b)^{\geq \epsilon_0} \\
\sqcup & \coprod_{b_1,b_2 \in \mathcal{R}(\Lambda_K)} \M_J(a;b_1,b_2) \times \left( \M_1(b_1)^{\geq \epsilon_0} \times \M_1(b_2)^{\geq \epsilon_0} \right) ,
\end{split}
\end{align}
which consists of $J'$-holomorphic buildings drawn by upper two pictures in Figure \ref{figure-hol-bdry}, and
\begin{align}\label{boundary-M2-3}
 \left( \tilde{\M}_{1,2}(a) \ftimes{\tilde{\ev}_2}{\ev_0} \mathcal{M}_{\R^n,L_K,\R^n} \right) \cap \overline{\M}_2(a)^{\geq \epsilon_0} ,
\end{align}
which consists of  $J'$-holomorphic curves with a boundary node drawn in the lower right of  Figure \ref{figure-hol-bdry}.
Therefore,
$\overline{\N}_{g'}(a;x_1,x_2)$ is the disjoint union of
\begin{align*}
 &\M_2(a) \ftimes{\overline{\Gamma}_2}{i^s_2} (\overline{\W}^{s,1}_{g'}(x_1)\times \overline{\W}^{s,1}_{g'}(x_2)) \\
 =\ &\N_{g'}(a;x_1,x_2) \\ 
& \sqcup \coprod_{\ind x'_1=\ind x_1+1} \N_{g'}(a;x'_1,x_2)\times \mathcal{T}_{g'}(x'_1;x_1) \sqcup \coprod_{\ind x'_2=\ind x_2 +1}  \N_{g'}(a;x_1,x'_2)\times \mathcal{T}_{g'}(x'_2;x_2)  
\end{align*}
and the fiber product
\[ \left( (\ref{boundary-M2-1}) \sqcup (\ref{boundary-M2-3}) \right) \ftimes{\overline{\Gamma}_2}{i^s_2} (\W^s_{g'}(x_1)\times \W^s_{g'}(x_2)) . \]

From the definitions of $\mathcal{N}_{g'}(b;x)$ and $\mathcal{N}_{g'}(b;x_1,x_2)$,
\begin{align*}
(\ref{boundary-M2-1}) \ftimes{\overline{\Gamma}_2}{i^s_2} (\W^s_{g'}(x_1)\times \W^s_{g'}(x_2)) =& \coprod_{|b|=|a|-1} \M_{J}(a;b) \times \N_{g'}(b;x_1,x_2)  \\
&\sqcup \coprod_{|b_1|+|b_2|=|a|-1}\M_{J}(a;b_1,b_2)\times (\N_{g'}(b_1;x_1)\times \N_{g'}(b_2;x_2)).
\end{align*}
Recall that $\M_{\R^n,L_K,\R^n}$ consists only of constant disks and $\ev_0 \colon \M_{\R^n,L_K,\R^n} \to K$ is a diffeomorphism onto $K$. We have a diffeomorphism
\[ (\ev^1_a)^{-1}(K) \to (\ref{boundary-M2-3}) \colon (u,\tau) \mapsto ( (u,\psi(0,\tau)), v_{(u,\tau)} ) ,  \]
where $v_{(u,\tau)}\in \M_{\R^n,L_K,\R^n}$ is the constant disk at $\ev^0_a(u,\tau) \in K=L_K \cap \R^n$.
Via this diffeomorphism, $\overline{\Gamma}_2$ on  (\ref{boundary-M2-3}) coincides with the map
\[ \sp^0_a\colon (\ev^0_a)^{-1}(K) \to (K\times K) \times (K\times K) . \]
Therefore,
\[ (\ref{boundary-M2-3})  \ftimes{\overline{\Gamma}_2}{i^s_2} (\W^s_{g'}(x_1)\times \W^s_{g'}(x_2)) \cong  (\ev^0_a)^{-1}(K)  \ftimes{\sp^0_a}{i^s_2} (\W^s_{g'}(x_1)\times \W^s_{g'}(x_2))   = \mathcal{S}^0_{g'}(a;x_1,x_2) .\]
In summary, $\overline{\mathcal{N}}_{g'}(a;x_1,x_2)$ is the disjoint union of $\mathcal{N}_{g'}(a;x_1,x_2)$ and (\ref{boundary-NV-2}).

It remains to show that when $ \ind x_1 + \ind x_2 = |a| -2d +3$,
$\overline{\mathcal{N}}_{g'}(a;x_1,x_2)$ is a $1$-dimensional manifold with boundary  (\ref{boundary-NV-2}).

For neighborhoods of boundary points which are not in $\mathcal{S}^0_{g'}(a;x_1,x_2)$, the arguments are analogous to the case of the boundary points of $\overline{\mathcal{N}}_g(a;x)$ in the proof of Lemma \ref{lem-boundary-NV}.

For a neighborhood of $\mathcal{S}^0_{g'}(a;x_1,x_2)$, we modify the  argument in \cite[Section 10.3]{CELN} about gluing constant disks as follows:
Recall the Banach manifold $\mathscr{W}_{\rho}\ftimes{\Gamma_{1,\rho}}{i^s_2}\W^s_{g}(x)$ and the Cauchy-Riemann operator $\mathscr{F}$ on it in the proof of Lemma \ref{lem-boundary-NV}.
In the present case,  a pre-gluing $w_{\rho}$ and its neighborhood $\mathscr{W}_{\rho}$ are the ones constructed in \cite[Section 10.3]{CELN}, and we extend $\Gamma_2$ to a $C^{\infty}$ map $\Gamma_{2,\rho}  \colon \mathscr{W}_{\rho} \to (K\times K)^{\times 2}$  by evaluating maps in $\mathscr{W}_{\rho}$ at $p_1,\dots ,p_4$.
Here, $\rho \gg 0$ is a gluing parameter.
Then, we consider the Banach manifold $\mathscr{W}_{\rho} \ftimes{\Gamma_{2,\rho}}{i^s_2} \left(\W^s_{g'}(x_1)\times \W^s_{g'}(x_2)\right)$ and the Cauchy-Riemann operator on it, which maps $(w, (y_1,y_2))$ to $\Dbar_{J'}w$.
Again, we refer to \cite[Lemma 10.12]{CELN} to show that its linearization at $(w_{\rho}, \Gamma_{2,\rho}(w_{\rho}))$ admits a right inverse which is uniformly bounded as $\rho \to \infty$.
The rest of arguments are parallel to those in \cite[Section 10.3, Proof of Theorem 10.3]{CELN}.
(The present case corresponds to type (a) in \cite[Theorem 10.3]{CELN}.)
Then, it follows that $\mathcal{S}^0_{g'}(a;x_1,x_2)$ has an open neighborhood  in $\overline{\mathcal{N}}_{g'}(a;x_1,x_2)$ which is a $1$-dimensional manifold with boundary $\mathcal{S}^0_{g'}(a;x_1,x_2)$.
\end{proof}

We define a $\Z/2$-linear map $\kappa_1 \colon CL_*\to CM^{\otimes 2}_*[-2d+4]$
by
\[\kappa_1(a) \coloneqq \sum_{\ind x_1+\ind x_2 = |a|-2d+4} \#_{\Z/2} \mathcal{N}_{g'}(a;x_1,x_2) \cdot x_1\otimes x_2 \]
for every $a\in \mathcal{R}(\Lambda_K)$.
Lemma \ref{lem-boundary-NV2} implies that
\[ (d_{g'}\otimes \id_{CM_*}+ \id_{CM_*}\otimes d_{g'}) \circ \kappa_1 + \kappa_1 \circ d_g + (\Phi_{g'}\otimes \Phi_{g'}) \circ d_J + \sigma^0 =0.\]
This means that $\kappa_1$ is a chain homotopy from $(\Phi_{g'}\otimes \Phi_{g'}) \circ d_J$ to $\sigma^0$

\subsubsection{Chain homotopy from $\sigma^0$ to $\sigma^1$}\label{subsubsec-step2}

For every $a \in \mathcal{R}(\Lambda_K)$, we take a homotopy $(\ev^{\sigma}_a)_{\sigma \in [0,1]}$ from $\ev^0_a$ to $\ev^1_a$ by
\[  \ev^{\sigma}_a \colon \M_1(a) \times (0,1) \to \R^n \colon (u, \tau)\mapsto (1-\sigma) \cdot \ev^0_a(u,\tau) + \sigma \cdot \ev^1_a(u,\tau)\]
and define a splitting map for every $\sigma\in [0,1]$ by
\[\sp^{\sigma}_a \colon (\ev^{\sigma}_a)^{-1}(K) \to (K\times K)\times (K\times K)\colon (u,\tau) \mapsto ( (u(p_1),  \ev^{\sigma}_a(u,\tau))) , (\ev^{\sigma}_a(u,\tau), u(p_2)). \]
We then define
\[\tilde{\mathcal{S}}_{g'} (a;x_1,x_2) \coloneqq \coprod_{\sigma \in [0,1] } \{\sigma \} \times \left( (\ev^{\sigma}_a)^{-1}(K) \ftimes{\sp^{\sigma}_a}{i^s_2} (\W^s_{g'}(x_1)\times \W^s_{g'}(x_2))  \right) . \]
For generic $J$ and $g'$, it is cut out transversely and its dimension is 
$|a|-2d+4 -(\ind x_1 +\ind x_2)$.
This is proved by modifying the discussion in Subsection \ref{subsubsec-two-chain-map} to the case of the union of this parametrized moduli space.

By modifying the proof of Lemma \ref{lem-boundary-S} to the present case involving the parameter $\sigma\in [0,1]$, we can also show the following: If $\ind x_1 +\ind x_2= |a| - 2d+4$, $\tilde{\mathcal{S}}_{g'} (a;x_1,x_2)$  is a compact $0$-dimensional manifold.
If $\ind x_1 +\ind x_2 = |a| - 2d+3$, $\tilde{\mathcal{S}}_{g'} (a;x_1,x_2)$ is a $1$-dimensional manifold with boundary
\[\coprod_{i=0,1}\mathcal{S}^i_{g'}(a;x_1,x_2) \]
and it is compactified to a compact $1$-dimensional manifold by adding
\begin{align*}
& \coprod_{|b|=|a|-1} \M_J(a;b) \times \tilde{\mathcal{S}}_{g'}(b;x_1,x_2) \\
\sqcup & \coprod_{\ind x'_1=\ind x_1 +1} \tilde{\mathcal{S}}_{g'}(a;x'_1,x_2) \times \mathcal{T}_{g'}(x'_1;x_1) 
\sqcup \coprod_{\ind x'_2 = \ind x_2+1}  \tilde{\mathcal{S}}_{g'}(a;x_1,x'_2) \times \mathcal{T}_{g'}(x'_2;x_2)
\end{align*}
as boundaries.

 We define a $\Z/2$-linear map $\kappa_2 \colon CL_* \to CM_{*}^{\otimes 2}[-2d+4]$ by
\[\kappa_2(a) \coloneqq \sum_{\ind x_1+\ind x_2 = |a|-2d+4} \#_{\Z/2}\tilde{\mathcal{S}}_{g'}(a;x_1,x_2) \cdot x_1\otimes x_2\]
for every $a\in \mathcal{R}(\Lambda_K)$.
By observing the boundary of the compactification of $\tilde{\mathcal{S}}_{g'}(a;x_1,x_2)$ when  $\ind x_1 +\ind x_2 = |a| -2d+3$, we obtain
\[\sigma^0+\sigma^1 + \kappa_2 \circ d_J + (d_{g'}\otimes \id_{CM_*} + \id_{CM_*}\otimes d_{g'}) \circ \kappa_2 =0.\]
This means that $\kappa_2$ is a chain homotopy from $\sigma^0$ to $\sigma^1$.

\subsubsection{Chain homotopy from $\sigma^1$ to $\delta_{g,g'}\circ  \Phi_g$}\label{subsubsec-step3}

First, we note that
\[\ev^1_a = \ev\circ (\Gamma_1\times \id_{(0,1)})\]
for every $a \in \mathcal{R}(\Lambda_K)$.
For any $\rho \in [0,\infty)$, by using the flow $(\varphi^{\rho}_g)_{\rho \in \R}$ of $V_g$, we define
\[ \ev^1_{a,\rho} \colon \M_1(a) \times (0,1) \to \R^n \colon (u,\tau) \mapsto \ev (\varphi_g^{\rho}\circ \Gamma_1(u) , \tau)\]
and
\[ \sp^1_{a,\rho} \colon (\ev^1_{a,\rho})^{-1}(K) \to (K\times K) \times (K\times K) \colon (u,\tau) \mapsto \sp (\varphi^{\rho}_g\circ \Gamma_1(u) ,\tau).\]
We define
\[\mathcal{U}_{g'}(a;x_1,x_2)=\coprod_{\rho \in [0,\infty)} \{\rho\} \times \left( (\ev^1_{a,\rho})^{-1}(K) \ftimes{\sp^1_{a,\rho}}{i^s_2} (\W^s_{g'}(x_1)\times \W^s_{g'}(x_2))\right) .\]

For generic $J$, $\mathcal{U}_{g'}(a;x_1,x_2)$ is cut out transversely and it is a $C^{\infty}$ manifold of dimension $|a|-2d +4 -(\ind x_1 +\ind x_2)$.
Its compactification can be found in a similar manner as Lemma \ref{lem-boundary-NV}, except that we have the boundary component
\[ \coprod_{x \in \mathcal{C}(K) } \overline{\N}_g(a;x) \times \overline{\mathcal{T}}_{g,g'}(x;x_1,x_2) ,\]
which consists of the limit of sequences $(\rho_n, Y_n)_{n=1,2,\dots}$ in $\mathcal{U}_{g'}(a;x_1,x_2)$ such that $\lim_{n\to \infty} \rho_n = \infty$.
Here, $\overline{\mathcal{T}}_{g,g'}(x;x_1,x_2)$ defined by (\ref{compact-T}) is the compactification of $\mathcal{T}_{g,g'}(x;x_1,x_2)$.

We observe the case where the dimension of $\mathcal{U}_{g'}(a;x_1,x_2)$ is $0$ or $1$.
If $\ind x_1 +\ind x_2 = |a| - 2d+4$, then $\mathcal{U}_{g'}(a;x_1,x_2)$ is a compact $0$-dimensional manifold. If $\ind x_1 + \ind x_2 = |a| - 2d +3$, then $\mathcal{U}_{g'}(a;x_1,x_2)$ is a $1$-dimensional manifold with boundary
\[\{0\}\times \left( (\ev^1_{a})^{-1}(K) \ftimes{\sp^1_a}{i^s_2} (\W^s_{g'}(x_1)\times \W^s_{g'}(x_2) )  \right) \cong \mathcal{S}^1_{g'}(a;x_1,x_2),\]
and $\mathcal{U}_{g'}(a;x_1,x_2)$ is compactified to a compact $1$-dimensional manifold by adding
\begin{align*}
 & \coprod_{|b|=|a|-1} \M_J(a;b) \times \mathcal{U}_{g'}(b;x_1,x_2) \\
\sqcup & \coprod_{\ind x'_1=\ind x_1 +1} \mathcal{U}_{g'}(a;x'_1,x_2) \times \mathcal{T}_{g'}(x'_1;x_1) 
\sqcup  \coprod_{\ind x'_2 = \ind x_2+1}  \mathcal{U}_{g'}(a;x_1,x'_2) \times \mathcal{T}_{g'}(x'_2;x_2) \\
 \sqcup & \coprod_{\ind x= |a|-d+2} \mathcal{N}_{g}(a;x) \times \mathcal{T}_{g,g'}(x; x_1,x_2) 
\end{align*}
as boundaries.

We define a $\Z/2$-linear map
$ \kappa_3 \colon CL_*\to CM_{*}^{\otimes 2}[-2d+4]$
by
\[\kappa_3(a) \coloneqq \sum_{\ind x_1 + \ind x_2 = |a|-2d+4} \#_{\Z/2} \mathcal{U}_{g'}(a;x_1,x_2) \cdot x_1\otimes x_2.\]
for every $a\in \mathcal{R}(\Lambda_K)$.
By observing the boundary of the compactification of $\mathcal{U}_{g'}(a;x_1,x_2)$ when $\ind x_1 + \ind x_2 = |a| -2d +3$, we obtain
\[\sigma_1+ \delta_{g,g'} \circ \Phi_g + \kappa_3 \circ d_J + (d_{g'}\otimes \id_{CM_*} + \id_{CM_*}\otimes d_{g'}) \circ \kappa_3 =0.\]
This means that $\kappa_3$ is a chain map from $\sigma^1$ to $ \delta_{g,g'} \circ \Psi_g$.
This finishes the proof of Theorem \ref{thm-diagram}.

\section{Proof of main results}\label{sec-application}


\subsection{General properties of $\Lambda_K$ as a Legendrian submanifold}


\subsubsection{Thurston-Bennequin number}

Let us discuss the Thurston-Bennequin number of unit conormal bundles.
The condition in Definition \ref{def-TB}  that $H_n(Y;\Z)=0$ is satisfied for $Y=U^*\R^n$.
Since the conormal bundle $L_K$ is orientable for any submanifold $K$, $\Lambda_K$ is also an orientable manifold.
\begin{prop}\label{prop-TB}
Let $K$ be a compact submanifold of $\R^n$.
If there exists a nowhere vanishing vector field on $K$, then $\Lambda_K$ is null homologous in $U^*\R^n$ and moreover $\tb(\Lambda_K)=0$.
\end{prop}
\begin{proof}
 Let $\sigma \colon K\to TK$ be a section such that $|\sigma(q)|=1$ for every $q\in K$. Then, we define
\[ D_{\Lambda_K} \coloneqq \{( q, a\cdot \sigma(q) + v ) \in U^*\R^n \mid q\in K,\  a\geq 0 \text{ and }  v\in (T_q K)^{\perp} \text{ such that } a^2+ |v|^2 =1 \}. \]
It is an $n$-dimensional orientable manifold with boundary $\Lambda_K$. Therefore, $\Lambda_K$ is null homologous in $U^*\R^n$.

Next, in order to compute $\tb(\Lambda_K)$, we take $\epsilon>0$ such that $\varphi^t_{\alpha}(\Lambda_K) \subset U^*\R^n \setminus \Lambda_K$ for every $t\in (0,\epsilon]$.
Let $\pi_{\R^n} \colon T^*\R^n \to \R^n$ be the bundle projection.
If $\epsilon>0$ is sufficiently small, then $\pi_{\R^n} \circ \varphi^{\epsilon}_{\alpha}(q,p)= q+\epsilon p \notin K$ for every $(q,p)\in \Lambda_K$ since $p\in (T_qK)^{\perp}\setminus \{0\}$. On the other hand, $D_{\Lambda_K}$ is contained in $(\pi_{\R^n})^{-1}(K)$. This means that $D_{\Lambda_K}$ is disjoint from $\varphi^{\epsilon}_{\alpha}(\Lambda_K)$. Therefore, $\tb(\Lambda_K) = [D_{\Lambda_K}] \bullet [\varphi^{\epsilon}_{\alpha}(\Lambda_K)] =0$. 
\end{proof}

\subsubsection{Coproduct $\delta_K$ as a Legendrian isotopy invariant for $\Lambda_K$}
Hereafter, the coefficient of every singular homology group is $\Z/2$, unless otherwise specified.

Let $K$ be a compact connected submanifold of codimension $d\geq 4$. For generic $J\in \mathcal{J}_{U^*B^n_l}$, $\sigma \in \mathcal{O}_K$ and $g,g'\in \mathcal{G}_{K'}$, where $K'\coloneqq K_{\sigma}$, 
we have the following commutative diagram:
\begin{align}\label{diagram-coproduct-homology-2}
\begin{split}
\xymatrix@R=20pt{
HL_{*+1}(\Lambda_{K'}) \ar[r]^-{(\delta_J)_*} \ar[d]^-{ (\Phi_g)_*} & HL^2_{*}(\Lambda_{K'}) \ar[d]^-{ (\Phi_{g'}\otimes \Phi_{g'})_*} \\
HM_{*-d+3}(g)  \ar[r]^-{(\delta_{g,g'})_*} & HM^2_{*-2d+4}(g') \\
H_{*-d+3}(K'\times K',\Delta_{K'}) \ar[u]_-{\Psi_g\circ j_*} \ar[r]^-{\delta_{K'}} & H_{*-2d+4}((K'\times K',\Delta_{K'})^{\times 2}) \ar[u]_-{\Psi_{g'}^2\circ (j^{\times 2})_*} \\
H_{*-d+3}(K\times K,\Delta_K) \ar[u]_-{(f^{\times 2})_*} \ar[r]^-{\delta_K} & H_{*-2d+4}((K\times K,\Delta_K)^{\times 2}) \ar[u]_-{(f^{\times 4})_*} ,
}\end{split}
\end{align}
The top square is from Theorem \ref{thm-diagram}, the middle square is from Theorem \ref{thm-morse-singular} and the bottom square for $f\colon K\to K' \colon q \mapsto q+ \sigma(q)$ is from Theorem \ref{thm-cobordism} and Example \ref{ex-isotopy}.
Moreover, all of the vertical maps are isomorphisms. See Theorem \ref{thm-isom-LM} for $\Phi_g$ and Proposition \ref{prop-Morse-singular} for $\Psi_g$ and $\Psi^2_{g'}$.

\begin{thm}\label{thm-Leg-delta-K}
Let $K_0$ and $K_1$ be compact connected submanifolds in $\R^n$ with $\codim K_i= d_i \geq 4$ for $i \in \{ 0,1\}$.
If $\Lambda_{K_0}$ and $\Lambda_{K_1}$ are Legendrian isotopic in $U^*\R^n$, then there exists isomorphisms for $m\in \{1,2\}$
\[ \Theta^m \colon H_{*- m d_0}((K_0\times K_0 ,\Delta_{K_0})^{\times m}) \to H_{*- m d_1}(( K_1 \times K_1 ,\Delta_{K_1})^{\times m})\]
such that the following diagram commutes:
\[\xymatrix{
 H_{*-d_0}(K_0\times K_0,\Delta_{K_0}) \ar[r]^-{\delta_{K_0}} \ar[d]^-{\Theta^1} & H_{*-2d_0+1}((K_0\times K_0,\Delta_{K_0})^{\times 2}) \ar[d]^-{\Theta^2}\\ 
 H_{*-d_1}(K_1\times K_1,\Delta_{K_1}) \ar[r]^-{\delta_{K_1}} & H_{*-2d_1+1}((K_1\times K_1,\Delta_{K_1})^{\times 2}) .
}\]
Moreover, via the K\"{u}nneth formula $H_*((K_i\times K_i,\Delta_{K_i})^{\times 2}) \cong H_*(K_i\times K_i,\Delta_{K_i})^{\otimes 2}$ for $i \in \{0,1\}$, $\Theta^2$ coincides with
$\Theta^1\otimes \Theta^1$.
\end{thm}
\begin{proof}
For $i\in \{0,1\}$, we choose an admissible triple $(K'_i,g_i,g_i)$ and $J_i\in \mathcal{J}_{U^*B^n_l}$ as above, where $l\in \Z_{\geq 1}$ is sufficiently large, so that the diagram (\ref{diagram-coproduct-homology-2}) commutes for $K=K_i$. Combining with Theorem \ref{thm-Leg-invariant}, we get the following commutative diagram:
\[
\xymatrix@R=20pt{
H_{*-d_0+3}(K_0\times K_0, \Delta_{K_0})  \ar[r]^-{\delta_{K_0}} & H_{*-2d_0+4}((K_0 \times K_0, \Delta_{K_0})^{\times 2}) \\
HL_{*+1}(\Lambda_{K'_0})\ar[u]  \ar[r]^-{(\delta_{J_0})_*} \ar[d]^-{\varphi_*} & HL^2_{*}(\Lambda_{K'_0}) \ar[u]  \ar[d]^-{(\varphi\otimes \varphi)_*} \\
HL_{*+1}(\Lambda_{K'_1}) \ar[r]^-{(\delta_{J_1})_*} \ar[d] & HL^2_{*}(\Lambda_{K'_1}) \ar[d] \\
H_{*-d_1+3}(K_1 \times K_1, \Delta_{K_1}) \ar[r]^-{\delta_{K_1}} & H_{*-2d_1+4}((K_1\times K_1, \Delta_{K_1})^{\times 2}) .
}
\]
Then, $\Theta^1$ (resp. $\Theta^2$) is defined as the composition of the left (resp. the right) vertical isomorphisms. From their constructions, it is easy to see that $\Theta^2$ coincides with $\Theta^1\otimes \Theta^1$ via the K\"{u}nneth formula.
\end{proof}

Moreover, we can compute the dimension of the lower degree part of the Legendrian contact homology of unit conormal bundles.
\begin{thm}\label{thm-LCH-deltaK}
Let $K$ be a compact connected submanifold of $\R^n$ of codimension $d \geq 4$ such that every $x \in \mathcal{C}(K)$ is non-degenerate. Then, for every $p\in \Z$ with $1\leq p\leq 3d-7$,
\begin{align*}
 \dim_{\Z/2} \LCH_p(\Lambda_K) = & \dim_{\Z/2} \ker \left(\delta_K \colon H_{p-d+2}(K\times K,\Delta_K) \to H_{p-2d+3}((K\times K,\Delta_K)^{\times 2})  \right) \\
 & + \dim_{\Z/2} \coker \left( \delta_K \colon H_{p-d+3}(K\times K,\Delta_K) \to H_{p-2d+4}((K\times K,\Delta_K)^{\times 2})  \right).
\end{align*}
\end{thm}
\begin{proof}
For every $a\in \mathcal{R}(\Lambda_K)$, $|a| \geq d-2$ by Proposition \ref{prop-correspond}. Therefore, the computation by Proposition \ref{prop-LCH-dim} of $\LCH_p(\Lambda_K)$ can be applied if $1\leq p\leq 3(d-2)-1 = 3d-7$. From the diagram (\ref{diagram-coproduct-homology-2}), the equation  is proved.
\end{proof}

\subsection{Non-trivial examples}

Throughout this subsection, we denote the unit sphere in $\R^m$ by $S^{m-1} \coloneqq \{v \in \R^m \mid |v| =1\}$ for any $m\in \Z_{\geq 1}$.

\subsubsection{Connected sum with embedded sphere}\label{subsubsec-connected-sum}
Suppose that we are given the following data:
\begin{itemize}
\item a compact submanifold $K$ of $\R^n$ of codimension $d$.
\item an embedding $f\colon S^{n-d} \to \R^n$. Let us write $S\coloneqq f(S^{n-d})$.
\item $p_0\in K$ and $p_1 = f(0,\dots ,0,-1) \in S$ such that $p_0\neq p_1$.
\item an embedding $\gamma \colon [0,1] \to \R^n$ such that:
\begin{itemize}
\item[(i)] $p_0 = \gamma(0) \in K$ and $p_1 = \gamma(1) \in S$.
\item[(ii)] $\frac{d \gamma}{dt}(0) \in (T_{p_0} K)^{\perp}$, $\frac{d \gamma}{dt}(1) \in (T_{p_1} S)^{\perp}$.
\item[(iii)] $\gamma(t) \notin K \cup S $ for every $t\in (0,1)$.
\end{itemize}
\end{itemize}
We construct a \textit{connected sum} of $K$ and $S$ along $\gamma$.

We choose a family of linearly independent vectors $(b_1(t),\dots ,b_{n-d}(t))$ depending smoothly on $t\in [0,1]$ such that:
\begin{itemize}
\item  $\la b_i(t) , \frac{d \gamma}{dt}(t) \ra = 0$ for $i=1,\dots ,n-d$ and $t\in [0,1]$.
\item  $\{b_i(0) \mid i=1,\dots ,n-d\}$ spans $T_{p_0}K$ and $\{b_i(1) \mid i=1,\dots ,n-d\}$ spans $T_{p_1} S$.
\end{itemize}

Let us take $r_0>0$ and a $C^{\infty}$ embedding
\[[r_0,2r_0] \to (0,2r_0] \times [0,1] \colon r \mapsto (a_1(r), a_2(r))\]
such that $(a_1(r), a_2(r)) = (3 r_0-r  ,1)$ near $r=r_0$ and $(a_1(r),a_2(r)) =(r, 0)$ near $r=2r_0$. See the left picture in Figure \ref{figure-curve}.

\begin{figure}
\centering
\begin{overpic}[height=5cm]{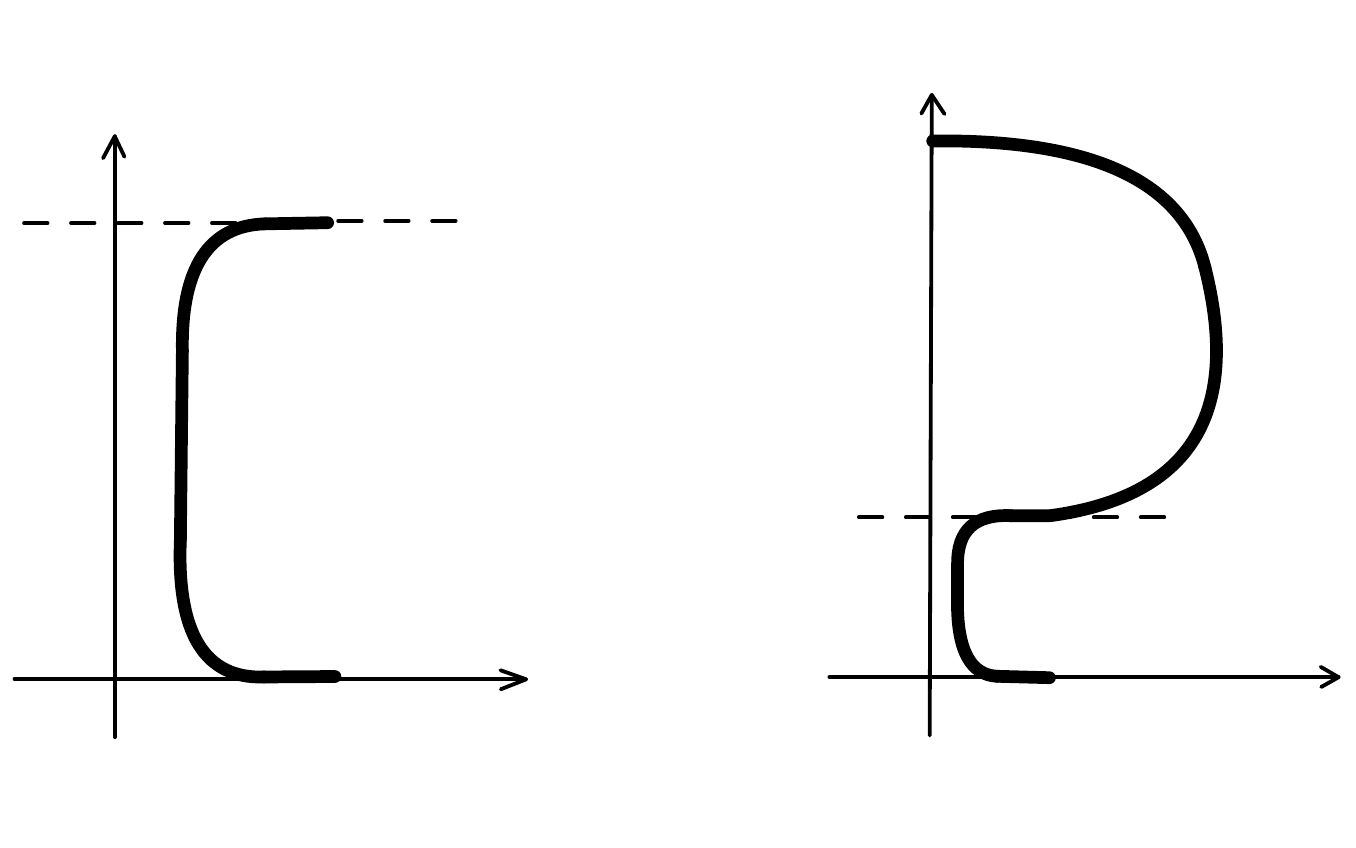}
\put(22,8){$2r_0$}
\put(5.5,41.5){$1$}
\put(5,8){$0$}
\put(76,8){$2r_0$}
\put(65,8){$0$}
\put(65,20){$1$}
\put(55,50.5){$2c+1$}
\end{overpic}
\caption{The left-hand side draws the image of the curve $\{(a_1(r),a_2(r)) \mid r_0\leq r\leq 2r_0\}$. The right-hand side draws the image of the curve  $\{(a'_1(r),a'_2(r)) \mid 0\leq r\leq 2r_0\}$ in the proof of Lemma \ref{lem-isotopy-connected-sum}.}\label{figure-curve}
\end{figure}

By $C^{\infty}$ isotopies, we may change $K$ and $S$ near $p_0$ and $p_1$ so that they are flat near these points.
For sufficiently small $r_0>0$, we take $C^{\infty}$ charts $\varphi$ for $K$ and $\psi$ for $S$ such that:
\begin{itemize}
\item $\varphi \colon \{u\in \R^{n-d} \mid |u|<3r_0 \}\to K$ with $\varphi(0)=p_0$
is determined by
\begin{align}\label{phi-flat}
\varphi (u)= p_0+ \sum_{i=1}^{n-d}u_i \cdot b_i(0)
\end{align}
for every $u=(u_1,\dots ,u_{n-d})$ with $|u| < 3r_0$.
\item $\psi \colon \{u\in \R^{n-d} \mid |u|<3r_0 \} \to S$ satisfies $\Im \psi \sqcup \{p_1\}= S$ and
\[ \psi  (u)= p_1+ (3r_0-|u|) \cdot\sum_{i=1}^{n-d} \frac{u_i}{|u|} \cdot b_i(1) 
 \]
for every $u=(u_1,\dots ,u_{n-d}) \in \R^{n-d}$ with $r_0\leq |u| < 3 r_0$.
\end{itemize}
Since $p_0\neq p_1$, we may assume that $\Im \varphi \cap \Im \psi =\emptyset$.

We define a map $\iota_f\colon K \to \R^n$ explicitly by
\[ \iota_f (q) \coloneqq \begin{cases} q& \text{ if }q \in K \setminus \{\varphi(u) \mid |u| \leq 2 r_0 \} , \\
 \displaystyle{ \gamma( a_2(|u|) ) + a_1(|u|) \cdot  \sum_{i=1}^{n-d} \frac{ u_i}{|u|} \cdot b_i(a_2(|u|)) }
  & \text{ if } q=\varphi(u) \text{ with } r_0\leq |u| \leq 2r_0 , \\
\psi(u) & \text{ if }q= \varphi(u) \text{ with } |u| \leq r_0. 
\end{cases}\]
Then, $\iota_f$ is an immersion when $r_0>0$ is sufficiently small. Moreover, if $K \cap S =\emptyset$, then $\iota_f$ is an embedding.

\begin{lem}\label{lem-isotopy-connected-sum}
Suppose that $T_{p_0}K = T_{p_1}S$ as subspaces of $\R^n$ and  $\gamma$ is a linear path, that is, $\gamma(t) =(1-t)\cdot p_0 + t\cdot p_1 $ and $p_1-p_0\in (T_{p_0}K)^{\perp} = (T_{p_1}S)^{\perp}$.
Moreover, suppose that there exists a linear injection $l\colon \R^{n-d+1} \to \R^m$ such that:
\begin{itemize}
\item $l(1,0,\dots ,0) = \frac{p_1-p_0}{|p_1-p_0|} $ and $l(\{0\}\times \R^{n-d}) = T_{p_0}K = T_{p_1}S$. 
\item For every $v\in S^{n-d}\subset \R^{n-d+1}$,
\begin{align}\label{emb-f}
f(v) = p_1 + c\cdot l(v + (1,0,\dots ,0)) ,
\end{align}
where $c>0$ is a constant. 
\item $K$ is disjoint from $\{p_1 + c\cdot l(v + (1,0,\dots ,0)) \mid |v|\leq 1+\epsilon \}$, where $\epsilon$ is a constant such that $0 < \epsilon < c^{-1}|p_1-p_0|$.
\end{itemize}
Then, $\iota_f$ is an embedding and the submanifold $\iota_f(K)$ is isotopic $K$ in $\R^n$.
\end{lem}

\begin{proof}
Without loss of generality, we may assume that $T_{p_0}K = T_{p_1}S= \R^{n-d}\times \{0\}$ and $\frac{p_1-p_0}{|p_1-p_0|}= (0,\dots ,0,1) \in \R^n$.
In addition, we may assume that the linearly independent vectors $(b_1(t),\dots ,b_{n-d}(t))$ used in the construction of $\iota_f$ is constant on $t\in [0,1]$.
Indeed, given any $(b_1(t),\dots ,b_{n-d}(t))$ satisfying the required conditions, we take a $C^{\infty}$ family of linear injections $(l_s)_{s\in [0,1]}$ such that $l_0=l$ and for every $s\in [0,1]$, $l_s(1,0,\dots ,0)= \frac{p_1-p_0}{|p_1-p_0|}$ and $l_s(\{0\}\times \R^{n-d}) = \Span ( b_1(1-s),\dots ,b_{n-d}(1-s))$.
Replacing $l$ in (\ref{emb-f}) by $l_s$, we define embeddings $f_s$ for all $s\in [0,1]$ and an isotopy $(\iota_{f_s})_{s\in [0,1]}$ such that $\iota_{f_0} = \iota_f$ and
$\iota_{f_1}$ is defined from linearly independent vectors constant on $[0,1]$.

Via an affine transformation on $\R^n$, we may assume that $\gamma(t) = (0,\dots ,0,t)$ for every $t\in [0,1]$ and $(b_1(t),\dots ,b_{n-d}(t))= (b_1(0),\dots ,b_{n-d}(0))$ is a standard basis of $\R^{n-d}$.
Then, on $\{u\in \R^{n-d} \mid |u|\leq 2r_0 \}$, $\iota_f\circ \varphi(u)$ has the form
\begin{align}\label{iotaf-phi}
\iota_f\circ \varphi(u) = \left( a'_1(|u|) \frac{u}{|u|}, (0,\dots,0, a'_2(|u|)\right) \in \R^{n-d}\times \R^d
\end{align}
for an embedded curve
\[(a'_1,a'_2)\colon [0,2r_0] \to (0,2r_0]\times [0,1] \cup \{ z\in \R^2 \mid |z-(0,c+1)|\leq c(1+\epsilon) \}\]
which is an extension of $(a_1,a_2)$ used in the construction of $\iota_f$.
It is drawn as in the right-hand side of Figure \ref{figure-curve}.
Via a $C^{\infty}$ isotopy near $0\in [0,2r_0]$, we change $(a'_1,a'_2)$ to satisfy the two conditions:
\begin{itemize}
\item Near $r=0$, $a'_2(r)$ is constant and  $a'_1(r) =r$.
\item Near $r=2r_0$, $(a'_1(r) ,a'_2(r)) = (a_1(r),a_2(r)) = (r,0)$.
\end{itemize}
We take a $C^{\infty}$ isotopy of embedded curves
\[(a'_{1,s},a'_{2,s})\colon [0,2r_0] \to (0,2r_0]\times [0,1] \cup \{ z\in \R^2 \mid |z-(0,c+1)|\leq c(1+\epsilon) \},\]
preserving the above two conditions, such that $(a'_{1,0},a'_{2,0}) = (a'_{1},a'_{2})$ and $(a'_{1,1},a'_{2,1}) = (r,0)$ for every $r\in [0,2r_0]$.
If we replace $(a'_1,a'_2)$ in (\ref{iotaf-phi}) by $(a'_{1,s},a'_{2,s})$, then we obtain an embedding which coincides with $\varphi$ near $\{u\in \R^{n-d} \mid |u| = 2r_0\}$ and disjoint from $K\setminus \{\varphi(u) \mid |u|\leq 2r_0\}$, and moreover when $s=1$, it agrees with $\varphi$ on $\{u\in \R^{n-d} \mid |u| \leq 2r_0\}$ given by (\ref{phi-flat}).
This shows that $\iota_f$ is isotopic to the inclusion map $K \to R^n \colon x \mapsto x$.
\end{proof}

\subsubsection{Construction of two submanifolds $K_0,K_1$}\label{subsubsec-const-K0K1}
Fix $n,d\in \Z$ such that $\frac{n}{2} \geq d \geq 2$.
Let $M$ be a compact connected submanifold of $\R^{n-d}$ of codimension $d-1$.
We change $M$ by $C^{\infty}$ isotopy and assume the following:
\begin{itemize}
\item $\max_{w \in M} |w| \leq 1$.
\item Consider the function $h \colon M\to \R \colon (w_1,\dots ,w_{n-d}) \mapsto w_{n-d}$. Then $\max_{w\in M} h(w) = 1$ and $h^{-1}(1) = \{ (0,\dots ,0,1) \}$.
\end{itemize}
We take the product with $S^{d-1}\subset \R^d$ to define a submanifold of  $\R^n$ by
\[ K_0\coloneqq  S^{d-1}\times M \subset \R^{n}= \R^{d} \times \R^{n-d}.\]
It is a compact connected submanifold of codimension $d$. See Figure \ref{figure-K0}.

\begin{figure}
\centering
\begin{overpic}[height=5cm]{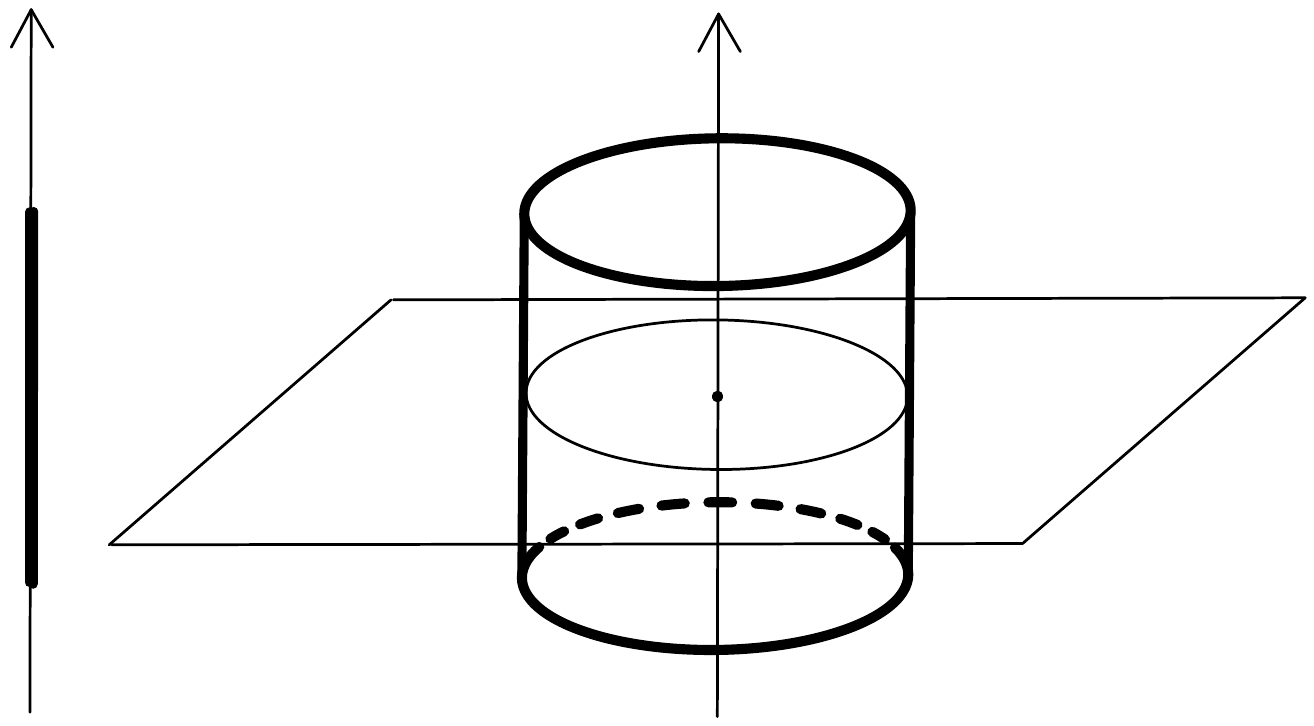}
\put(-3,25){$M$}
\put(5,52){$\R^{n-d}$}
\put(57,52){$\R^{n-d}$}
\put(15,15){$\R^d$}
\put(57.5,20.5){$S^{d-1}$}
\put(51.5,23.5){$0$}
\put(67,43){$K_0$}
\put(90,48){$\R^n$}
\end{overpic}
\caption{Let us depict $M$ by a thick line as in the left-hand side, although $M$ is a compact submanifold of $\R^{n-d}$ without boundary. Then, $K_0=S^{d-1}\times M$ can be illustrated as in the right-hand side. We remark that $M$ does not necessarily contain $0\in \R^{n-d}$.}\label{figure-K0}
\end{figure}

Let $e_1 \coloneqq (1,0,\dots ,0)\in \R^d$ and $e_{n-d} \coloneqq (0,\dots ,0,1) \in \R^{n-d}$.
We choose a $C^{\infty}$ function
$\mu \colon [0,3] \to [0,1]$ such that $\mu (r) = \begin{cases} \sqrt{9-r^2} & \text{ if } \sqrt{\frac{17}{2}} \leq r \leq 3 , \\ 1 & \text{ if } r\leq \sqrt{8} .\end{cases}$ Then we define an embedding $f \colon S^{n-d}\to \R^d\times \R^{n-d}$ by
\begin{align}\label{embedding-S_R}
f  (y_0,y_1,\dots,y_{n-d}) \coloneqq
\begin{cases}((1+3 y_0)\cdot e_1 ,(3y_1,\dots ,3 y_{n-d})) & \text{ if }y_0\geq 0 , \\
((1-\mu (3\sqrt{1- y_0^2}) ) \cdot e_1 , (3y_1,\dots ,3 y_{n-d})) & \text{ if }y_0\leq 0 . \end{cases}
\end{align}
Its image $S=f(S^{n-d})$ is illustrated in the left-hand side of Figure \ref{figure-K1}.
We note that $S$ is disjoint from $K_0$. Indeed, for any $((1+r)\cdot e_1,w) \in S$ with $-1 \leq r\leq 3$, if we suppose that $(1+r)\cdot e_1 \in S^{d-1}$ and $w\in M$, then $r=0$ and $|w|=3$ follows, which contradicts with the assumption on $M$.

We consider the connected sum of $K$ and $S$ along the segment
\begin{align}\label{segment-gamma}
 \gamma \colon [0,1]\to \R^n \colon s \mapsto (e_1 , (1+2s) \cdot e_{n-d} ).
 \end{align}
The boundary points are $p_0\coloneqq (e_1, e_{n-d}) \in K_0$ and $p_1 \coloneqq  ( e_1, 3 e_{n-d} ) \in S$. 
The right-hand side of Figure \ref{figure-K1} illustrates $K_0$, $S$ and $\gamma$.

\begin{figure}
\centering
\begin{overpic}[height=6.5cm]{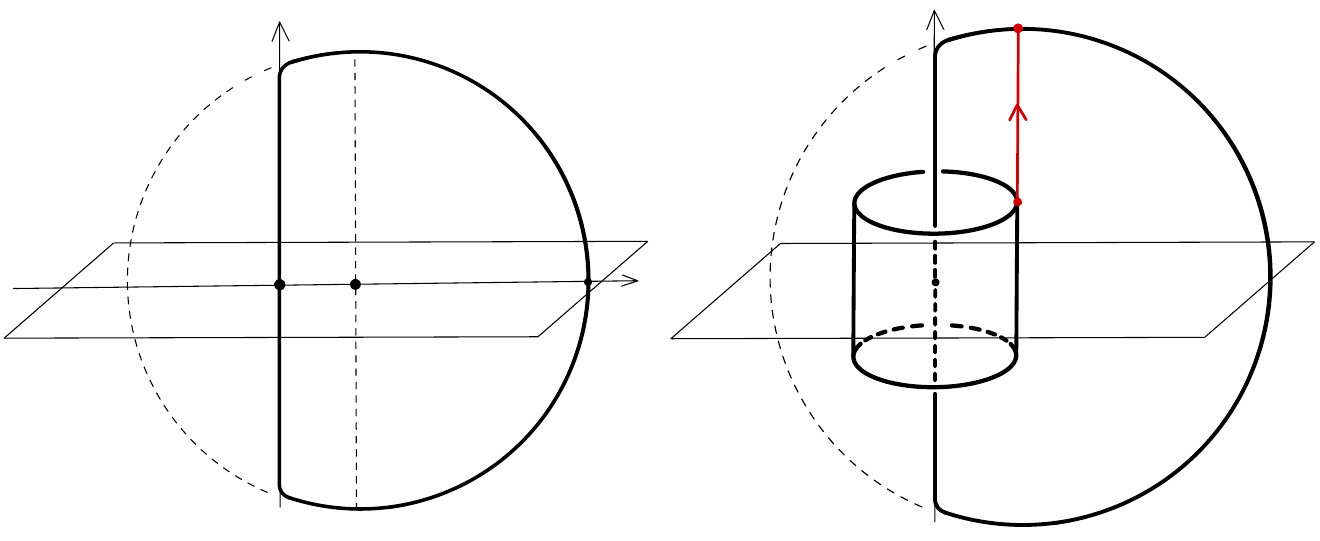}
\put(45.5,17.5){$\R\cdot e_1$}
\put(15.5,38){$\R^{n-d}$}
\put(4,16){$\R^d$}
\put(19,17){$0$}
\put(27.5,17){$(e_1,0)$}
\put(41,31){$S$}
\put(93,31){$S$}
\put(76,40){$p_1$}
\put(78,25.5){$p_0$}
\put(78,32){$\gamma$}
\put(64,28){$K_0$}
\end{overpic}
\caption{As described in the left-hand side, $S$ is an ($n-d$)-dimensional sphere embedded in $(\R \cdot e_1) \times \R^{n-d}$.
$K_1$ is obtained by a connected sum of $K_0$ and $S$ along $\gamma$ illustrated in the right-hand side.
}\label{figure-K1}
\end{figure}

We choose an open ball $B$ in $M$ and $a\in (0,1)$ such that $h^{-1}([a,1]) \subset B$ and $1\in [a,1]$ is the only critical value (the maximum value) of $h\colon M\to \R$ in $[a,1]$.
By the connected sum of $K_0$ and $S$ along $\gamma$, which we have given in Subsection \ref{subsubsec-connected-sum}, we obtain an embedding
$\iota_{f} \colon K_0\to \R^n$. From its construction, we have open subsets of $K_0$ defined by
\[\begin{array}{cc}
U \coloneqq \{ \varphi (u) \mid |u|< 3r_0 \},& V \coloneqq \{ \varphi(u) \mid |u|< r_0\},
\end{array}\]
with the following properties if $r_0>0$ is sufficiently small:
\begin{itemize}
\item $p_0 \in V \subset U \subset S^{d} \times B$.
\item $\iota_{f} (q)= q$ for every $q\in K_0\setminus U$.
\item $\iota_f(U \setminus V)$ is contained in a small neighborhood of $\gamma([0,1])$. In particular, it is contained in $\R^d \times \{(w_1,\dots ,w_{n-d}) \in \R^{n-d} \mid w_{n-d}>a \}$.
\item $\iota_f(V) $ is an open subset of $S$ which is the complement of a small neighborhood of $p_1\in S$. In particular, it contains
\[ \{(0, w)\in \R^d\times \R^{n-d} \mid |w| \leq \sqrt{2} \} . \]
\end{itemize}
Let us denote the new submanifold of $\R^n$ by $K_1 \coloneqq \iota_{f}(K_0)$.

\subsubsection{Distinguishing $\Lambda_{K_0}$ and $\Lambda_{K_1}$ up to Legendrian isotopies}\label{subsubsec-Leg-non-isotopic}

We first discuss the Legendrian regular homotopy class and the $C^{\infty}$ isotopy class of $\Lambda_{K_0}$ and $\Lambda_{K_1}$.

\begin{prop}\label{prop-smooth-isotopy}
$\Lambda_{K_0}$ and $\Lambda_{K_1}$ are Legendrian regular homotopic in $(U^*\R^n,\alpha)$.
Moreover, they are  $C^{\infty}$ isotopic in $U^*\R^n$.
\end{prop}
\begin{proof}
We choose a family of $C^{\infty}$ functions $\mu_{s} \colon [0, 3]\to [0,1]$ depending smoothly on $s\in [0,2]$ such that:
\begin{itemize}
\item  $\mu_{0}=\mu$ and $\mu_{2}(r) = \sqrt{9 -r^2}$ for every $r\in [0,3]$.
\item For $s\in [0, 1]$, $\mu_{s}(r)= \begin{cases} \sqrt{9 -r^2} & \text{ if } \sqrt{\frac{17}{2}} \leq r \leq 3, \\ 1+ s & \text{ if } r\leq (2-s)\sqrt{2} .\end{cases}$
\item For $s\in (1 ,2]$, $\mu_s(r) > 2$ for every $r\in [0,\sqrt{2}]$.
\end{itemize}
For every $s\in [0,2]$, we define $f_s \colon S^{n-d} \to \R^n=\R^d\times \R^{n-d}$ by replacing $\mu$ in (\ref{embedding-S_R}) by $\mu_s$.
By a connected sum with $K_0$ along $\gamma$, we obtain $\iota_s\coloneqq \iota_{f_s} \colon K_0\to \R^n$ for every $s\in [0,2]$ such that $\iota_0=\iota_f$ and $(\iota_{s})_{s\in [0,2]}$ is a $C^{\infty}$ family of immersions.
From the definition of $f_s$, $f_s(S^{n-d})$ intersects $K_0$ if and only if $s= 1$. See Figure \ref{figure-immersion}. Therefore, $\iota_{s}$ is an embedding if $s \neq 1$.

We claim that  $\iota_2(K_0)$ is $C^{\infty}$ isotopic to $K_0$ in $\R^n$. Indeed, $f_2(S^{n-d})$ is the boundary of $(n-d+1)$-dimensional disk
\[ \tilde{D} \coloneqq \{ ((r_1+1) e_1+r_2 e_2, w ) \in \R^d\times \R^{n-d} \mid (r_1)^2+(r_2)^2+|w|^2= 9 ,\ r_2\geq 1 \}, \]
which is disjoint from $K_0 \cup \gamma([0,1))$. Here, $e_2=( 0,1,0,\dots ,0)\in \R^{d}$. 
There exists an isotopy $(f_s)_{s\in [2,3]}$ in a neighborhood of $\tilde{D}$ from $f_2$ to $f_3 \colon S^{n-d}\to \R^n$ such that $f_s(0,\dots ,0,-1)= p_1$, $\frac{d\gamma}{dt}(1) \in T_{p_1} (f_s(S^{n-d}))^{\perp}$ and $f_s(S^{n-d}) \cap \gamma([0,1)) =\emptyset$ for every $s\in [2,3]$, and $f_3$ satisfies the conditions of Lemma \ref{lem-isotopy-connected-sum}.
Then, $(\iota_s)_{s\in [2,3]} \coloneqq (\iota_{f_s})_{s\in [2,3]}$ is an isotopy of embeddings.
By Lemma \ref{lem-isotopy-connected-sum}, there exists an isotopy $(\iota_s)_{s\in [3,4]}$ such that $\iota_3=\iota_{f_3}$ and $\iota_4(K_0) = K_0$.
Therefore, $\iota_2(K_0)$ is isotopic to $ K_0$ as a $C^{\infty}$ submanifold of $\R^n$.

We consider the vector bundle $\widehat{L} \to K_0\times [0,4]\colon ((q,s),p) \to (q,s)$, where 
\[ \widehat{L} \coloneqq \{ ((q,s) ,p) \in K_0 \times [0,4] \times \R^n \mid  \la p, d\iota_s(v) \ra =0 \text{ for every } v\in T_q K_0 \} ,\]
and define a map $F \colon \widehat{L} \to T^*\R^n \colon ((q,s),p) \mapsto (\iota_s(q),p)$.
We take an isomorphism $\Phi \colon L_{K_0} \times [0,4] \to \widehat{L}$ of vector bundles preserving the metrics and define
\[ I_s\colon \Lambda_{K_0} \to U^*\R^n \colon (q,p) \mapsto F\circ \Phi((q,p), s) . \]
Then, $(I_s)_{s\in [0,4]}$ is a $C^{\infty}$ family of Legendrian immersions. For any $s\neq 1$, $I_s$ is an embedding and $I_s(\Lambda_{K_0}) = \Lambda_{\iota_s(K_0)}$.
In particular, it gives a Legendrian regular homotopy from $\Lambda_{K_1}$ to $\Lambda_{K_0}$
and the first assertion is proved.

\begin{figure}
\centering
\begin{overpic}[height=5cm]{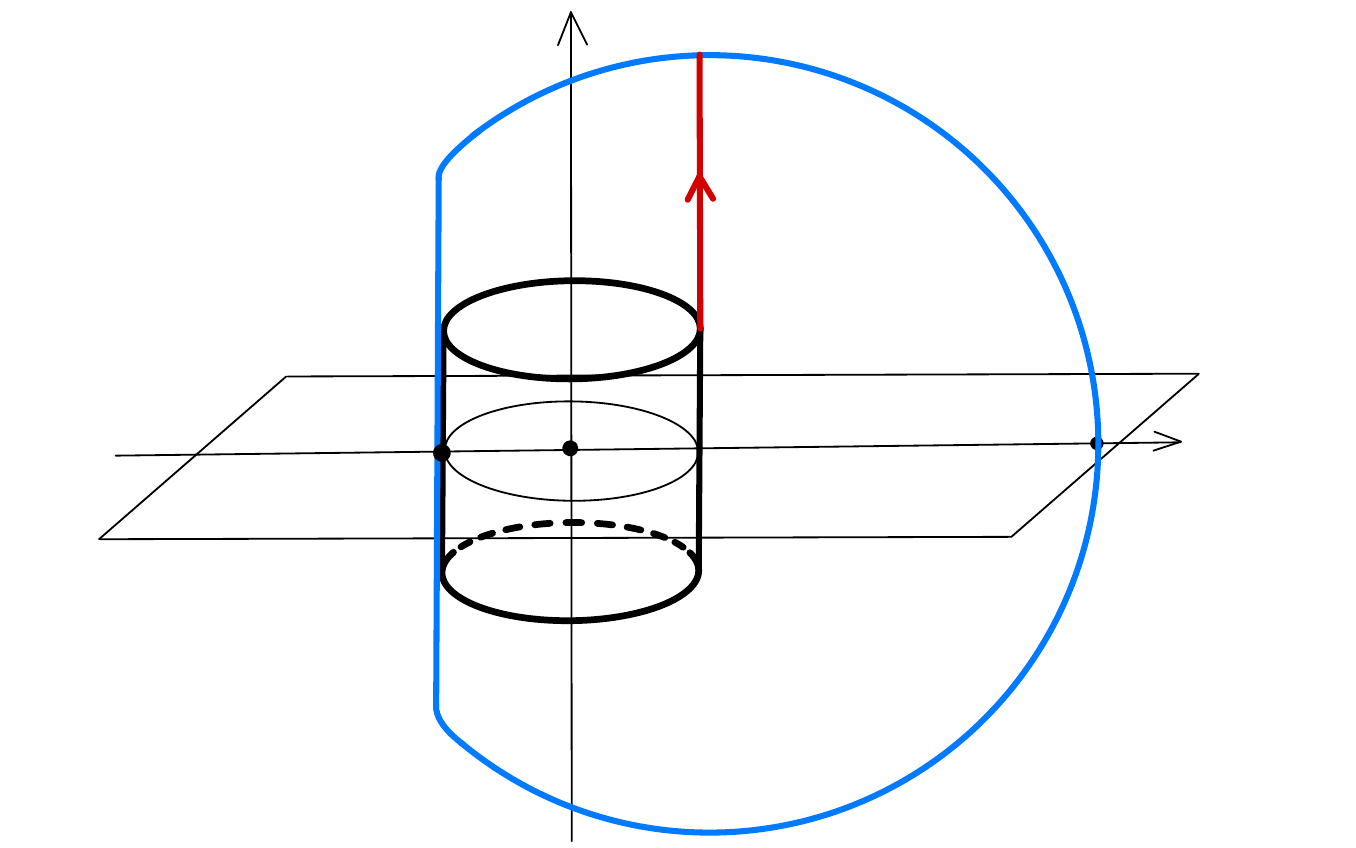}
\put(23.5,26){$-e_1$}
\put(83,25.5){$\R\cdot e_1$}
\put(75,49){$f_1(S^{n-d})$}
\put(53,48){$\gamma$}
\put(35,43){$K_0$}
\end{overpic}
\caption{$\iota_1(K_0)$ is obtained by a connected sum of $K_0$ and $f_1(S^{n-d})$ along $\gamma$. It has a self-intersection along $K_0\cap f_1(S^{n-d}) = \{-e_1\}\times M$.}\label{figure-immersion}
\end{figure}

If $s=1$, $\iota_1(K_0)$ has self-intersections along
$\{ (-e_1 , w) \in \R^d \times \R^{n-d} \mid w\in M \}$ as described in Figure \ref{figure-immersion}. Noting that $\iota_1 (V) $ contains $\{ (-e_1,w) \in \R^d\times \R^{n-d} \mid |w| \leq 1\}$, we define
\[V' \coloneqq V \cap \iota_1^{-1} (\{ (-e_1,w) \in \R^d\times \R^{n-d} \mid |w| \leq 1\}).\] 
Then, for any $q\in V'$,
$I_1(\Lambda_{K_0} \cap U^*_q\R^n)= S^{n-1} \cap (\{0\}\times \R^{n-d})^{\perp} = S^{n-d} \times \{0\}$.

We change $I_s(q,p)$ when $s$ is close to $1$ and $q \in V$ to define an isotopy of $C^{\infty}$ embeddings $(I'_s \colon \Lambda_{K_0} \to U^*\R^n)_{s\in [0,4]}$ so that:
\begin{itemize}
\item $I'_0=I_0$ and $I'_4=I_4$
\item $I'_s(q,p) \in U^*_{\iota_s(q)} \R^n$ for every $s\in [0,4]$ and $(q,p)\in \Lambda_{K_0}$.
\item $I'_s(q,p) = I_s(q,p)$ for every $s\in [0,4]$ if $q \in K_0\setminus V$.
\item For every $q \in V'$,
$I'_1 ( \Lambda_{K_0} \cap U^*_q\R^n) \subset \{ (p_1,\dots ,p_n) \in U_{\iota_1(q)}S^{n-1} \mid p_2=\frac{1}{2} \}$. 
\end{itemize}
Here, we remark that for every $q\in V'$ with $\iota_1(q) \in \{-e_1\}\times M$, the set $\{ (p_1,\dots ,p_n) \in U_{\iota_1(q)}\R^n \mid p_2=\frac{1}{2} \}$ is disjoint from
\begin{align*}
I'_1 (\Lambda_{K_0} \cap U^*_{q'}\R^n) & = \{ ((\iota_1(q) , p) \mid \la p, v \ra=0 \text{ for every }v \in T_{\iota_1(q)} (S^{d-1}\times M) \} \\
& \subset \{(p_1,\dots ,p_n) \in U_{\iota_1(q)} \R^n \mid p_2=\dots =p_d=0\}
\end{align*}
for the unique point $q' \in \iota_1^{-1} (\iota_1(q)) \setminus \{q\}$. Therefore, $\Lambda_{K_1}=I'_s(\Lambda_{K_0})$ is $C^{\infty}$ isotopic to $\Lambda_{K_0}=I'_4(\Lambda_{K_0})$.
\end{proof}

%
By Proposition \ref{prop-TB},  Proposition \ref{prop-smooth-isotopy} and the isomorphism $\Theta^1$ in Theorem \ref{thm-Leg-delta-K} when $d\geq 4$,
$\Lambda_{K_0}$ and $\Lambda_{K_1}$ cannot be distinguished as Legendrian submanifolds of $U^*\R^n$ by the following invariants:
\begin{itemize}
\item The Thurston-Bennequin number when $K_0$ admits a nowhere vanishing vector field.
\item The Legendrian regular homotopy class.
\item The $C^{\infty}$ isotopy class.
\item The isomorphism class of the strip Legendrian contact homology as a $\Z/2$-vector space when $d\geq 4$.
\end{itemize}
When $d\geq 4$ and $M$ satisfies a topological condition, we will prove that $\Lambda_{K_0}$ and $\Lambda_{K_1}$ are not Legendrian isotopic by computing $\delta_{K_0}$ and $\delta_{K_1}$.

\begin{rem}\label{rem-Asp} Fix $x_0\in S^n$ and identify $S^n\setminus \{x_0\}$ with $\R^n$. Then, for $i=0,1$, we may think of $K_i$ as a submanifold of $S^n$. Let $\Omega (S^n\setminus K_i)$ denote the space of loops in the complement $S^n\setminus K_i$ based at $x_0$. By \cite[Theorem 1]{Asp}, if we can show that there is no isomorphism $H_*(\Omega(S^n\setminus K_0);\Z) \to H_*(\Omega (S^n \setminus K_1);\Z)$ preserving the Pontryagin products, then $\Lambda_{K_0}$ is not Legendrian isotopic to $\Lambda_{K_1}$ in $U^*\R^n$. However, the author does not have an idea how to show $H_*(\Omega(S^n\setminus K_0);\Z) \ncong H_*(\Omega (S^n \setminus K_1);\Z)$. \end{rem}

\begin{prop}\label{prop-coprod-zero}
$\delta_{K_0}$ is the zero map
\end{prop}
\begin{proof}
We take a section $K_0\to (TK_0)^{\perp} \colon q \mapsto (q, \sigma(q))$ with $\sup_{q\in K} |\sigma(q)| < \frac{r}{2}$ as Proposition \ref{prop-geometric-shift} such that $\sigma(v,w) = (\frac{r}{4}\cdot v,0)$ for every $(v,w) \in K_0$. We then obtain
\[K_{0,\sigma}\coloneqq \{q+\sigma(q) \mid q\in K_0\} = \{v\in \R^d \mid |v| = 1+ \textstyle{\frac{r}{4}}\} \times  M .\]
This satisfies
\begin{align}\label{K0sigma}
K_{0,\sigma} \cap (\{ v\in \R^d \mid |v|\leq 1 \} \times \R^{n-d}) =\emptyset .\end{align}
On the other hand, for every $(y,\tau) \in (K_0 \times K_0) \times I$, $\ev(y,\tau) $ is contained in $\{ v\in \R^d \mid |v|\leq 1 \} \times \R^{n-d}$.

Any homology class in $H_p(K_0\times K_0,\Delta_{K_0})$ with $\Z/2$-coefficient is written as $c_*([P])$ for some closed $p$-dimensional $C^{\infty}$ manifold $P$ and a $C^{\infty}$ map $c\colon P\to K_0 \times K_0$ (see the proof of Theorem \ref{thm-morse-singular}).
Since $\ev_c(P\times [0,1])$ is disjoint from $K_{0,\sigma}$ by (\ref{K0sigma}), we have $\ev_c^{-1}(K_{0,\sigma})=\emptyset$ and thus $\delta_{K_0}(c_*([P])) =0$ by Proposition \ref{prop-geometric-shift}.
\end{proof}

Next, suppose that $H_k(M;\Z/2)\neq 0$ for some $k\in \Z$ with $0< k <\dim M=n-2d+1$.
Then, by \cite[Th\'{e}or\`{e}me III.2]{T}, there exists a closed $k$-dimensional manifold $Q$ and a $C^{\infty}$ map $b\colon Q\to M$ such that $0 \neq b_*([Q]) \in H_k(M;\Z/2)$.
Let us use the notations in Subsection \ref{subsubsec-const-K0K1}.
Since $\dim Q < \dim M$, we may assume  that $b(Q) \subset M\setminus B$.
We define a homology class $c_*([P])\in H_{2k+d-1}(K_1\times K_1,\Delta_{K_1})$ by the closed manifold
\[ P\coloneqq S^{d-1}\times Q\times Q\]
and the map
\[c \colon P \to K_1\times K_1 \colon (v ,z_1,z_2) \mapsto ( (v, b(z_1) ), (e_1, b(z_2) )).\]
Since $b(Q) \subset M\setminus B \subset h^{-1}((-\infty, a] ) $,
\[ c(P) \subset ( S^{d-1} \times h^{-1}((-\infty, a] ) ) \times  (S^{d-1} \times h^{-1}((-\infty, a] ) ).\]
This implies that for every $(z,\tau) \in P\times I$,
\[ \ev_c(z,\tau) \in \{v \in \R^d \mid |v| \leq 1\} \times \{ w=(w_1,\dots , w_{n-d} )\in \R^{n-d} \mid w_{n-d} \leq a,\ |w|\leq 1\}. \]

We take a section $K_1 \to (TK_1)^{\perp}\colon q\mapsto (q,\sigma(q))$ with $\sup_{q\in K} |\sigma(q)| <\frac{r}{2}$ as Proposition \ref{prop-geometric-shift} which moreover satisfies
\begin{itemize}
\item $\sigma(v,w) \in (\frac{r}{4}\cdot v,0)$ if $(v,w) \notin U$.
\item $\sigma(\iota_{f}(q))=0$ if $q\in V$. 
\end{itemize}
Then, we define an embedding $\iota_{f,\sigma} \colon K_0\to \R^n \colon q \mapsto \iota_f(q) + \sigma(q) $ and a submanifold $K_{1,\sigma} \coloneqq \iota_{f,\sigma} (K_0) = \{q+\sigma(q) \in \R^n \mid q\in K_1\}$. If $\sup_{q\in U\setminus V} |\sigma(\iota_f(q))|$ is sufficiently small, then \begin{align}\label{intersection-K'1}
\begin{split}
& K_{1,\sigma} \cap (\{v\in \R^d \mid |v| \leq 1 \}\times  \{w= (w_1,\dots ,w_{n-d}) \in \R^{n-d} \mid w_{n-d} \leq a ,\ |w| \leq 1\}) \\
= 
& \{0\} \times \{w= (w_1,\dots ,w_{n-d}) \in \R^{n-d} \mid w_{n-d} \leq a , |w|\leq 1 \} .
\end{split}
\end{align}
\begin{lem}\label{lem-intersection-K1}
The conditions of Proposition \ref{prop-geometric-shift} are satisfied for $K=K_1$, $K_{\sigma}=K_{1,\sigma}$ and $c\colon P\to K_1\times K_1$.
Moreover, $\delta_K(c_*([P]))$ is equal to  
$(\sp'_{c})_*([Q\times Q]) \in H_{2k}((K_1\times K_1,\Delta_{K_1})^{\times 2})$, where
\[ \sp'_c \colon Q \times Q \to (K_1\times K_1)\times (K_1\times K_1) \colon (z_1,z_2) \mapsto ( ((-e_1,b(z_1)), 0), (0, (e_1,b(z_2))) ).\]
\end{lem}
\begin{proof}
We check the conditions of Proposition \ref{prop-geometric-shift}. For any $z=(v,z_1,z_2)\in P$,
$q_z=(v,b(z_1))$ and $q'_z=(e_1,b(z_2))$ are contained in $S^{d-1} \times h^{-1}((-\infty,a])$, which is disjoint from $K_{1,\sigma}$ by (\ref{intersection-K'1}).
Therefore $\{q_z,q'_z\}\cap K_{1,\sigma} =\emptyset$.

In order to check the condition about transversality,
we observe the intersection of $K_{1,\sigma}$ and the map
\[\begin{array}{rl}
\ev_c \colon P\times (0,1) \to \R^n \colon ((v,z_1,z_2),\tau) \mapsto & (1-\tau) \cdot (v,b(z_1)) + \tau\cdot (e_1,b(z_2))  \\
 = & ((1-\tau)\cdot v + \tau \cdot e_1 , (1-\tau)\cdot b(z_1) + \tau \cdot b(z_2) ) .
\end{array}\]
One can refer to the middle picture of Figure \ref{figure-intersection}.
Its image $\ev_c(P\times (0,1))$ is contained in
\[ \{v\in \R^d \mid |v|\leq 1 \} \times \{w=(w_1,\dots ,w_{n-d})\in \R^{n-d}\mid  w_{n-d}\leq a,\ |w|\leq 1 \}.\]
It is disjoint from $\iota_{f,\sigma}(U\setminus V)$, which is a subset of $\R^d \times \{(w_1,\dots ,w_{n-d}) \in \R^{n-d} \mid w_{n-d} > a\}$.
It is also disjoint from $\iota_{f,\sigma}(K_0\setminus U)$, which is a subset of $\{v\in \R^d \mid |v|=1+\frac{r}{4}\} \times M$.

\begin{figure}
\centering
\begin{overpic}[height=5.8cm]{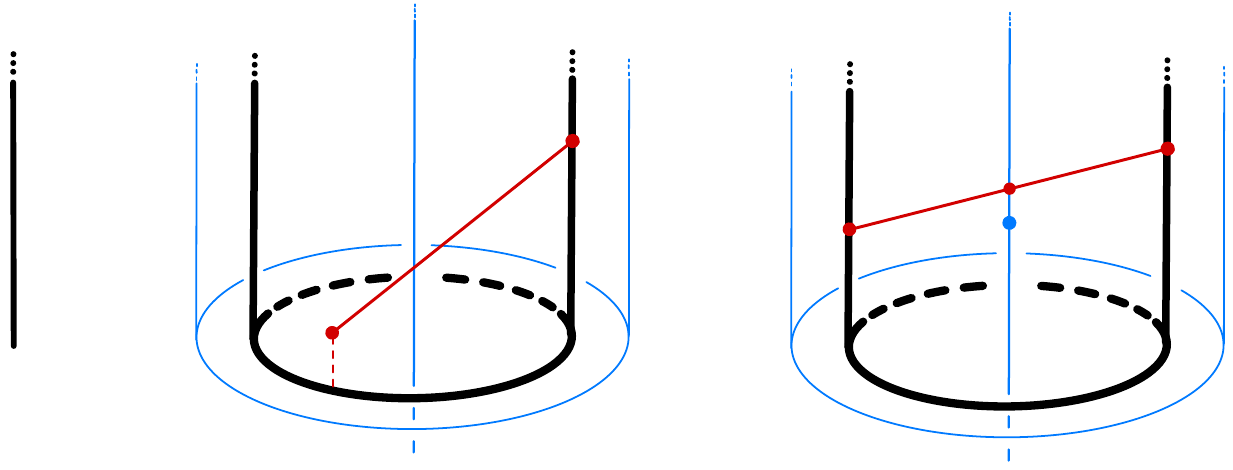}
\put(2,20){$M\setminus B$}
\put(11,15){\color{blue}$K_{1,\sigma}$}
\put(21.5,25){$K_1$}
\put(35,26.5){$(e_1,b(z_1))$}
\put(32,9.5){\color{white}\Huge$\bullet$}
\put(32,8){\color{white}\Huge$\bullet$}
\put(27.5,9.5){$(u,b(z_2))$}
\put(62.5,19){\color{white}\Huge$\bullet$}
\put(62.5,17.5){\color{white}\Huge$\bullet$}
\put(55.5,19){$(-e_1,b(z_2))$}
\put(79.5,18.5){$0$}
\put(83.5,26.5){$(e_1,b(z_1))$}
\end{overpic}
\caption{The linear path from $(e_1,b(z_1))$ to $(u,b(z_2))$ in the middle picture intersects $K_{1,\sigma}$ if and only if $u=-e_1$ as described in the right-hand side.}\label{figure-intersection}
\end{figure}

It remains to consider the intersection $\ev_c$ with $\iota_{f,\sigma}(V) = \iota_f(V)$, which is contained in $S$. Take any $(v,z_1,z_2) \in \ev_c^{-1}(S) $.
Then, we have
$(1-\tau)\cdot u +\tau \cdot e_1 =0$, which implies that $u=-e_1$ and $\tau=\frac{1}{2}$. See Figure \ref{figure-intersection}.
This is a transversal intersection since
\begin{align*}
   &(d \ev_c) (T_{-e_1}S^d \times \mathbb{O} \times T_{\frac{1}{2}}(0,1) )= \R^d \times \{0\}, \\
   &T_{\ev_c((-e_1,z_1,z_2),\frac{1}{2})} \iota_{f}(V)= \{0\}\times \R^{n-d},
   \end{align*}
where $\mathbb{O}$ is the zero subspace of $ T_{z_1}Q\times T_{z_2}Q$.

Therefore, $P$ and $c$ satisfy the conditions of Proposition \ref{prop-geometric-shift} for $K=K_1$ and $K_{\sigma} = K_{1,\sigma}$.
We obtain that $\ev_c^{-1}(K_{1,\sigma}) = \{-e_1\}\times Q\times Q$ and
\[ \sp_{c,\sigma} \colon  Q\times Q \cong \ev_c^{-1}(K_{1,\sigma}) \to (K_1\times K_1)^{\times 2}\]
which maps $(z_1,z_2) \in Q\times Q$ to
\[ \textstyle{ \left( ( (-e_1,b(z_1)), (0,\frac{b(z_1)+b(z_2)}{2} ) ) ,  ( (0,\frac{b(z_1)+b(z_2)}{2} ), ( e_1,b(z_2)) ) \right) } . \]
By Proposition \ref{prop-geometric-shift}, $\delta_{K_1}(c_*([P])) =( \sp_{c,\sigma})_* ([Q\times Q]) $.
Since the disk $ \{(0,w)\in \R^d \times \R^{n-d} \mid |w|\leq 1 \}$ in $K_1$ is contractible to the point $0\in K_1$, $\sp_{c,\sigma}$ is homotopic to the map
\[ \sp'_c \colon Q\times Q \to (K_1\times K_1)^{\times 2} \colon (z_1,z_2) \mapsto ( ((-e_1,b(z_1)), 0), (0, (e_1,b(z_2))) ).\]
Therefore, $(\sp_{c,\sigma})_*([Q\times Q] )= (\sp'_c)_*([Q\times Q])$.
\end{proof}

\begin{prop}\label{prop-coprod-nonzero}
If $H_k (M)\neq 0$ for some $ k\in \Z$ with $1\leq k\leq n-2d$, then
\[ \delta_{K_1} \colon H_{2k+d-1} (K_1\times K_1,\Delta_{K_1}) \to H_{2k} ((K_1\times K_1,\Delta_{K_1})^{\times 2}) \]
is not the zero map.
\end{prop}
\begin{proof}
By Lemma \ref{lem-intersection-K1}, it suffices to show that $(\sp'_c)_*([Q\times Q]) \neq 0$.
Consider a map between pairs of spaces
\[ p_M \colon ( K_1\times K_1,\Delta_{K_1})  \to (M \times M , \Delta_M) \colon (\iota_{f}(u,v) , \iota_f(u',v') )\mapsto (v,v') .\]
This induces
\[\begin{array}{rl}
 (p_M^{\times 2})_*\colon H_{2k}((K_1\times K_1,\Delta_{K_1})^{\times 2}) \to & H_{2k}((M\times M,\Delta_M)^{\times 2}) \\
 \cong & \displaystyle{ \bigoplus_{k_1+k_2=2k} H_{k_1}(M\times M,\Delta_M) \otimes H_{k_2}(M\times M,\Delta_M) . }
 \end{array}\]
By this map, $(\sp'_c)_*([Q\times Q])$ is mapped to
\[ (b_1)_* ([Q]) \otimes (b_2)_*([Q]) \in  H_k(M\times M,\Delta_M) \otimes H_k(M\times M,\Delta_M) ,\]
where $b_1\colon Q\to M\times M \colon z \mapsto (b(z),q_0)$ and $b_2\colon Q\to M\times M \colon z \mapsto (q_0,b(z))$ for $q_0 \in M$ determined by $\iota_f(e_1,q_0) =0 \in \R^n$.
The projections $\pr_i \colon M \times M \to M \colon (z_1,z_2) \mapsto z_i$ for $i=1,2$ defines a map
\[(\pr_1)_* - (\pr_2)_* \colon H_k(M\times M) \to H_k(M)  \]
From the exact sequence for $(M\times M,\Delta_M)$ as (\ref{ex-seq-diagonal}), this induces a well-defined map
from $H_k(M\times M,\Delta_M)$ to $H_k(M)$, by which $(b_i)_*([Q])$ is mapped to $b_*([Q]) \in H_k(M)$ for $i=1,2$. Here, note that $k\geq 1$.
Since $0\neq b_*([Q]) \in H_k(M)$, the assertion follows.
\end{proof}

We obtain the main result of this section.

\begin{thm}\label{thm-Leg-non-isotopic}
If $d\geq 4$ and $H_k(M) \neq 0$ for some $k\in \Z$ with $1\leq k\leq n-2d$, then $\Lambda_{K_1}$ is not Legendrian isotopic to $\Lambda_{K_0}$ in $U^*\R^n$.
\end{thm}
\begin{proof}
This assertion directly follows from Theorem \ref{thm-Leg-delta-K}, Proposition \ref{prop-coprod-zero} and Proposition \ref{prop-coprod-nonzero}.
\end{proof}

In a more special case, $\Lambda_{K_0}$ and $\Lambda_{K_1}$ are distinguished by the Legendrian contact homology.
\begin{thm}\label{thm-LCH-conormal}
If $d\geq 4$ and $H_k(M) \neq 0$ for some $k\in \Z$ with $1\leq k \leq \min \{ n-2d, \frac{d-3}{2} \}$, then
\[\dim_{\Z/2} \LCH_{2k +2d-4}(\Lambda_{K_0}) > \dim_{\Z/2} \LCH_{2k+2d-4}(\Lambda_{K_1}). \]
\end{thm}
\begin{proof}
By Theorem \ref{thm-LCH-deltaK}, for $i=0,1$
\begin{align}\label{ineq-LCH-deltaK} \dim_{\Z/2} \LCH_p(\Lambda_{K_i}) \leq \dim_{\Z/2}  H_{p-d+2}(K_i\times K_i,\Delta_{K_i}) + \dim_{\Z/2} H_{p-2d+4}( (K_i\times K_i,\Delta_{K_i} )^{\times 2})  
\end{align}
if $1\leq p \leq 3d-7$.
For $i=0$, the equality holds because $\delta_{K_0}$ is the zero map. On the other hand, for $i=1$,  
$\delta_{K_1} \colon H_{p-d+2}(K_1\times K_1,\Delta_{K_1}) \to H_{p-2d+3}( (K_1\times K_1,\Delta_{K_1})^{\times 2} ) $
is not the zero map if $p-d+2 = 2k -d +1 $, i.e. $p=2k+ 2d-3$.
Therefore, if $ 2k \leq d-3$, then
the equality of (\ref{ineq-LCH-deltaK}) does not hold for $p= 2k+2d-3$.
\end{proof}

\appendix

\section{Appendix}\label{appendix}

\subsection{Morse index and Conley-Zehnder index}\label{subsec-A1}

Fix an arbitrary Reeb chord $(a \colon [0,T] \to U^*\R^n)\in \mathcal{R}(\Lambda_K)$. Since the Reeb vector field is $(R_{\alpha})_{(q,p)} = \sum_{i=1}^np_i\partial_{q_i}$, there exists $q_0\in K$ and $p_0\in \Lambda_K\cap T^*_{q_0}\R^n$ such that $a(t)= q_0+t  p_0$ for every $t\in [0,T]$. Let us denote $q_1 \coloneqq q_0+ T p_0$.
Without loss of generality, we may assume that $p_0=(0,\dots ,0,1)\in \R^n$.
In addition, the contact distribution $\xi$ is written by
\[\xi_{(q,p)} = \{(u,v)\in \R^n\times \R^n \mid u\perp p,\ v\perp p\}\]
for every $(q,p)\in U^*\R^n$.
Since $p_0= (0,\dots,0,1)$, we may write $\xi_{a(t)} =\R^{n-1}\times \R^{n-1}$ for every $t\in[0,1]$ by identifying $\R^{n-1}\times \{0\}\subset \R^n$ with $\R^{n-1}$.

We associate to $a$ a path on $[0,1]$
\[\hat{a}\colon [0,1]\to T^*\R^n \colon t\mapsto  (q_0+(Tt)p_0,T p_0).\]
Then, it is a Hamiltonian chord with its end points in $L_K$ with respect to the Hamiltonian
$H\colon T^*\R^n\to \R\colon (q,p) \mapsto \frac{1}{2}|p|^2$.
Namely, $\hat{a}(0),\hat{a}(1)\in L_K$ and $\hat{a}$ satisfy $\frac{d \hat{a}}{dt}(t)=(X_H)_{\hat{a}(t)}$ for every $t\in [0,1]$, where the vector field $X_H$ is characterized by $\sum_{i=1}^n (dp_i\wedge dq_i) (\cdot, X_H)=dH$.
The flow of $X_H$ is given by
\[\varphi_H^t \colon T^*\R^n \to T^*\R^n \colon (q,p) \mapsto (q+tp, p).\]

We recall the definition in \cite[page 272]{APS} of the nullity and the index of Hamiltonian chords with conormal boundary condition.
Here, we identify the symplectic vector space $\R^n \times (\R^n)^*$ in \cite[Section 1]{APS} with $\C^n$ via the map
\[\C^n \to \R^n \times (\R^n)^*\colon q+\sqrt{-1}p \mapsto (p, \la q, \cdot \ra ).\]
Let $d$ be the codimension of $K$.
\begin{defi}{\cite[page 272, Definition 3.1]{APS}}\label{def-ind-APS}
For the Hamiltonian chord $\hat{a}$, we choose real subspaces $U_0,U_1\subset \R^n$ of codimension $d$ and a symplectic trivialization
\[\Phi \colon \hat{a}^*T(T^*\R^n) \to \C^n\]
such that $\Phi_{t}(\{0\}\times \R^n)=\R^n$ for every $t\in [0,1]$ and
$\Phi_i(T_{\hat{a}(i)}L_K) = \sqrt{-1} U_i\oplus U_i^{\perp}$ for $i\in \{0,1\}$. Then, we obtain a path
\[ C \colon [0,1]\to \mathcal{L}(\C^n) \colon t\mapsto  \Phi_t\circ d\varphi_H^t (T_{\hat{a}(0)}L_K). \]
We define the nullity $\nu'(\hat{a})$ and index $\mu'(\hat{a})$ of $\hat{a}$ by
\begin{align*}
\nu'(\hat{a}) & \coloneqq  \dim (T_{\hat{a}(1)}L_K \cap d\varphi^1_H(T_{\hat{a}(0)}L_K) ), \\
\mu' (\hat{a}) & \coloneqq \mu (C, \sqrt{-1} U_1\oplus U_1^{\perp}) + \frac{1}{2}(n-2d).
\end{align*}
\end{defi}
\begin{rem}
By \cite[Proposition 3.2]{APS}, $\mu'(\hat{a})$ is independent of the choice of $U_i$ ($i=0,1$) and $\Phi$.
The shift $\frac{1}{2}(n-2d)$ comes from $\frac{1}{2}(\dim K + \dim K -n)$ in \cite[Definition 3.1]{APS}.
\end{rem}
In order to choose $\Phi$ and $U_i$ ($i=0,1$) explicitly, we use the description
\begin{align}\label{tangent-LK}
T_{(q,p)}L_K= \{ (u, v- h_{(q,p)}(u) )  \mid u\in T_{q}K , v\in (T_qK)^{\perp}\},
\end{align}
where $h_{(q,p)}$ is a symmetric operator on $T_qK$ determined by $\la h_{(q,p)} (u), u' \ra = \la \mathrm{I\hspace{-1.2pt}I}_{q}(u, u') , p \ra $ for every $u,u' \in T_qK$,
where $\mathrm{I\hspace{-1.2pt}I}_{q} \colon T_q K\times T_q K \to (T_qK)^{\perp}$ is the second fundamental form of $K$ at $q\in K$.
Let us extend $h_{(q,p)}$ to a symmetric matrix $h_{(q,p)}\in M_{n\times n}(\R)$ by $h_{(q,p)}(u+u')=h_{(q,p)}(u)$ for every $u\in T_qK$ and $u'\in (T_qK)^{\perp}$.
We choose a continuous family $(h^t)_{t\in [0,1]}$ of symmetric matrices   such that $h^0=h_{\hat{a}(0)}$ and $h^1=h_{\hat{a}(1)}$ and define
\[\Phi_t \colon T_{\hat{a}(t)} (T^*\R^n) \to \C^n \colon (u,v) \mapsto \sqrt{-1} u+ (v+h^t(u)) .\]
Then, $\Phi$ satisfies the condition of Definition \ref{def-ind-APS}.
for $U_i=T_{q_i} K$ ($i=0,1$).

Next, we recall Definition \ref{def-ind-cap}.
As in Subsection \ref{subsubsec-capping}, we choose a capping path $\gamma_a \colon[0,1]\to \Lambda_K\colon t\mapsto (q(t),p(t))$ such that $\gamma_a(0)=(q_1,p_0)$ and $\gamma_a(1)=(q_0,p_0)$.
We define symplectic isomorphisms $\Psi_t\colon \xi_{\gamma_a(t)}\to \C^{n-1}$ for $t\in \{0,1\}$ by
\[\Psi_0 \colon \R^{n-1}\times \R^{n-1} \to \C^{n-1}\colon (u,v) \mapsto \sqrt{-1} T^{-\frac{1}{2}}u + T^{\frac{1}{2}}(v+h_{(q_1,p_0)}(u)).\]
and $\Psi_1 \coloneqq \Psi_0 \circ d\varphi^T_{\alpha}$.
Note that for $t\in \{0,1\}$,
\[\Psi_t(\xi_{\gamma_a(t)} \cap (\R^n\times\{0\})) = \Psi_t(\R^{n-1}\times\{0\}) = \{ \sqrt{-1} T^{-\frac{1}{2}}u + T^{\frac{1}{2}} h_{(q_1,p_0)}(u)\mid u\in \R^{n-1}\} . \]
We extend $\Psi_t$ for $t\in \{0,1\}$ to $\Psi\colon \gamma_a^*\xi \to \C^{n-1}$ so that $\Psi_t(\xi_{\gamma_a(t)}\cap(\R^n\times \{0\})) \in \mathcal{L}(\C^{n-1})$ is transversal to $\R^{n-1}$ for every $t\in [0,1]$.
By setting $\Psi_{t+1}\coloneqq \Psi_0\circ d\varphi^{T-t}_{\alpha}$ for $t\in [0,T]$, we obtain a trivialization $\Psi \colon (l_a)^*\xi \to \C^{n-1}$ over $\R/ (T+1)\Z$.
\begin{lem}\label{lem-triv-extend}
$\Psi$ extends to a symplectic trivialization of $\xi$ on a disk which bounds $l_a$.
\end{lem}
\begin{proof}
If we take any map $f\colon D\to U^*\R^n$ which bounds $l_a\colon \R / (T+1)\Z \to U^*\R^n$, we can find a trivialization $\Psi'\colon f^*\xi \to \C^{n-1}$ such that $\Psi'_z(\xi_{f(z)} \cap(\R^n\times \{0\})) = \sqrt{-1} \R^{n-1}$ for every $z\in D$ since $D$ is contractible. Then, we get a loop
$H\colon \R/(T+1)\Z \to \mathrm{Sp}(\C^{n-1}) \colon t\mapsto \Psi_t \circ (\Psi'_t)^{-1}$, where $\mathrm{Sp}(\C^{n-1})$ is the group of $\R$-linear maps preserving the symplectic form on $\C^{n-1}$. Composing it with the map $P\colon \mathrm{Sp}(\C^{n-1})\to \mathcal{L}(\C^{n-1})\colon A\to A(\sqrt{-1}\R^{n-1})$, we get a loop
\[P\circ H \colon\R/(T+1)\Z \to \{V\in \mathcal{L}(\C^{n-1})\mid V\text{ is transversal to }\R^{n-1}\}. \]
It is contractible since the target space is contractible. Moreover, since $P_*\colon\pi_1(\mathrm{Sp}(\C^{n-1})) \to \pi_1(\mathcal{L}(\C^{n-1}))$ is injective, $H$ is also contractible. This means that $\Psi$ can be extended to a symplectic trivialization of $f^*\xi$ as well as $\Psi'$.
\end{proof}

We define $\Gamma_a \colon [0,1]\to \mathcal{L}(\C^{n-1})\colon t\mapsto \Psi_t(T_{\gamma_a(t)} \Lambda_K)$ as in Subsection \ref{subsubsec-capping}.

Next, we take a capping path of the Hamiltonian chord $\hat{a}$ by
\[\hat{\gamma}_a\colon [0,1]\to L_K \colon t\mapsto (q(t),T \cdot p(t)).\] 
We define a symplectic trivialization $\hat{\Psi}\colon (\hat{\gamma}_a)^*T(T^*\R^n)\to \C^n$ such that for every $t\in [0,1]$,
\begin{align*}
&\hat{\Psi}_t(p(t),0) = \sqrt{-1} p_0 ,\ \hat{\Psi}_t(0,p(t)) = (1 + \sqrt{-1}t)p_0, \\
& \hat{\Psi}_t(u,v) = (\Psi_t(T^{\frac{1}{2}}u,T^{-\frac{1}{2}}v) , 0 ) \in \C^{n-1}\times \C,
\end{align*}
for every $(u,v)\in \R^n\times \R^n$ with $u \perp p(t)$ and  $v \perp p(t)$.
It is straightforward to check that
\begin{align}\label{Psi-Phi}
\begin{array}{cc} \hat{\Psi}_1= \hat{\Psi}_0\circ d\varphi^1_H, &
\hat{\Psi}_0 = \Phi_1 .\end{array}
\end{align}
(To check this, note that $h_{\hat{a}(1)} = T\cdot h_{(q_1,p_0)}$.)
Via $\hat{\Psi}$, we obtain a path of Lagrangian subspaces
\[\hat{\Gamma}_a\colon  [0,1]\to \mathcal{L}(\C^n)\colon t\mapsto \hat{\Psi}_t(T_{\hat{\gamma}_a(t)}L_K). \]
We can explicitly compute from (\ref{tangent-LK}) that $\hat{\Gamma}(t) = \Gamma_a(t) \oplus V(t)\subset \C^{n-1}\oplus \C$ for every $t\in [0,1]$, where $V$ is a path
\[V\colon [0,1]\to \mathcal{L}(\C) \colon t \mapsto (1 + \sqrt{-1} t)\R .
\]
Therefore, $\Gamma_a(0)$ is transversal to $\Gamma_a(1)$ if and only if $\hat{\Gamma}_a(0) \cap \hat{\Gamma}_a(1)=0$, and the letter condition is equivalent to that $\nu'(\hat{a})=0$.
Moreover, from the definition of \cite[page 831]{RS}, $\mu(V,V(0)) = \frac{1}{2}$.
Therefore, 
\begin{align}\label{hat-Gamma}
\mu(\hat{\Gamma}_a, \hat{\Gamma}_a(0)) = \mu(\Gamma_a,\Gamma_a(0))+\frac{1}{2}. 
\end{align}
In addition, we note that $\hat{\Gamma}_a(0) = \Phi_1(T_{\hat{a}(1)}L_K)=\sqrt{-1}U_1 \oplus U_1^{\perp}$.

\begin{lem}\label{lem-hat-Gamma}
$\mu(\hat{\Gamma}_a, \hat{\Gamma}_a(0)) = \mu(C,\sqrt{-1} U_1\oplus U_1^{\perp})$.
\end{lem}

\begin{proof}
We consider a loop in $L_K$
\[\hat{l}_a \colon \R/3\Z \to L_K \colon t \mapsto 
\begin{cases}
\hat{\gamma}_a(t) & \text{ if }0\leq t\leq  1, \\
\hat{a}(0) & \text{ if }1\leq t\leq 2 , \\
\hat{\gamma}_a(3-t) & \text{ if }2\leq t\leq 3.
\end{cases}\]
$\hat{\Psi}$ gives a trivialization of $(\hat{l}_a)^*T(T^*\R^n)$ on $[0,1]$. We extend it to a symplectic trivialization on $[0,3]$ as follows: First, by (\ref{Psi-Phi}), we can extend $\hat{\Psi}$ on $[1,2]$ by
\[\hat{\Psi}_{t+1} \coloneqq \Phi_{1-t} \circ d\varphi^{1-t}_H\colon T_{\hat{a}(0)} (T^*\R^n) \to \C^n\]
for every $t\in [0,1]$.
Since $\hat{\Psi}_2= \Phi_0$, we can extend $\hat{\Psi}$ on $[2,3]$ by
\[ \hat{\Psi}_{t+2}(u,v) \coloneqq \sqrt{-1} u+ (v +h_{\hat{\gamma}_a(1-t)}(u))\]
for every $t\in [0,1]$. Then, $\hat{\Psi}_3 = \hat{\Psi}_0$. Therefore, we get a symplectic trivialization $\hat{\Psi}$ of $(\hat{l}_a)^*T(T^*\R^n)$ on $\R/3\Z$.

Since $\hat{\Psi}_t(\R^n \times \{0\})$ is transversal to $\R^n$ for every $t\in \R/3\Z$,
in a similar way as Lemma \ref{lem-triv-extend}, we can check that
$\hat{\Psi}$ extends to a trivialization of $T^*\R^n$ on a disk which bounds $\hat{l}_a$.
As a Lagrangian submanifold of $T^*\R^n$, the Maslov class of $L_K$ vanishes. Therefore, the Maslov index of the loop
\[ \tilde{\Gamma} \colon \R/3\Z \mapsto \mathcal{L}(\C^n) \colon t \mapsto \hat{\Psi} (T_{\hat{l}_a(t)} L_K)\]
is equal to $0$.
By definition,
$\tilde{\Gamma}(t)=\hat{\Gamma}_a(t)$ if $0\leq t\leq 1$ and $\tilde{\Gamma}(t)= C(2-t) $ if $1\leq t\leq 2$. Therefore,
\[ 0 = \mu(\tilde{\Gamma}) = \mu(\hat{\Gamma}_a, \hat{\Gamma}_a(0)) - \mu (C, \sqrt{-1} U_1\oplus U_1^{\perp}) + \mu(  \tilde{\Gamma}|_{[2,3]} , \sqrt{-1} U_1\oplus U_1^{\perp}).\]
If $2\leq t \leq 3$, we can compute from (\ref{tangent-LK}) that
\[ \tilde{\Gamma}(t) = \sqrt{-1} T_{\pi_{\R^n} \circ \gamma_a(3-t)}K \oplus (T_{\pi_{\R^n}\circ \gamma_a(3-t)}K)^{\perp}, \] where $\pi_{\R^n} \colon T^*\R^n \to \R^n$ is the bundle projection.
By this description and \cite[Proposition 1.2]{APS},
\[\mu(  \tilde{\Gamma}|_{[2,3]} , \sqrt{-1} U_1\oplus U_1^{\perp})=0 .\]
Hence, $\mu(\hat{\Gamma}_a, \hat{\Gamma}_a(0)) - \mu (C, \sqrt{-1} U_1\oplus U_1^{\perp}) =0$.
\end{proof}
Summarizing the above computations, we get the following:
$a\in \mathcal{R}(\Lambda_K)$ is non-degenerate if and only if $\nu'(\hat{a})=0$. Moreover, the Conley-Zehnder index $\mu(a)$ of $a$ is
\begin{align*}
\mu(a)& = \mu(\Gamma_a,\Gamma_a(1)) + \frac{n-1}{2}\\
&= \mu(\hat{\Gamma}_a ,\hat{\Gamma}_a(1))- \frac{1}{2} + \frac{n-1}{2}\\ &=\mu'(\hat{a})-\frac{1}{2}(n-2d)-\frac{1}{2}+\frac{n-1}{2} \\
&= \mu'(\hat{a})+d-1.
\end{align*}
Here, the second equality follows from (\ref{hat-Gamma}) and the third equality follows from Lemma \ref{lem-hat-Gamma} and the definition of $\mu'(\hat{a})$.

As in \cite[Section 6]{APS}, let $\mathcal{P}_K$ be the Hilbert manifold which consists of paths $c \colon [0,1]\to \R^n$ such that $c(0),c(1)\in K$ and as an $\R^n$-valued function,  $c $ belongs to the Sobolev space $W^{1,2}([0,1],\R^n)$.
By the Fenchel transformation defined in \cite[Section 4]{APS}, the time-independent Hamiltonian $H\colon T^*\R^n \to \R$ defines a time-independent Lagrangian
$ L\colon T\R^n\to \R \colon (q,v) \mapsto \frac{1}{2} |v|^2$.
Given any Lagrangian on $T\R^n$,  a functional on $\mathcal{P}_K$ is defined as described in \cite[Section 4]{APS}. In the case of  $L$, we get
\begin{align}\label{energy-functional}
 \mathcal{E}\colon \mathcal{P}_K \to \R\colon c \mapsto \frac{1}{2}\int_{[0,1]} \left| \frac{d c}{dt} \right|^2 dt  . 
 \end{align}
Then, $c_a \coloneqq \pi_{\R^n} \circ \hat{a}$ is a critical point of $\mathcal{E}$.

\begin{proof}[Proof of Proposition \ref{prop-correspond}]

\cite[Corollary 4.2]{APS} shows the following:
\begin{itemize}
\item $\nu'(\hat{a})$ is equal to the nullity of $c_a$ as a critical point of $\mathcal{E}$.
\item If $\nu'(\hat{a}) =0$, then $\mu'(\hat{a})$ is equal to the Morse index of $c_a$ as  a critical point of $\mathcal{E}$.
\end{itemize}
It is a standard fact that the nullity (resp. Morse index) of $c_a$ with respect to $\mathcal{E}$ is equal to the nullity (resp. Morse index) of $x_a =(c_a(0),c_a(1)) \in \mathcal{C}(K)$ with respect to $E$.
Therefore, $a\in \mathcal{R}(\Lambda_K)$ is a non-degenerate Reeb chord if and only if the nullity of $x_a$ is $0$. Moreover, if $a$ is non-degenerate, then
\[\mu(a) = \mu'(\hat{a}) +d-1 = \ind x_a +d-1 . \]
This finishes the proof.
\end{proof}

\subsection{Proof of Proposition \ref{prop-Morse-singular}}\label{subsec-A2}

Fix $g\in \mathcal{G}_K$ such that the Morse-Smale condition holds for $V_g$. We give filtrations on $CS_*(g)$ and $CM_*$ by the function $E \colon K\times K\to \R_{\geq 0}$ as follow:
For every $l\in \R_{>0}$ and $p\in \Z$, we define $CS^{\leq l}_p(g)$ to be the subcomplex of $CS_p(g)$ which is generated by chains represented by $c\colon \Delta^p \to K\times K$ satisfying the condition (T) and $c(\Delta^p) \subset E^{-1}([0,l ])$. (For the condition (T), see Subsection \ref{subsubsec-singular-Morse})
We also define $CM^{\leq l}_*$ to be the subcomplex of $CM_*$ spanned by the set $\{x \in \mathcal{C}(K) \mid E(x) \leq l\}$. We denote their homology by $HS^{\leq l}_*$ and $HM_*^{\leq l}$ respectively.
It is clear from the definition that the chain map $F_g \colon CS_*(g) \to CM_*$ preserves the filtrations.

For any $l'>l$, we define the quotient complexes
\[
\begin{array}{cc} CS^{(l,l']}_*(g) \coloneqq CS^{\leq l'}_*(g)/ CS_*^{\leq l}(g), & CM^{(l,l']}_*\coloneqq CM^{\leq l'}_*/CM^{\leq l}_*. \end{array}\]
We denote their homology by $HS^{(l,l']}_*$ and $HM_*^{(l,l']}$ respectively.

\begin{proof}[Proof of Proposition \ref{prop-Morse-singular}]
The second assertion is proved by taking a triangulation of $P$ such that the restriction of $c$ on each triangle satisfy the condition (T). The third assertion is proved in the same way.

We prove the first assertion for $\Phi_g$. 
We take a finite sequence 
$\{l_i\}_{i=0,\dots , m}$ in $\R_{>0} \setminus E(\mathcal{C}(K))$ such that
$0<\epsilon_0 = l_0<l_1<\dots <l_{m-1} < \max  E  < l_m$ and $( l_i,l_{i+1}]$ contains only a single critical value of $E$.
Then,
\begin{align*}
 (I_g)_* \colon HS_*^{(l_i,l_{i+1}]} \to & H_*( E^{-1}([0, l_{i+1}]) , E^{-1}([0,l_i]) ) \\
&\cong    H_*( E^{-1}([l_i, l_{i+1}]) , E^{-1}(l_i) ) 
 \end{align*}
is an isomorphism and
\[ HM_*^{(l_i,l_{i+1}]}= CM_*^{(l_i,l_{i+1}]} = \bigoplus_{x\in \mathcal{C}(K) \cap E^{-1}((l_i,l_{i+1}]) } \Z/2 \cdot  x.\]

For every $x\in \mathcal{C}(K) \cap E^{-1}((l_i,l_{i+1}])$, we define \[D_x\coloneqq \W_g^u(x) \cap E^{-1}([l_i,l_{i+1}]) . \]
Then, $D_x$ is an embedded disk with boundary $\partial D_x$ in $E^{-1}(l_i)$. By a standard argument in Morse theory, a basis of the singular homology $H_*( E^{-1}([l_i, l_{i+1}]) , E^{-1}(l_i) )$ is given by the set of the homology classes
\[ \{ [D_x] \mid x \in \mathcal{C}(K) \cap E^{-1}((l_i,l_{i+1}]) \} .\]

 $D_x$ and $\partial D_x$ are transversal to $\W^s_g(x')$ for every $x' \in \mathcal{C}(K)$. Therefore, in the same way as the second assertion, $(F_g)_*\circ (I_g)_*^{-1}([D_x]) \in HM_*^{(l_i,l_{i+1}]} = CM_*^{(l_i,l_{i+1}]} $ is equal to
\[\sum_{y\in \mathcal{C}(K) \cap E^{-1}((l_i,l_{i+1}])} \#_{\Z/2} (D_x \cap \W^s_g(y)) \cdot y = x \]
for every $x \in \mathcal{C}(K)\cap E^{-1}((l_i,l_{i+1}])$. This means that
\begin{align}\label{isom-FV}
(F_g)_* \colon HS^{(l_i,l_{i+1}]}_* \to HM_*^{(l_i,l_{i+1}]}\colon (I_g)_*^{-1}([D_x]) \mapsto x
\end{align}
 is an isomorphism.

Clearly, we have $HS_*^{\leq l_0}=0 =HM_*^{\leq l_0}$.
For each  $i\in \{0,\dots ,m-1\}$, we have a commutative diagram
\[\xymatrix{
HS_{*+1}^{(l_i,l_{i+1}]} \ar[r] \ar[d]_-{(F_g)_*} & HS_{*}^{\leq l_i} \ar[r] \ar[d]_-{(F_g)_*} &  HS_{*}^{\leq l_{i+1}} \ar[r] \ar[d]_-{(F_g)_*} & HS_{*}^{(l_i,l_{i+1}]}  \ar[r] \ar[d]_-{(F_g)_*} & HS_{*-1}^{\leq l_i} \ar[d]_-{(F_g)_*}  \\
HM_{*+1}^{(l_i,l_{i+1}]} \ar[r] &HM_{*}^{\leq l_i} \ar[r]  &  HM_{*}^{\leq l_{i+1}} \ar[r] & HM_{*}^{(l_i,l_{i+1}]} \ar[r] & HM_{*-1}^{\leq l_i} ,
}\]
where the horizontal sequences are exact.
If $(F_g)_*\colon HS^{\leq l_i}_* \to HM_*^{\leq l_i}$ is an isomorphism,
then (\ref{isom-FV}) and the five lemma show that the middle map $(F_g)_*\colon HS^{\leq l_{i+1}}_* \to HM_*^{\leq l_{i+1}}$ is also an isomorphism. 
Therefore, by an induction on $i=0,1,\dots ,m-1$, we can show that  $(F_g)_*\colon HS^{\leq l_i}_* \to HM_*^{\leq l_i}$ is an isomorphism for every $i=0,\dots ,m$.
In particular,
\[(F_g)_* \colon H_*(CS_*(g),\partial^{\sing}) =  HS^{\leq l_m}_* \to HM^{\leq l_m}_* = HM_*(g)\]
is an isomorphism.
This finishes the proof that $\Psi_g= (F_g)_* \circ (I_g)_*^{-1} \colon H_*(K\times K,\Delta_K ) \to HM_*(g)$ is an isomorphism.
%
In a similar way, we can show that $\Phi_g^2$ is an isomorphism.
\end{proof}

\subsection{Transversality at the trivial strip}\label{subsec-A3}

Let us introduce several notations for analysis.  Recall the map (\ref{diffeo-symp-cot}) $F\colon \R\times U^*\R^n \to T^*\R^n$. We take a Riemannian metric $g_0$ on $T^*\R^n$ such that:
\begin{itemize}
\item $F^*g_0$ on $\R_{\geq 0} \times U^*\R^n$ is preserved by $\tau_s$ for every $s\geq 0$.
\item Both $L_K$ and $\R^n$ (identified with the image of the zero section) are totally geodesic.
\end{itemize}
Fix a Reeb chord $(a\colon [0,T]\to U^*\R^n)\in \mathcal{R}(\Lambda_K)$ which is non-degenerate and denote $x_a \in \mathcal{C}(K)$ by $x_a =(q_1,q_2)$.
Since we have the biholomorphic map (\ref{bihol}), let us rewrite
\[D_3 \coloneqq [0,\infty)\times [0,1] \setminus \{(0,0), (0,1)\}.\]
In addition, recall the trivial strip $u_a\colon D_3 \to T^*\R^n$ over $a$ defined in Subsection \ref{subsubsec-def-Mm}.

Let $p>2$ and $\delta>0$. We fix them later.
For $k\in \Z_{\geq 0}$, we define the weighted Sobolev space
\[W^{k,p}_{\delta}(D_3,u_a^*T(T^*\R^n))\]
which consists of sections of $u_a^*T(T^*\R^n)$ with weight $e^{\delta s}$ on the positive strip $\psi_0\colon [0,\infty) \times [0,1] \to D_3$ and $e^{-\delta s }$ on the negative strip $\psi_i\colon  (-\infty,0] \times [0,1] \to D_3$ for $i=1,2$. For more details, see \cite[Section 7.2]{W}. Here, $s$ is the coordinate of $[0,\infty)$ and $(-\infty,0]$.
We use the Riemannian metric $g_0$ on $T^*\R^n$ and the Levi-Civita connection with respect to it in order to define this weighted Sobolev space 
and its norm $||\cdot ||_{k,p}$.

Let $C^{\infty}(D_3, u_a^*T(T^*\R^n))$ be the vector space of $C^{\infty}$ sections of $u^*T(T^*\R^n)$.
We fix $\R$-linear maps $\zeta_0\colon \R \to C^{\infty}(D_3, u_a^*T(T^*\R^n))$ and $\zeta_i \colon T_{q_i}K \to  C^{\infty}(D_3, u_a^*T(T^*\R^n))$ for $i=1,2$ satisfying the following:
\begin{itemize}
\item For $\sigma\in \R$, $\zeta_0(\sigma)$ is supported in $\Im \psi_0$ and $\zeta_0(\sigma)_{\psi_0(s,t)} = F_*(\sigma\cdot \partial_r) \in  T_{u_a\circ \psi_0(s,t)}(T^*\R^n)$ if $s \gg 0$. Here, $r$ is the $\R$-coordinate of $\R\times U^*\R^n$.
\item For $v\in T_{q_i}K$, $\zeta_i(v)$ is supported in $\Im \psi_i$ and $\zeta_i(v)_{\psi_i(s,t)}$ is constant to $v$ if $s \ll 0$ via the identification $T^*\R^n \cong \R^n\times \R^n$ by (\ref{identify-TRn}).
\end{itemize}
For any $\zeta\in W^{1,p}_{\delta}(D_3,u_a^*T(T^*\R^n))$, $\sigma \in \R$ and $v_i\in T_{q_i}K$ ($i=1,2$), let us denote
\[\zeta_{(\sigma ,v_1,v_2)} \coloneqq \zeta +\zeta_0(\sigma) + \zeta_1(v_1) +\zeta_2(v_2). \]
Then, we define
\begin{align*}
H_1 &\coloneqq \left\{ \eta =\zeta_{(\sigma,v_1,v_2)} \ \middle|
\begin{array}{cll}
  &\zeta\in W^{1,p}_{\delta}(D_3,u_a^*T(T^*\R^n)),\ \sigma\in \R,\ v_i\in T_{q_i}K \ (i=1,2) ,\\
  &\eta(z)\in \begin{cases}  T_{u_a(z)}L_K & \text{ if }z\in \partial_1 D_3\cup \partial_3 D_3 , \\ T_{u_a(z)}\R^n & \text{ if }z\in \partial_2 D_3. 
  \end{cases}
\end{array} \right\} , \\
H_2 & \coloneqq W^{0,p}_{\delta} (D_3, u_a^*T(T^*\R^n)).
\end{align*}
\begin{rem}\label{rem-eta-bounded}
The norm $||\eta ||$ of $\eta = \zeta_{(\sigma,v_1,v_2)}\in H_1$ is defined by
$|| \eta || = ||\zeta||_{p,1} + |\sigma| + |v_1| +|v_2|$.
We note that any $\eta = \zeta_{(\sigma,v_1,v_2)}\in H_1$ is $C^0$-bounded and $\eta(s,t)$ is asymptotic to $F_*(\sigma\cdot \partial_r) \in \R \cdot (\partial_su_a)$  when $s\to \infty$.
\end{rem}

We consider the linearlization at $u_a$ of the Cauchy-Riemann operator $\Dbar_{J'}$ with respect to $J'$. It is a bounded operator
\[  D_{u_a} \Dbar_{J'} \colon H_1\to H_2.\]
It is a Fredholm operator if $\delta>0$ is sufficiently small. See, for instance, \cite[Lemma 7.10]{W}. The Fredholm index of $D_{u_a} \Dbar_{J'}$ is equal to $\vdim \M_1(a)= |a|-d+2$ by Proposition \ref{prop-vdim}.
Note that $\delta>0$ can be taken independently on $p>2$. Let us fix $p\coloneqq \frac{2\pi}{\pi -\delta}$. For the reason, see Remark \ref{rem-boundedness}.
\begin{lem}\label{lemA-trivial-strip}
If $\eta \in\ker D_{u_a} \Dbar_J$ satisfies $\eta_z \in  \Im (du_a)_{z}$ for every $z\in (0,\infty)\times [0,1]$, then $\eta=0$.
\end{lem}
\begin{proof}
Referring to definition (\ref{eq-trivial-strip}) of $u_a$, consider the $2$-plane $H\coloneqq  \{ (q_1+ k p_0  , l p_0) \in T^*\R^n \mid k,l\in \R  \}$ containing $\Im u_a$ and a diffeomorphism
\[ \Phi\colon \{z\in \C \mid \Re z \geq 0\} \to H \colon (s+\sqrt{-1}t) \mapsto (q_1 + Tt\cdot p_0, h(s)p_0).\]
 (Note that $a(0)$ is denoted here by $(q_1,p_0)$.)
$\Phi$ is $J'$-holomorphic and coincides with $u_a$ on $\{z\in \C \mid \Re z > 0 ,  0  \leq \Im z \leq 1 \}\cong (0,\infty)\times [0,1]$. If  $\eta_z \in  \Im (du_a)_{z}$ for every $z\in (0,\infty)\times [0,1]$, $\eta$ is tangent to $H$ and via the map $\Phi$,
$\eta \in\ker D_{u_a} \Dbar_{J'} $ can be regarded as a holomorphic function $\eta \colon [0,\infty) \times [0,1] \to  \C$ such that $\eta(\R\times \{0,1\}) \subset \R$, $\eta (\{0\}\times [0,1]) \subset \sqrt{-1}\R$ and $\eta(s,t)$ is uniformly bounded by Remark \ref{rem-eta-bounded}. These properties imply that $\eta$ must be constant to $0$.
\end{proof}

Let us give a lemma inspired by a construction of Kuranishi chart \cite[Subsection 7.2.1]{fooo}.
To state this, we note that $\M_1(a)$ is the zero locus of $\Dbar_{J'}$, which is a section of a Banach bundle $\mathfrak{E}$ over a Banach manifold $\mathfrak{B}$. A local coordinate of $\mathfrak{B}$ near $u_a$ is given by
$\{ \eta \in H_1 | || \eta|| < \epsilon \} \to \mathfrak{B} \colon \eta \mapsto \exp_{u_a} \eta$ for small $\epsilon>0$, where $\exp$ is the exponential map with respect to $g_0$. The fiber $\mathfrak{E}_{u}$ at $u\in \mathfrak{B}$ is equal to $H_2$ when $u=u_a$.
\begin{lem}\label{lem-kuranishi}
For any $\eta \in \ker D_{u_a} \Dbar_{J'}$, there exist $R_1>\delta_1>0$ and a $C^{\infty}$ path $[-\epsilon_1,\epsilon_1] \to \mathfrak{B} \colon \epsilon \mapsto u_{\epsilon}$ such that:
\begin{enumerate}
\item[(i)] $u_{\epsilon} \colon D_3 \to T^*\R^n$ is of class $C^{\infty}$.
\item[(ii)] $\Dbar_{J'} (u_{\epsilon}) =0 $ outside $[\delta_1,R_1]\times [0,1]$.
\item[(iii)] $ \rest{ \frac{d u_{\epsilon}}{d\epsilon} }{\epsilon=0} = \eta $.
\end{enumerate}
\end{lem}
\begin{proof}
Since $\coker D_{u_a} \Dbar_{J'}$ is finite dimensional, there exists finite elements $\theta_1,\dots ,\theta_m \in H_2$ such that $\Im D_{u_a} \Dbar_{J'} + \mathrm{Span}(\theta_1, \dots ,\theta_m) = H_2$.
Then, we take $\theta'_1,\dots ,\theta'_m \in H_2$ and $R_1\gg \delta_1>0$ such that for every $k\in \{1,\dots ,m\}$: 
\begin{itemize}
\item $||\theta'_k -\theta_k||_{0,p}$ is so small that $\Im D_{u_a} \Dbar_{J'} + \mathrm{Span}(\theta'_1, \dots ,\theta'_m) = H_2$
\item $\theta'_k$ is of class $C^{\infty}$.
\item $\theta'_k(s,t)=0$ for every $(s,t) \in D_3 \setminus ([\delta_1, R_1]\times [0,1])$.
\end{itemize}
On a small neighborhood $\mathfrak{U} \subset \mathfrak{B}$ of $u_a$, we trivialize $\rest{\mathfrak{E}}{\mathfrak{U}}$ by parallel transport as in \cite[Subsection 7.2.1]{fooo}. Let us denote it by $\mathrm{Par}_{u} \colon \mathfrak{E}_{u} \to \mathfrak{E}_{u_a} = H_2$ for every $u\in \mathfrak{U}$.

We consider the map
\[ \mathfrak{D}\colon \mathfrak{U} \to H_2 \colon u\mapsto \mathrm{Par}_u ( \Dbar_{J'}u).\]
This map is transversal to $S\coloneqq \mathrm{Span}(\theta'_1,\dots ,\theta'_m)$ at $u_a$, so if we take $\mathfrak{U}$ smaller if necessary, $\mathfrak{D}$ is transversal to $S$ on $\mathfrak{U}$ and  
$\mathfrak{M}\coloneqq \mathfrak{D}^{-1}(S)$ is a finite dimensional $C^{\infty}$ manifold.
From the conditions on $\{\theta'_k\mid k=1,\dots ,m\}$ and elliptic bootstrapping, every $u\in \mathfrak{M}$ satisfies the conditions (i) and (ii). 
Since the tangent space $T_{u_a} \mathfrak{M}= (D_{u_a}\Dbar_{J'})^{-1}(S)$ at $u_a$ contains $\ker D_{u_a} \Dbar_{J'}$, for every $\eta \in \ker D_{u_a} \Dbar_{J'}$, there exists a $C^{\infty}$ path $[-\epsilon_1,\epsilon_1]\to \mathfrak{M} \colon \epsilon \mapsto u_{\epsilon}$ such that $\rest{ \frac{d u_{\epsilon}}{d \epsilon} }{\epsilon=0}=\eta$.
\end{proof}


\begin{proof}[Proof of Lemma \ref{lem-trivial-strip}]
The transversality for $\N_g(a;x_a)$ at the trivial strip $u_a$ is equivalent to the following: The restriction of $ D_{u_a} \Dbar_{J'} $ on the subspace
\[ \{ \eta= \zeta_{(\sigma,v_1,v_2)}\in H_1 \mid   (v_1,v_2) \in T_{x_a} \W^s_g(x_a) \} \]
is a surjection onto $H_2$. The Fredholm index of the restriction of $D_{u_a} \Dbar_{J'}$ on this subspace is $0$, so it suffices to show that 
\begin{align}\label{kernel-Dbar}
\ker  D_{u_a} \Dbar_{J'}  \cap \{ \eta= \zeta_{(\sigma,v_1,v_2)} \in  H_1  \mid   (v_1,v_2) \in T_{x_a} \W^s_g(x_a) \} 
\end{align}
is the zero vector space.

Let us assume that there exists a non-zero element $\eta =  \zeta_{(\sigma,v_1,v_2)} \neq 0$ of (\ref{kernel-Dbar}). We will deduce a contradiction.
Let $f_{\eta} \in T_{c_a}\mathcal{P}_K$ be an element determined by $f_{\eta}(t) = \eta(0,t)$ for every $t\in (0,1)$. (A priori, we only have $C^0$-boundedness of $f_{\eta}$. For the $C^{\infty}$ boundedness of $f_{\eta}$ on $[0,1]$, see Remark \ref{rem-boundedness} below.) We denote $X_{\eta}\coloneqq (f_{\eta}(0), f_{\eta}(1)) = (v_1,v_2)$.

We take a family $(u_{\epsilon})_{\epsilon\in [-\epsilon_1,\epsilon_1]}$ for $\eta$ from Lemma \ref{lem-kuranishi}. Then, we get a path
$c_{\epsilon} \in \mathcal{P}_K$ defined by $c_{\epsilon}(t) \coloneqq u_{\epsilon}(0,t)$ for every $t\in [0,1]$. (Its $C^{\infty}$-boundedness follows by a similar argument as Remark \ref{rem-boundedness}.)
Let us consider the function
\[\mathcal{L} \colon [-\epsilon_1 ,\epsilon_1] \to \R \colon \epsilon \mapsto \len (c_\epsilon) = \int_{[0,1]} \left| \frac{d c_{\epsilon}}{dt} \right| dt .\]
Here, $|\cdot |$ is the norm on $\R^n$ with respect to the standard metric $\la\cdot ,\cdot \ra$.
$T p_0 =\frac{d c_a}{dt}$ is a constant vector with $|p_0|=1$ and $\frac{d c_{\epsilon}}{dt} = Tp_0 + O(\epsilon) \neq 0$ if $|\epsilon|$ is sufficiently small. Hence $\mathcal{L}$ is of class $C^{\infty}$ near $\epsilon=0$.
Its derivative at $\epsilon =0$ vanishes since
\[\rest{ \frac{d \mathcal{L}}{d\epsilon} }{\epsilon=0} = \int_{[0,1]} |Tp_0|^{-\frac{3}{2}} \left\la  Tp_0 , \frac{d f_{\eta}}{dt} \right\ra dt = T^{-\frac{1}{2}} \left\la p_0 , \frac{d f_{\eta}}{dt}(1) - \frac{d f_{\eta}}{dt} (0) \right\ra  \]
and both $\frac{d f_{\eta}}{dt}(0)\in T_{q_1}K $ and $\frac{d f_{\eta}}{dt} (1) \in T_{q_2}K$ are perpendicular to $p_0$.
Likewise, we consider
\[ L \colon [-\epsilon_1,\epsilon_1] \to \R\colon \epsilon \mapsto \sqrt{ 2 E( c_{\epsilon}(0) , c_{\epsilon}(1))} = | c_{\epsilon}(1) - c_{\epsilon}(0)|.\]
Its first derivative at $\epsilon=0$ vanishes. The second derivative $\frac{d^2 L}{d\epsilon^2} $ is non-negative  at $\epsilon=0$ since
$X_{\eta}= ( \rest{ \frac{d}{d\epsilon}}{\epsilon=0} c_{\epsilon}(0) ,\rest{ \frac{d}{d\epsilon}}{\epsilon=0}  c_{\epsilon}(1))$ belongs to $T_{x_a} \W^s_g(x_a) $.
It is obvious that $\mathcal{L}(0) = L(0) =T$ and $\mathcal{L}(\epsilon) \geq L(\epsilon)$ for every $\epsilon\in [-\epsilon_1,\epsilon_1]$. This implies that
\begin{align}\label{ineq-hessian}
 \rest{ \frac{d^2 \mathcal{L}}{d\epsilon^2}}{\epsilon=0} \geq \rest{ \frac{d^2 L}{d\epsilon^2}}{\epsilon=0} \geq 0.
 \end{align}

On the other hand, we claim that there exists a positive constant $A$ such that
\[ \mathcal{L} (\epsilon) \leq \mathcal{L}(0) -A \epsilon^2 .\]
when $\epsilon$ is close to $0$.
We leave its proof to Lemma \ref{lem-approx} below.
This inequality means that $ \rest{ \frac{d^2 \mathcal{L}}{d\epsilon^2}}{\epsilon=0} \leq-A<0$, which contradicts (\ref{ineq-hessian}). Therefore, (\ref{kernel-Dbar}) is the zero vector space.
\end{proof}

\begin{rem}\label{rem-boundedness}
On $R\coloneqq [0,s'_0)\times [0,1]$ for small $s'_0>0$, let $(s',t')$ denote the coordinate. For $\eta = \zeta_{(\sigma , v_1,v_2)}\in \ker D_{u_a} \Dbar_{J'}$, we set $\eta(0,0) \coloneqq v_1$ and $\eta_{(0,1)} \coloneqq v_2$, then $\eta$ is defined on $R$. Obviously, it is $C^0$-bounded. We claim that $\eta$ is $C^{\infty}$ bounded on $[0,\frac{s'_0}{2})\times [0,1]$.
By elliptic bootstrapping argument, it suffices to show the $L^p$-boundedness of $\nabla_{s'} \eta$ and $\nabla_{t'}\eta$. We want to deduce it from the $W^{1,p}$-boundedness of $e^{-\delta s} (\eta\circ \psi_i )$ on $ (-\infty,s_i] \times [0,1] $ for $i=1,2$.
We can take explicitly the holomorphic map $\psi_1 \colon (-\infty, s_1]\times [0,1]\colon (s,t)\mapsto (e^{\frac{\pi}{2}s}\cos (\frac{\pi}{2} t) , e^{\frac{\pi}{2}s} \sin(\frac{\pi}{2} t))$ for $s_i<0$. From our choice $p= \frac{2\pi}{\pi-\delta}$,
\begin{align*}
\int_{\Im \psi_1 } |\nabla_{s'} \eta |^p ds'dt' & \leq C \int_{(-\infty,s_1]\times [0,1]}  |\nabla_{s} \eta \cdot e^{-\frac{\pi}{2}s} |^p (e^{\frac{\pi}{2} s})^2 dsdt \\
&= C \int_{(-\infty,s]\times [0,1]} | e^{-\delta s } \nabla_{s} \eta |^p dsdt <\infty
\end{align*}
for $C\coloneqq  (\frac{\pi}{2})^{2-p}$.
By similar estimates on $\Im \psi_2$ and for $\nabla_{t'} \eta$, the $L^p$-boundedness of $\nabla_{s'} \eta$ and $\nabla_{t'}\eta$ are proved.
\end{rem}

The proof of Lemma \ref{lem-trivial-strip} is completed by the next lemma.
\begin{lem}\label{lem-approx}
There exists $A>0$ and $\epsilon'_2\in (0,\epsilon_1]$ such that $\len (c_{\epsilon}) \leq T -A\epsilon^2$ for any $\epsilon \in [-\epsilon'_2,\epsilon'_2]$.
\end{lem}
\begin{proof}
Recall the assumption that $\eta \neq 0$.  Since  $\eta(s,t)$ is asymptotic to $F_*(\sigma\cdot \partial_r) \in \R \cdot (\partial_su_a)$ for some $\sigma\in \R$ when $s\to \infty$ by Remark \ref{rem-eta-bounded}, if we suppose that $(\nabla_s\eta)_z \in \Im (d u_a)_z$ for every $z\in (0,\infty)\times [0,1]$, then $\eta_z \in \Im (d u_a)_z$ for every $z\in (0,\infty)\times [0,1]$, which contradicts with Lemma \ref{lemA-trivial-strip}.
Hence, there exists $z_0 \in (0,\infty)\times [0,1]$ such that $(\nabla_s \eta)_{z_0} \notin \Im (du_a)_{z_0}$.
We may take $R_1>\delta_1>0$ of Lemma \ref{lem-kuranishi} so that $z_0\in [\delta_1,R_1]\times [0,1]$.

Since $u_a([\delta_1,\infty) \times [0,1])$ is disjoint from the zero section, there exists $\epsilon_2>0$ such that $u_{\epsilon} ([\delta_1,\infty)\times [0,1])$ is disjoint from the zero section for every $\epsilon \in [-\epsilon_2, \epsilon_2]$. Let us take $r_1>0$ such that
\[ \bigcup_{\epsilon\in  [-\epsilon_2,\epsilon_2]} u_{\epsilon}([\delta_1,\infty)\times [0,1]) \subset \{ (q,p) \in T^*\R^n \mid |p| >r_1 \}.\]

Fix $\epsilon\in [-\epsilon_2,\epsilon_2]$.
Since $\rest{ u_{\epsilon}}{[0,\delta_1]\times [0,1]}$ is $J'$-holomorphic, we can choose arbitrarily small $\delta\in (0,\delta_1)$ such that $u_{\epsilon}(\{\delta\}\times [0,1])$ is disjoint from the zero section.
Associated to $\delta$, we take a $C^{\infty}$ function $\tau \colon [0,\infty) \to [0,1]$ such that:
\begin{itemize}
\item  $\frac{d \tau_{\delta}}{dr}(r)\geq 0$ for every $r \in [0,\infty)$.
\item $\tau(r)=0$ near $0\in [0,\infty)$.
\item There exists $r_{\delta}\in (0,r_1)$ such that $\tau(r)=1$ for every $r\geq r_{\delta}$ and \[u_{\epsilon}(\{\delta\} \times [0,1]) \subset \{(q,p)\in T^*\R^n \mid |p|>r_{\delta}\}.\]
 \end{itemize}
 Then, as in \cite[Section 8.2]{CELN} and \cite[Section 4.3]{Asp}, we define a $1$-form $\lambda_{\tau} \coloneqq \frac{\tau(|p|)}{|p|} p dq$ on $T^*\R^n$.
Now, we consider the integral $\int_{[\delta, \infty) \times [0,1]} (u_{\epsilon})^* d\lambda_{\tau}$. By Stokes' theorem and the estimate
\[\int_{\{\delta\} \times [0,1]} (u_{\epsilon})^* \lambda_{\tau} = \int_{\{\delta\} \times [0,1]} (u_{\epsilon})^* \left( \frac{p}{|p|} dq\right) = \len (c_{\epsilon}) + O(\delta)\]
as in the proof of \cite[Proposition 8.9]{CELN}, we obtain
\[ \int_{[\delta, \infty) \times [0,1]} (u_{\epsilon})^* d\lambda_{\delta} = T- \len (c_{\epsilon}) + O(\delta). \]
Moreover, $\int_{[\delta,\infty) \times [0,1]} (u_{\epsilon})^* d\lambda_{\tau} $ is equal to
\begin{align}\label{int-approx}
 \int_{ ([\delta,\delta_1]\cup [R_1,\infty)) \times [0,1]} (u_{\epsilon})^* d\lambda_{\tau} +   \int_{[\delta_1,R_1] \times [0,1]} (u_{\epsilon})^* d\lambda_1, 
\end{align}
where $\lambda_1 \coloneqq \frac{p}{|p|} dq$ is a $1$-form on $\{ (q,p) \in T^*\R^n \mid |p| > r_1 \}$.
The first term of (\ref{int-approx}) is non-negative by \cite[Lemma 8.8]{CELN} since $u_{\epsilon}$ is $J'$-holomorphic on $([\delta,\delta_1]\cup [R_1,\infty)) \times [0,1]$. Moreover, since $\Dbar_{J'}u_{\epsilon} = O(\epsilon^2)$ and $ \partial_s u_{\epsilon} = \partial_s u_a + O(\epsilon)$ on $[\delta_1,R_1]\times [0,1]$,
\begin{align*}
 (u_{\epsilon})^*(d\lambda_1) (\partial_s,\partial_t)&= d\lambda_1 ( \partial_s u_{\epsilon}, \partial_t u_{\epsilon})  \\
 &=d\lambda_1( \partial_s u_{\epsilon}, J' \partial_s u_{\epsilon}) + d\lambda_1( \partial_s u_a + O(\epsilon), O(\epsilon^2)) \\
&=\epsilon^2 | \nabla_s \eta |^2_1 + O(\epsilon^3),
 \end{align*}
where $|v |^2_1\coloneqq d\lambda_1(v, J'v)$, which is non-negative. Here, we use the facts that $| \partial_s u_a|^2_1 =0$ and $d\lambda_1 (\partial_s u_a,\cdot )=0$ since $u_a$ is tangent to $\mathrm{Span}(p\partial_p, p\partial_q)$ on $[\delta_1,R_1]\times [0,1]$. See \cite[Lemma 8.8]{CELN}. Moreover, $|\nabla_s \eta|^2_1$ is positive at $z_0\in[\delta_1,R_1]\times [0,1]$.
These estimates imply that if we take $A$ to be
\[0< A < \int_{[\delta_1,R_1]\times [0,1]} |\nabla_s \eta|_1^2\]
and $\epsilon'_2\in (0,\epsilon_2]$ to be sufficiently small, then the second term of (\ref{int-approx}) is bounded from below by $A\epsilon^2$ for any $\epsilon \in [-\epsilon'_2,\epsilon'_2]$.
 Note that $A$ and $\epsilon'_2$ are independent of $\delta$. Therefore,
\[ T- \len (c_{\epsilon}) + O(\delta) \geq A\epsilon^2  .\]
By taking the limit of $\delta \to 0$, it follows that $T- \len (c_{\epsilon}) \geq A\epsilon^2$ for any $\epsilon \in [-\epsilon'_2,\epsilon'_2]$.
\end{proof}

\bibliography{reference.bib}

\end{document}